\documentclass{article}
\usepackage{amsmath}
\usepackage{amssymb}
	\usepackage{amsthm}
\usepackage{slashed}
\usepackage{cite}
\usepackage{color} 
\usepackage{caption}
\usepackage{euscript}
\usepackage{comment}
\usepackage[arrow,curve,matrix,arc,2cell]{xy}
\usepackage{graphicx}    
\usepackage[utf8]{inputenc}
\usepackage[unicode]{hyperref}
\usepackage[a4paper, margin=1.5in]{geometry}
\usepackage{pifont}
\UseAllTwocells
\DeclareFontFamily{U}{rsfs}{} \DeclareFontShape{U}{rsfs}{n}{it}{<->
rsfs10}{} \DeclareSymbolFont{mscr}{U}{rsfs}{n}{it}
\DeclareSymbolFontAlphabet{\scr}{mscr}
\def\mathscr{\scr}
\begin{document}

\newcommand{\red}[1]{\textcolor{red}{#1}}
\def\e#1\e{\begin{equation}#1\end{equation}}
\def\ea#1\ea{\begin{align}#1\end{align}}
\def\eas#1\eas{\begin{align*}#1\end{align*}}
\def\eq#1{{\rm(\ref{#1})}}
\newenvironment{claim}[1]{\par\noindent\underline{Claim:}\space#1}{}
\newenvironment{claimproof}[1]{\par\noindent\underline{Proof:}\space#1}{\hfill $\blacksquare$}
\theoremstyle{plain}
\newtheorem{thm}{Theorem}[section]
\newtheorem{prop}[thm]{Proposition}
\newtheorem{lem}[thm]{Lemma}
\newtheorem{cor}[thm]{Corollary}
\newtheorem{quest}[thm]{Question}
\newtheorem{conj}[thm]{Conjecture}
\theoremstyle{definition}
\newtheorem{dfn}[thm]{Definition}
\newtheorem{ex}[thm]{Example}
\newtheorem{rem}[thm]{Remark}
\numberwithin{equation}{section}

\def\loc{{\mathop{\rm loc}\nolimits}}
\def\ope{{\mathop{\rm op}\nolimits}}

\def\cyl{{\mathop{\rm cyl}\nolimits}}
\def\con{{\mathop{\rm con}\nolimits}}

\def\cay{\mathop{\rm cay}\nolimits}
\def\Cay{\mathop{\rm Cay}\nolimits}
\def\ACay{\mathop{\rm ACay}\nolimits}

\def\ACyl{{\mathop{\rm ACyl}\nolimits}}
\def\AC{{\mathop{\rm AC}\nolimits}}
\def\CS{{\mathop{\rm CS}\nolimits}}
\def\supp{\mathop{\rm supp}\nolimits}
\def\dist{\mathop{\rm dist}\nolimits}
\def\sgn{\mathop{\rm sgn}\nolimits}
\def\dim{\mathop{\rm dim}\nolimits}
\def\Ker{\mathop{\rm Ker}}
\def\Coker{\mathop{\rm Coker}}
\def\Ho{{\mathop{\rm H}}}
\def\sign{\mathop{\rm sign}\nolimits}
\def\id{\mathop{\rm id}\nolimits}
\def\dvol{\mathop{\rm dvol}\nolimits}
\def\spn{\mathop{\rm span}\nolimits}
\def\SO{\mathop{\rm SO}\nolimits}
\def\Orth{\mathop{\rm O}\nolimits}
\def\Fr{\mathop{\rm Fr}\nolimits}
\def\Gr{\mathop{\rm Gr}\nolimits}
\def\cay{\mathop{\rm cay}\nolimits}
\def\inj{\mathop{\rm inj}\nolimits}
\def\SF{\mathop{\rm SF}\nolimits}
\def\Or{\mathop{\rm Or}\nolimits}
\def\ad{\mathop{\rm ad}\nolimits}
\def\Hom{\mathop{\rm Hom}\nolimits}
\def\Map{\mathop{\rm Map}\nolimits}
\def\Crit{\mathop{\rm Crit}\nolimits}
\def\ev{\mathop{\rm ev}\nolimits}
\def\Univ{\mathop{\rm Univ}\nolimits}
\def\Fix{\mathop{\rm Fix}\nolimits}
\def\Hol{\mathop{\rm Hol}\nolimits}
\def\Iso{\mathop{\rm Iso}\nolimits}
\def\Hess{\mathop{\rm Hess}\nolimits}
\def\Stab{\mathop{\rm Stab}\nolimits}
\def\Pd{\mathop{\rm Pd}\nolimits}
\def\Aut{\mathop{\rm Aut}\nolimits}
\def\Diff{\mathop{\rm Diff}\nolimits}
\def\boFlag{\mathop{\rm Flag}\nolimits}
\def\boFlagSt{\mathop{\rm FlagSt}\nolimits}
\def\dOrb{{\mathop{\bf dOrb}}}
\def\dMan{{\mathop{\bf dMan}}}
\def\mKur{{\mathop{\bf mKur}}}
\def\Kur{{\mathop{\bf Kur}}}
\def\Re{\mathop{\rm Re}}
\def\Im{\mathop{\rm Im}}
\def\re{\mathop{\rm re}}
\def\im{\mathop{\rm im}}
\def\SU{\mathop{\rm SU}}
\def\Sp{\mathop{\rm Sp}}
\def\Spin{\mathop{\rm Spin}}
\def\GL{\mathop{\rm GL}}
\def\ind{\mathop{\rm ind}}
\def\area{\mathop{\rm area}}
\def\U{{\rm U}}
\def\vol{\mathop{\rm vol}\nolimits}
\def\virt{{\rm virt}}
\def\emb{{\rm emb}}
\def\bs{\boldsymbol}
\def\ge{\geqslant}
\def\le{\leqslant\nobreak}
\def\O{{\mathbin{\mathcal O}}}
\def\cA{{\mathbin{\mathcal A}}}
\def\cB{{\mathbin{\mathcal B}}}
\def\cC{{\mathbin{\mathcal C}}}
\def\cD{{\mathbin{\scr D}}}
\def\cDHS{{\mathbin{\scr D}_{\Q HS}}}
\def\cE{{\mathbin{\mathcal E}}}
\def\boE{{\mathbin{\mathbf E}}}
\def\cF{{\mathbin{\mathcal F}}}
\def\boF{{\mathbin{\mathbf F}}}
\def\cG{{\mathbin{\mathcal G}}}
\def\cH{{\mathbin{\mathcal H}}}
\def\cI{{\mathbin{\mathcal I}}}
\def\cJ{{\mathbin{\mathcal J}}}
\def\cK{{\mathbin{\mathcal K}}}
\def\cL{{\mathbin{\mathcal L}}}
\def\cM{{\mathbin{\mathcal M}}}
\def\bcM{{\mathbin{\bs{\mathcal M}}}}
\def\cN{{\mathbin{\mathcal N}}}
\def\cO{{\mathbin{\mathcal O}}}
\def\cP{{\mathbin{\mathcal P}}}
\def\boR{{\mathbin{\mathbf R}}}
\def\cS{{\mathbin{\mathcal S}}}
\def\cT{{\mathbin{\mathcal T}}}
\def\cU{{\mathbin{\mathcal U}}}
\def\cQ{{\mathbin{\mathcal Q}}}
\def\boQ{{\mathbin{\mathbf Q}}}
\def\cW{{\mathbin{\mathcal W}}}
\def\C{{\mathbin{\mathbb C}}}
\def\bQ{{\mathbin{\mathbb Q}}}
\def\bV{{\mathbin{\mathbb V}}}
\def\bE{{\mathbin{\mathbb E}}}
\def\bD{{\mathbin{\mathbb D}}}
\def\boF{{\mathbin{\mathbf F}}}
\def\bF{{\mathbin{\mathbb F}}}
\def\H{{\mathbin{\mathbb H}}}
\def\N{{\mathbin{\mathbb N}}}
\def\Q{{\mathbin{\mathbb Q}}}
\def\R{{\mathbin{\mathbb R}}}
\def\bS{{\mathbin{\mathbb S}}}
\def\Z{{\mathbin{\mathbb Z}}}
\def\sF{{\mathbin{\mathscr F}}}
\def\al{\alpha}
\def\be{\beta}
\def\ga{\gamma}
\def\de{\delta}
\def\io{\iota}
\def\ep{\epsilon}
\def\eps{\epsilon}
\def\la{\lambda}
\def\ka{\kappa}
\def\th{\theta}
\def\ze{\zeta}
\def\up{\upsilon}
\def\vp{\varphi}
\def\si{\sigma}
\def\om{\omega}
\def\De{\de}
\def\La{\Lambda}
\def\Si{\Sigma}
\def\Th{\Theta}
\def\Om{\Omega}
\def\Ga{\Gamma}
\def\Up{\Upsilon}
\def\pd{\partial}
\def\ts{\textstyle}
\def\st{\scriptstyle}
\def\sst{\scriptscriptstyle}
\def\w{\wedge}
\def\sm{\setminus}
\def\bu{\bullet}
\def\op{\oplus}
\def\ot{\otimes}
\def\ov{\overline}
\def\ul{\underline}
\def\bigop{\bigoplus}
\def\bigot{\bigotimes}
\def\iy{\infty}
\def\es{\emptyset}
\def\ra{\rightarrow}
\def\Ra{\Rightarrow}
\def\Longra{\Longrightarrow}
\def\ab{\allowbreak}
\def\longra{\longrightarrow}
\def\hookra{\hookrightarrow}
\def\dashra{\dashrightarrow}
\def\t{\times}
\def\ci{\circ}
\def\ti{\tilde}
\def\d{{\rm d}}
\def\dt{{\rm dt}}
\def\D{{\rm D}}
\def\Lie{{\mathcal{L}}}
\def\ha{{\ts\frac{1}{2}}}
\def\md#1{\vert #1 \vert}
\def\bmd#1{\big\vert #1 \big\vert}
\def\ms#1{\vert #1 \vert^2}
\def\nm#1{\Vert #1 \Vert}
\title{Conically singular Cayley submanifolds I: Deformations }
\author{Gilles Englebert}
\date{\today}
\maketitle

\begin{abstract} 
This is the first in a series of three papers working towards constructing fibrations of compact $\Spin(7)$ manifolds by Cayley submanifolds. In this paper we describe the deformation theory of conically singular and asymptotically conical Cayley submanifolds.
\end{abstract}

\setcounter{tocdepth}{2}
\tableofcontents
\section{Introduction}
Riemannian $8$-manifolds with holonomy group $\Spin(7)$ first arose as objects of study after Berger proved the classification theorem for holonomy groups of simply-connected, irreducible and non-symmetric Riemannian manifolds  \cite[Thm. 3]{bergerGroupesHolonomieHomogenes1955}. In Berger's theorem, these manifolds appeared as an exceptional case associated to the octonions. However, he did not prove their existence. This question remained open until Bryant provided local examples in \cite{bryantMetricsExceptionalHolonomy1987} and complete but noncompact examples were later supplied by Bryant and Salamon in \cite{bryantConstructionCompleteMetrics1989}. The last major contribution to the existence problem came when Joyce constructed closed $\Spin(7)$-manifolds in \cite{joyceCompact8manifoldsHolonomy1996} by using gluing techniques. One reason for the interest in finding examples of $\Spin(7)$-manifolds is that they are necessarily Ricci-flat, which was first noted by Bonan in \cite{bonanVarietesRiemanniennesGroupe1966}. Since closed Ricci-flat manifolds are rare, the construction of compact $\Spin(7)$-manifolds is of great interest. Another reason is that a $\Spin(7)$-manifold is an example of a calibrated geometry as introduced by Harvey and Lawson in \cite{HarvLaws}. That is, it comes equipped with a closed $p$-form which satisfies a calibration inequality. The prototypical example is a Kähler manifold $M$ together with its Kähler $2$-form $\omega$. The so called $\omega$-calibrated submanifolds in $M$ are the complex curves. Similarly, every $\Spin(7)$-manifold admits a distinguished closed $4$-form called the Cayley form $\Phi$, and the $\Phi$-calibrated submanifolds are four-dimensional submanifolds called Cayley submanifolds, which in particular are minimal submanifolds. Finally, $\Spin(7)$-manifolds are interesting from the point of view of String theory, as they admit a constant spinor, meaning that they admit a supersymmetry. Similarly, Cayley cycles admit a supersymmetry as well, due to their calibration property.

Thus, the manifolds we will study in this paper are analogues of complex submanifolds in $\Spin(7)$-manifolds. We will however not restrict ourselves to the torsion-free setting (i.e. we may have $\d \Phi \ne 0$), but work with the more general class of almost $\Spin(7)$-manifolds and their Cayley submanifolds, which we will define precisely in the next section. The deformation theory of compact Cayleys was studied by McLean in the foundational paper \cite{mcleanDeformationsCalibratedSubmanifolds1998}. Recently, some progress has been made in the deformation theory of noncompact Cayley submanifolds. Ohst took on the case of asymptotically cylindrical Cayleys in \cite{ohstDeformationsCayleySubmanifolds2016}, that is Cayley submanifolds which have ends modeled on Riemannian cylinders at infinity. Moore described the case of conically singular Cayleys in a torsion-free $\Spin(7)$-manifold in \cite{mooreDeformationsConicallySingular2019}. Here, conically singular means that the Cayleys are compact topological manifolds, smooth away from finitely many singular points. Locally, around every singular point, they are modeled on Riemannian cones. In this paper we extend the work of Moore to be applicable to families of not necessarily torsion-free $\Spin(7)$-manifolds. In addition to this we describe the deformation theory of asympotically conical manifolds in $\R^8$, where the $\Spin(7)$-structure is allowed to be asympotically conical as well. Before we look at noncompact Cayleys we revisit the compact theory and extend it to submanifolds that are not Cayley themselves. This will allow for a more concise treatment of the family moduli spaces in the noncompact case.

The motivation for studying the deformation theory of conical Cayley submanifolds comes from the desire to construct fibrations of compact $\Spin(7)$-manifolds by Cayleys, in analogy to the programme by Kovalev in the $G_2$-case \cite{KovalevFibration}. The key idea is to construct a fibration on a $\Spin(7)$-manifold with small torsion from simpler pieces, and then perturb the $\Spin(7)$-structure to a torsion free one. The fibration will be perturbed with it and hopefully remain fibering after the perturbation. A topological argument shows that more generally a fibration of a compact manifold of holonomy $\Spin(7)$ by nonsingular Cayleys must necessarily include singular fibres. However there is hope that in some cases the singular fibres may admit only conical singularities. In such a case one can then understand the fibration near the singularities via a desingularisation procedure that we discuss in the second paper in this series. In the final paper we  investigate the question of stability of fibrations and construct examples.

\subsubsection*{Notation}
We will denote by $C$ an unspecified constant, which may refer to different constants within the same derivation. To indicate the dependence of this constant on other variables $x, y, \dots$, we will write it as $C(x, y, \dots)$. Similarly, if an inequality holds up to an unspecified constant, we will write $A \lesssim B$ instead of $A \le C B$.
We denote by $T^p_qM = (TM)^{\ot^p} \ot (T^*M)^{\ot^q} $ the bundle of $(p,q)$-tensors.

\subsubsection*{Acknowledgments}
This research has been supported by the Simons Collaboration on Special Holonomy in Geometry, Analysis, and Physics. I want to thank my DPhil supervisor Dominic Joyce for his excellent guidance during this project.

\section{Preliminaries}

We first recall basic results from the theory of $\Spin(7)$-manifolds and their Cayley submanifolds, as well as the Fredholm properties of elliptic operators on manifolds with ends.

\subsection{\texorpdfstring{Geometry of $\Spin(7)$-manifolds}{Geometry of Spin(7)-manifolds}}

The group $\Spin(7)$ is the double cover of $\SO(7)$, and thus a $21$-dimensional connected, simply-connected and compact Lie group. Its real spinor representation $\de_7 : \Spin(7) \ra \GL(8, \R)$ gives an embedding into $\SO(8)$, after choosing an invariant metric. Alternatively, this subgroup of $\SO(8)$ can be seen as the stabiliser of the \textbf{standard Cayley form} in $\R^8$. If $\R^8$ has coordinates $(x_1, \dots, x_8)$ then this form is given by: 
\ea
\label{2_1_cayley_form}
\Phi_0 = \d x_{1234} - \d x_{1256} - \d x_{1278} - \d x_{1357} + \d x_{1368} - \d x_{1458} - \d x_{1467} \nonumber \\
- \d x_{2358} - \d x_{2367} + \d x_{2457} - \d x_{2468} - \d x_{3456} - \d x_{3478} + \d x_{5678}, 
\ea 
where $\d x_{ijkl} = \d x_i \wedge \d x_j \wedge \d x_k \wedge \d x_l$.                                                                                                                                                                                                                                                                                                                                                                                                                                                                                                                                                                                                                                                                                                                                                                                                                                                                                                                                                                                                                                                                                                                                                                                                                                                                                                                                                                                                                                                                                                                                                                                                                                                                                                                                                                                                                                                                                                                                                                                                                                                                                                                                                                                                                                                                                                                           

More generally, we say that a 4-form  $\Phi$ on an $8$-dimensional vector space $V$ is a \textbf{Cayley form} if $V$ admits an isomorphism with $\R^8$ taking $\Phi$ to $\Phi_0$. We call the pair $(V, \Phi)$ a \textbf{$\Spin(7)$-vector space}. Any such form then determines a $\Spin(7)$-subgroup  $\Spin_\Phi(7) \subset \GL(V)$. Let $(V, \Phi)$ be a $\Spin(7)$-vector space. Then $\Phi$ induces a Riemannian metric $g_\Phi$ on $V$ obtained as the pullback of the standard metric $g_0 = \sum_{i = 1} ^8 \d x_i^2 $ via the isomorphism $V \simeq \R^8$. Note that the isomorphism is not unique, but since $\Spin(7) \subset \SO(8)$ the pullback metric will be independent of the choice of identification with $\R^8$. Pulling back the standard orientation on $\R^8$ orientation induces a well-defined orientation on $V$ in the same manner. Thus a Cayley form induces a metric and an orientation. In fact, the unoriented vector space $V$ admits two classes of Cayley forms,  determined by the orientation they induce. This is reflected in the fact that $\SO(8)$ admits exactly two conjugacy classes of $\Spin(7)$-subgroups, which are conjugated inside $\Orth(8)$ \cite[Thm. 1.3]{varadarajanSpinSubgroupsSpin2001}. Consequently, when we consider a vector space which already admits an orientation, we will only consider Cayley forms which induce the given orientation. If $V$ does not have an orientation, we allow the Cayley form to induce the orientation. In particular the Cayley form $\Phi$ then induces a Hodge star operator $\star: \Lambda^k V^* \ra \Lambda^{8-k}V^*$ and musical isomorphisms $\flat : V \ra V^*$ and $\sharp: V^* \ra V$. The Cayley form is self-dual with respect to the Hodge star it induces. Next, the action of $\Spin_\Phi(7)$ on $V$ induces representations on the tensor and exterior bundles, which decompose into irreducible representations of $\Spin_\Phi(7)$. We are mostly interested in the action on $2$-forms, which decomposes as follows:
\begin{prop}[{\cite[p. 546]{bryantMetricsExceptionalHolonomy1987}}]
\label{2.two_form_splitting}
There is an orthogonal splitting: 
\e 
\Lambda^2 V^* = \Lambda^2_7 V^* \op \Lambda^2_{21} V^*, 
\e
where $\Lambda^2_i$ is an $i$-dimensional irreducible representation. Explicitly they are given by: 
\ea
\Lambda^2_7 V^* &= \{ \omega \in \Lambda^2 V: \star(\Phi \wedge \omega) = 3\omega\}\nonumber \\ 
 &=  \{ u^\flat \wedge v^\flat - \iota(u)\iota(v)\Phi: u, v \in V \}, \label{2_2_7_def}\\
\Lambda^2_{21} V^* &= \{ \omega \in \Lambda^2 V: \star(\Phi \wedge \omega) = - \omega\}. \label{2_2_21_def}
\ea
\end{prop}

Using the Cayley form we now define various product structures on a $\Spin(7)$-vector space $(V, \Phi)$. First we define the \textbf{cross product} as the bilinear map $V \times V \ra \Lambda^2_7 V^*$:
\ea
\label{2_1_cross_prod}
    u \times v =  \pi_7(u^\flat \wedge v^\flat) := \frac{1}{4}\left(u^\flat \wedge v^\flat - \iota(u)\iota(v)\Phi\right),
\ea

Here $\pi_7(\omega) = \frac{1}{4}(\omega -\star (\omega \wedge \Phi))$ for $\omega \in \Lambda^2 V^*$ is the orthogonal projection onto the $\Lambda^2_7$-summand. The \textbf{triple product} is a trilinear map $V \times V \times V \ra V$ defined by: 
\ea
\label{2_1_trip_prod}
    u \times v \times  w =  \left(\iota(u)\iota(v)\iota(w)\Phi\right)^\sharp,
\ea

Finally, the \textbf{quadruple product} is a $\Lambda^2_7 V^*$-valued four-form:
\ea
\label{2_1_quad_prod}
    \tau(u,v, w,x) =  u \times (v& \times w \times x) - g_\phi(u,v)(w \times x) \nonumber \\
    -& g_\phi(u,w)(u \times x) +g_\phi(u,x)(v \times w). 
\ea
On $(\R^8, \Phi_0)$, this form has the following coordinate expression: 
\ea
\label{2_1_tau_coord}
\tau =& \frac{1}{4}\sum_{1 \le i < j \le 8} (e^j\wedge (\iota(e_i)\Phi)-e^i\wedge (\iota(e_j)\Phi))\ot (e^i \times e^j) \nonumber \\ 
=& (\d x_{1358}+\d x_{1367}-\d x_{1457}+\d x_{1468} \nonumber \\
&-\d x_{2357}+\d x_{2368}-\d x_{2458}-\d x_{2467}) \ot (e_1 \times e_2) \nonumber \\
&+(-\d x_{1258}-\d x_{1267}+\d x_{1456}+\d x_{1478}\nonumber \\
&+\d x_{2356}+\d x_{2378}-\d x_{3458}-\d x_{3467}) \ot (e_1 \times e_3) \nonumber\\
&+(\d x_{1257}-\d x_{1268}-\d x_{1356}-\d x_{1378}\nonumber \\
&+\d x_{2456}+\d x_{2478}+\d x_{3457}-\d x_{3468}) \ot (e_1 \times e_4) \nonumber \\
&+(\d x_{1238}-\d x_{1247}+\d x_{1346}-\d x_{1678}\nonumber \\
&-\d x_{2345}+\d x_{2578}-\d x_{3568}+\d x_{4567}) \ot (e_1 \times e_5) \nonumber \\ 
&+(\d x_{1237}+\d x_{1248}-\d x_{1345}+\d x_{1578}\nonumber \\
&-\d x_{2346}+\d x_{2678}-\d x_{3568}-\d x_{4568}) \ot (e_1 \times e_6) \nonumber \\ 
&+(-\d x_{1236}+\d x_{1245}+\d x_{1348}-\d x_{1568}\nonumber \\
&-\d x_{2347}+\d x_{2567}+\d x_{3678}-\d x_{1568}) \ot (e_1 \times e_7) \nonumber \\ 
&+(-\d x_{1235}-\d x_{1246}-\d x_{1347}+\d x_{1568}\nonumber \\
&-\d x_{2348}+\d x_{2568}+\d x_{3578}+\d x_{4678}) \ot (e_1 \times e_8). 
\ea

We now introduce $\Spin(7)$-manifolds by applying these linear algebraic constructions to the tangent bundle of smooth $8$-manifolds. To be precise, we take a \textbf{$\Spin(7)$-manifold} to be a smooth $8$-dimensional manifold $M$ together with a choice of $\Spin(7)$-structure, i.e. a choice of $\Spin(7)$-subbundle $\Fr_{\Spin(7)}$ of the frame bundle $\Fr(M)$. This data is equivalent to the choice of a smooth differential 4-form $\Phi$ on $M$ which is a Cayley form at every point. In other words, $\Phi$ is a smooth section of a bundle $\mathscr{A}(M)$ whose fiber over the point $x$ is the set of all Cayley forms of $T_xM$. With a choice of $\Spin(7)$-structure $T_xM$ is a $\Spin(7)$-vector space at every point $x\in M$, which gives $M$ the structure of an oriented Riemannian manifold. If we already assume an orientation on $M$, we require the $\Spin(7)$-structure to be compatible pointwise. The pair $(M, \Phi)$ will also be called a $\Spin(7)$-structure. 
\begin{ex}
\label{2_1_ex_other_spin}
There are examples of $\Spin(7)$ manifolds which come from dimensional reductions.
\begin{itemize}
\item $G_2$ \textbf{geometry}: We can write the Cayley form in \eqref{2_1_cayley_form} as $\Phi_0 = \d x_1 \wedge \phi_0 + \star_7 \phi_0$ for a three form $\phi_0\in \La^3\R^7$, and where $\star_7$ is the Hodge star on 	$\{0\}\times \R^7$. Define the group $G_2$ to be the stabiliser of the form	$\phi_0$ under the action of $\SO(7)$, which is a $14$-dimensional subgroup of 	$\Spin(7)$. A  $G_2	$-manifold is a smooth seven-manifold $M$ together with a three-form $\phi$ that is point-wise isomorphic to $\phi_0$. If we are given a $G_2$-manifold, then it induces a $\Spin(7)$-structure on $\R \times M$ with Cayley form $\Phi = \d t \wedge \phi + \star_M \phi.$ 
\item \textbf{Calabi-Yau geometry}: Let $(M^{n},g,\omega, J)$ be a Kähler manifold of complex dimension $n$. If in addition the bundle of $(n, 0)$-forms admits a nowhere vanishing holomorphic section $\Omega$ we call it an \textbf{ almost Calabi-Yau} manifold. Such a manifold is modelled at each point on  $(\C^n, g_0,\omega_0, J_0, \Omega_0)$, where $g_0$ and $J_0$ are the standard Riemannian metric and complex structure respectively and:
\eas
\omega_0 &= \sum_{i=1}^n \d z_i\wedge \d \bar{z}_i, \\
\Omega_0 &= \d z_1\wedge \dots \wedge \d z_n.
\eas
In the complex four-dimensional case, i.e. on  $\C^4$. we have $\Phi_0 = \Re \Omega_0 + \ha \omega_0 \wedge \omega_0$. Thus in particular any almost Calabi-Yau fourfold (CY4) is also a $\Spin(7)$ manifold. 
\end{itemize}
\end{ex}

Let $(M, \Phi)$ be a $\Spin(7)$-manifold with $\Spin(7)$-bundle $\Fr_{\Spin(7)}$. Then the tensor and exterior bundles of $M$ are associated to $\Fr_{\Spin(7)}$ via representations induced from the embedding $\Spin(7) \subset \SO(8)$. Thus the fibres of these bundles can be seen as representations of $\Spin(7)$, and as such decompose into bundles of irreducible representations. For two-forms, Proposition \ref{2.two_form_splitting} implies that there is an orthogonal splitting:  
\e 
\Lambda^2 T^*M = \Lambda^2_{21}\op \Lambda^2_7,
\e
where the fibres of $\Lambda^2_{21}$ and $\Lambda^2_7$ are given by \eq{2_2_21_def} and \eq{2_2_7_def} respectively. On a $\Spin(7)$-manifold $(M, \Phi)$ we can define the cross, triple and quadruple product of tangent vectors using the differential form $\Phi$, and these extend to bundle homomorphisms.

\subsection{Geometry of Cayley submanifolds}

Let $(V, \Phi)$ be a $\Spin(7)$-vector space. A fundamental property of the Cayley form $\Phi$ is that when restricted to any four-plane $\xi = \spn\{e_1, e_2, e_3, e_4\} $ with $\{e_1, e_2, e_3, e_4\}$ a positively oriented, $g_\Phi$-orthonormal basis, the \textbf{Cayley inequality} holds \cite[Th. 1.24, Ch. IV]{HarvLaws}): 
\e
\label{2_2_cay}
\Phi(e_1, e_2, e_3, e_4) \le 1.
\e
This means that $\Phi|_\xi = \alpha \cdot \dvol_\xi$ with $\alpha \le 1$, where $\dvol_\xi$ is the four-dimensional volume form induced by $g_\Phi$ on $\xi$. The oriented four-planes for which $\alpha$ attains its maximal value $1$ are the \textbf{Cayley planes}. They are said to be \textbf{calibrated} by $\Phi$. Note that if $\xi$ is Cayley, its orthogonal complement will be Cayley as well. If $u, v, w \in V$ are three independent vectors, then there is a unique Cayley plane which contains them, namely $\xi = \spn \{u, v, w, u\times v \times w \}$. Moreover, a four-plane is Cayley exactly when the quadruple product $\tau$ vanishes on it. 

Given a Cayley plane $\xi$ in a $\Spin(7)$-vector space $(V, \Phi)$, the cross product on $V$ decomposes with regards to the splitting $V = \xi \op \xi^\perp$, which we will now explain. Define:

\e
\label{2_2_E_bundle}
E_\xi = \{ \omega \in \Lambda^2_7V^*: \omega|_\xi = 0\},
\e
which is a rank four subspace of $\Lambda^2_7V^*$ (with an orthonormal basis given by $\pi_7(\d x_1 \wedge \d x_i)$ for $i \in \{5, 6, 7, 8\}$). Also note that any $\omega \in \Lambda^2_-\xi$ can be extended by $0$ to a two-form on $V$, and their projections under $\pi_7$ form the rank three subspace of $\Lambda^2_7V^*$ that we will also denote by $\Lambda^2_-\xi$. It has an orthonormal basis given by $\pi_7(\d x_1 \wedge \d x_i)$ for $i \in \{2, 3, 4\}$. Denote the orthogonal projection map to $E_\xi$ by $\pi_E: \La^2_7 \longra E_\xi$. From the above we see that there is an orthogonal splitting: $\Lambda^2_7 V^* = E_\xi \op \Lambda^2_-\xi$. The cross product then restricts as follows: 
\ea
\xi \times \xi &\longra \Lambda^2_- \xi, \nonumber \\
\xi^\perp \times \xi^\perp &\longra \Lambda^2_- \xi, \nonumber \\
\xi \times \xi^\perp &\longra E_\xi.
\ea
Let now $(M, \Phi)$ be an almost $\Spin(7)$-manifold. We call a four-dimensional submanifold $N \subset M$ all of whose tangent planes are Cayley planes a \textbf{Cayley submanifold}. In this situation, the cross product splits into: 
\ea
TN \times TN &\longra \Lambda^2_- TN,  \nonumber\\
\nu(N) \times \nu(N) &\longra \Lambda^2_- TN,  \nonumber\\
TN \times \nu(N) &\longra E.
\ea
Note that we can carry out the same construction whenever we are given a rank $4$ subbundle of $TM|_N$ whose fibres are Cayley planes, irrespective of whether $N$ is Cayley. 
\begin{ex}
\label{2_2_ex_cplx_sl}
The $\Spin(7)$-manifolds coming from reductions of the structure group to $G_2$ and $\SU(4)$ (see Example \ref{2_1_ex_other_spin}) admit their own classes of calibrated submanifolds, which give examples of Cayley submanifolds.
\begin{itemize}

\item $G_2$ \textbf{geometry}: The three-form $\phi_0$ and the four-form $\star_7\phi$ satisfy a calibration inequality which is analogous to the Cayley inequality \eqref{2_2_21_def}. The calibrated hyperplanes are called \textbf{associative} $3$-planes and \textbf{coassociative} $4$-planes respectively. If we have an associative submanifold $A^3$ in $(N^7,\phi)$, then $\R \times A$ will be Cayley in the $\Spin(7)$-manifold $\R\times N$ with the Cayley form $\Phi = \d t \wedge \phi + \star_7 \phi$. Similarly, if $C^4$ is coassociative in $(N^7,\phi)$, then $\{t\} \times C$ will be Cayley in $\R\times N$.

\item \textbf{Calabi-Yau geometry}: The almost Calabi-Yau fourfold $(M^4,g, \om, J, \Om)$ admits two kinds of calibrated four-dimensional manifolds. First we have the \textbf{complex surfaces}, which are calibrated by $\ha \om \wedge \om$. Second we have the \textbf{special Lagrangian manifolds},  calibrated by $\Re \Om$. As the Cayley form on $M$ is $\Phi = \ha \om \wedge \om + \Re \Om$, which is the sum of both the previous calibrations, both complex surfaces and special Lagrangian submanifolds will be Cayley in the induced $\Spin(7)$ manifold. 
\end{itemize}

\end{ex}
\subsection*{The Dirac bundle associated to a Cayley}

We will see later that the linearised deformation operator associated to a spin Cayley is a twisted Dirac operator. On a non-spin Cayley the situation is more complicated, as neither the spinor bundles nor the bundle by which they are twisted are well-defined on their own, however one can still make sense of their product, in the form of a \textbf{Dirac bundle} (for a precise definition, see \cite[Ch. II.5]{lawsonSpinGeometryPMS382016}). The Cayley deformation operator will then linearise to the Dirac operator associated to this Dirac bundle, which we will define for any Cayley submanifold $N$ (be it spin or not) in a $\Spin(7)$-manifold $(M, \Phi)$. We have previously introduced the $\Spin(7)$-frame bundle associated to $\Phi$, which can be described as: 
\e
 (\Fr_{\Spin(7)})_x = \{ e: T_xM \xrightarrow{\simeq} \R^8 : e^*(\Phi_0) = \Phi_x\}.
\e
Using the splitting $TM|_N = TN \op \nu(N)$, where in the Cayley case both summands are bundles of Cayley planes, we can define the \textbf{adapted $\Spin(7)$-frame bundle}  $\Fr_{\Spin(7), N} \subset \Fr_{\Spin(7)}|_N$  as: 
\ea
	(\Fr_{\Spin(7), N})_x = \{e: T_xM \xrightarrow{\simeq} \R^8 :  e^*(\Phi_0)& = \Phi_x, e(T_xN) = \R^4 \times 0,  \nonumber \\
	e(\nu(N)) & = 0 \times \R^4 \}.
\ea

The structure group of this bundle is isomorphic to the stabiliser of a given Cayley plane (since it  automatically preserves the orthogonal complement).  It is given by 
\[
H = (\Sp(1) \times \Sp(1) \times \Sp(1))/ (\pm (1, 1, 1)),
\]
as shown in \cite[Thm. IV.1.8]{HarvLaws}. Here $H \subset \Spin(7) \subset \SO(8)$ via the following action on $\R^8 \simeq \H \op \H$. For $[p, q, r] \in H$ and $(u, v) \in \H \op \H$ we have: 
\e
 \label{2_2_H_action}
 [p_1, p_2, q] \cdot (u, v) = (p_1 u \bar{q}, p_2 v \bar{q}).
\e

Using the embedding $H \subset \SO(8)$, a number of bundles over $N$ can be represented as associated bundles to $\Fr_{\Spin(7), N}$. Here $u, v \in \H$ and $w \in \im \H$.
\begin{itemize}
\item $TN$ is associated via $\rho_{TN}( [p_1,p_2,q])\cdot u = (p_1 u \bar{q})$, since the projection $[p_1,p_2, q] \mapsto [p_1, q]$ maps $H$ surjectively onto $\SO(\R^4 \times 0)$.
\item $\nu(N)$ is associated via $\rho_{\nu(N)}([p_1,p_2,q])\cdot u = (p_2 u \bar{q})$, as $H$ also surjects onto $\SO(0 \times \R^4)$

\item If $N$ is spin, then the adapted $\Spin(7)$-frame bundle admits a double cover by a $G =  \Sp\left(1\right)^3$-bundle, which we will denote by $\widetilde{\Fr}_{\Spin(7), N} $. This can be seen as follows: as $N$ is spin, we can lift a co-cycle for the tangent bundle to $\Spin(4) \simeq \Sp(1)^2$. Similarly, since $M$ admits a spin structure induced by the $\Spin(7)$ structure (as $\Spin(7)$ is a simply connected subgroup of $\SO(8)$), the normal bundle $\nu(N)$ will also be canonically spin \cite[Prop. II.1.15]{lawsonSpinGeometryPMS382016}, thus a describing co-cycle can be lifted to $\Spin(4)$ as well. Using these two lifts one can then write down a lift to $G$ for a co-cycle of $ \Fr_{\Spin(7), N}$. The tangent and normal bundle will then be associated to this double cover via the lift of the representations $\rho_{TN}$ and $\rho_{\nu(N)}$ respectively. Furthermore, the spinor bundles of $N$ are associated bundles to this double cover as follows:
\e
\slashed{S}_\pm = \widetilde{\Fr}_{\Spin(7), N} \times _{\de_\pm} \H, \nonumber
\e 
where $\de_+ \op \de_- $ acts on $\H\op \H$ via $(p_1,p_2,q)(u,v) = (u\bar{p_1}, v \bar{q})$. Similarly the spinor bundles of $\nu(N)$ are associated via the representation $(p_1,p_2,q)(u,v) = (u\bar{p_2}, v \bar{q})$.

\item The irreducible representation $\La^2_7$ of $\Spin(7)$ restricted to $H$ can be described as follows. Let $\R^7 \simeq \im \H \op \H$ via the obvious isomorphism. Then we have the following (see \cite{mcleanDeformationsCalibratedSubmanifolds1998}): 
\e 
 \rho_7([p_1,p_2,q]) (w, u) = (\bar{q} w q, p_2 u \bar{p_1}). \nonumber
\e 

It turns out that in this splitting, the bundle associated via $[p_1,p_2,q]w =\bar{q} w q $ is exactly the bundle $\La^2_- N$ of anti-self-dual two-forms, and the bundle associated via $[p_1,p_2,q]w =p_2 u \bar{p_1}$ is $E$ .
\end{itemize}

From this discussion, we see that the suggestively named bundle $\slashed{S} := E \op \nu(N)$ arises from the representation:
\e
 \rho: [p_1,p_2, q] \cdot (u,v) = (p_2 u \bar{p_1}, p_2 v \bar{q}).
\e

The notation is explained as follows. If $N$ is spin, then consider the quaternionic line bundle $L$ associated to $\widetilde{\Fr}_{\Spin(7), N} $ via the representation $\rho_L: (p_1,p_2,q)u = p_2 u$. We then see from the representations, that as quaternionic bundles, $E \simeq \slashed{S}_+ \ot_\H L$ and similarly $\nu(N) \simeq \slashed{S}_- \ot_\H L$, which allows to represent the bundle $E \op \nu(N)$ as a twisted spinor bundle, if $N$ is spin.

 To complete the construction of the Dirac bundle, we need to define a Clifford multiplication and a compatible metric and connection. This is done in the following proposition: 

\begin{prop} [Dirac bundle]
\label{2_2_dirac_bundle}
There is a Clifford multiplication map $c: TN \times \slashed{S} \ra \slashed{S}$, a metric $h$ and a connection $\nabla$ on $\slashed{S}$ such that $(\slashed{S}, c, h, \nabla)$ is a Dirac bundle. In an adapted $\Spin(7)$-frame $\{e_i\}_{i=1, \dots, 8}$ the negative Dirac operator acts on $v \in C^\infty(\nu(N))$ as: 
\e
\label{2_2_dirac_op}
\slashed{D} v = \sum_{i = 1}^4 e_i \times \nabla^\perp_{e_i} v \in C^\infty (E),
\e
where $\nabla^\perp$ is induced from the Levi-Civita connection  on $M$.
\end{prop}
\begin{proof}
Consider the Clifford multiplication on the Clifford algebra $Cl(\H) \simeq \H \op \H$ (here $\H$ is equipped with the standard metric), whose action by vectors in $\H$ can be given as: 
\ea
\label{2_3_clifford_mult}
c: \H \times (\H \op \H)& \longra \H \op \H \\ \nonumber
(h, (v_1, v_2)) &\longmapsto (v_2 \bar{h}, -v_1h).
\ea
This action commutes with the representation determining $\slashed{S}$, and $TN$, in the sense that: 
\e
	 c \circ (\rho_{TN}, \rho_{E}\op \rho_{\nu(N)}) = (\rho_{E}\op \rho_{\nu(N)}) \circ c \nonumber.
\e
Thus we can extend $c$ to a map $c: TN \times \slashed{S} \ra \slashed{S}$ as required. For an adapted $\Spin(7)$-frame $\{e_i\}_{i=1, \dots, 8}$, we identify $(1,0)$, $(i, 0)$, $(j, 0)$ and $(k, 0)$ with the basis elements $e_1 \times e_i$ for ($5 \le i \le 8$) of $E$ and we identify $(0,1)$,  $(0, i)$, $(0, j)$ and $(0, k)$ with the basis elements $e_i (5 \le i \le 8)$ of $\nu(N)$. Using this identification we see that the Clifford multiplication $c : TN \times \nu(N) \ra E$ is exactly given by the cross-product. Since the $e_i$ are orthonormal with respect to the metric $g_\Phi$, as are $e_1 \times e_i$ for ($5 \le i \le 8$) with respect to the metric $g_E$ induced from $g_\Phi$ on the bundle of forms, we see that $c(v)$ is an isometry of $(\slashed{S}, h = g_\Phi \op  g_E)$, whenever $v$ is a unit vector. Finally, we choose as our connection $\nabla$ to be the given on $\nu(N)$ as the connection $\nabla^\perp$. On $E$ we choose the unique connection such that $c(e_1)v$ is a parallel section (along a curve), whenever $v$ is a parallel section along a curve in $\nu(N)$. From these definitions it follows readily that $(\slashed{S}, c, h, \nabla)$ is a Dirac bundle. The Dirac operator restricted to $\nu(N)$ is then of the required form.
\end{proof}
\begin{ex}
When $C$ is a Cayley cone in $\R^8$, with link $L = C \cap S^7$, then the Dirac operator $\slashed{D}$ can be rewritten as an evolution equation of vector fields $u\in \nu_{S^7}(L)$. For this, we introduce the nearly parallel $G_2$-structures on round spheres. Let $\partial_r$ be the outward radial unit vector field on $\R^8\setminus 0$. Then at the point $(r, p) \in \R_+ \times S^7 \simeq \R^8\setminus 0$ we have: 
\eas
\Phi_{r,p} = \d r \wedge (\phi^r)_p + (\star_{rS^7}\phi^r)_p.
\eas
Here $\phi^r$ is the associative form corresponding to the nearly parallel $G_2$-structure on the round sphere of radius $r$. Let $\{e_1, e_2, e_3\}$ be an orthonormal frame around a point $p \in L$ with dual coframe $\{e^1, e^2, e^3\}$. Then we can rewrite the Dirac operator \eqref{2_2_dirac_op} as follows for $v \in C^\infty(\nu(C))$: 
\eas
\slashed{D}v &= \partial_r \times \nabla^\perp_{\partial_r} v + \sum_{i = 1}^3 e_i \times \nabla^\perp_{e_i} v|_{rL} \\
&= \d r \wedge (\nabla^\perp_{\partial_r} v)^\flat - \iota (\nabla^\perp_{\partial_r} v)\phi^{r} + \sum_{i = 1}^3 e^i \wedge (\nabla^\perp_{e_i} v|_{rL})^\flat - \iota(e_i) \iota (\nabla^\perp_{e_i} v)\star_{rS^7}\phi^{r}\\
&= A_p(\nabla^\perp_{\partial_r} v) + B_r (v|_{rL}).
\eas
Here $A_p: \nu_{r, p}(C) \longrightarrow E_{r, p}$ is a linear map that is independent of the radius, and
\e
B_r: C^\infty(\nu_{rS^7}(rL)) \longra C^\infty(E|_{rL})
\e 
are a family of  first order partial differential operators on the links $rL$. We can furthermore identify $E|_{rL} \simeq\nu_{rS^7}(rL)$ via the map $\omega \longmapsto (\iota(\partial_r)\omega )^\sharp$, and identify sections $C^\infty( \nu_{S^7}(L)) \simeq C^\infty( \nu_{rS^7}(rL))$ via rescaling, at which point the operator has the following shape:
\ea
\label{2_2_dirac_op_cone}
\slashed{D}: C^\infty(\R_+, C^\infty( \nu_{S^7}(L))) &\longra C^\infty(\R_+, C^\infty( \nu_{S^7}(L))) \nonumber \\
 v &\longmapsto \frac{\d}{\d r} v + D_L v(r). 
\ea
Here:
\ea
D_L: C^\infty( \nu_{S^7}(L)) &\longra C^\infty( \nu_{S^7}(L))\nonumber \\
u &\longmapsto B_1 (u) = \sum_{i = 1}^3 e_i \times \nabla^\perp_{e_i} u,  
\ea
where $\times$ is the vector product associated with the associative manifold $L \subset (S^7, \phi^1)$. It is determined by the identity $g(u \times v, w) = \phi^1 (u, v, w)$.

\end{ex}
\begin{ex}
We noted in Example \ref{2_2_ex_cplx_sl} that complex surfaces $N$ in a CY4 manifold $M$ are examples of Cayley submanifolds. In this case the linearised Cayley deformation operator is a twisted Dirac operator on a Kähler surface, and thus necessarily of the form $\bar{\partial}+\bar{\partial}^*$ with twisted coefficients \cite{morganSeibergWittenEquationsApplications1996}. It has been computed in \cite{mooreDeformationTheoryCayley2017} and can be identified with:
\e
\bar{\partial}+\bar{\partial}^* : C^\infty (\nu^{1,0}(N)\op \La^{0,2}N \ot\nu^{1,0}(N)) \longra C^\infty(\La^{0,1}N \ot \nu^{1,0}(N)).
\e
For any complex surface, the kernel and cokernel of this operator are the complexifications of the kernel and cokernel respectively of $\slashed{D}$. Thus the real expected dimension of the Cayley moduli space is equal to the complex index of $\bar{\partial}+\bar{\partial}^*$. 
This can be compared to the linearised deformation operator of complex surfaces which is just $\bar{\partial}\op\bar{\partial}^*$. Hence being Cayley is a more general condition than being complex.
\end{ex}

We also noted in Example \ref{2_2_ex_cplx_sl} that special Lagrangians in CY4 manifolds are  examples of Cayley submanifolds. McLean (\cite{mcleanDeformationsCalibratedSubmanifolds1998}) showed that the infinitesimal deformations of a special Lagrangian $N \subset M$ are given by the kernel of the operator
\e
-\d\star\op \d : \Omega^1(N) \longra \Omega^n(N)\op \Omega^2(N),
\e 
which are the closed and co-closed forms. The Cayley deformation operator of a special Lagrangian is formed by a subset of these equations, reflecting the fact that the Cayley condition is a priori less restrictive than the special Lagrangian condition.

\begin{prop}
\label{2_2_cayley_sl}
Let $N$ be a special Lagrangian submanifold in the almost Calabi-Yau manifold $(M, g, \om, J, \Om)$.  Then the infinitesimal Cayley deformation operator can be identified with:
\e
-\d\star\op \d^- : \Omega^1(N) \longra \Omega^4(N)\op \Omega^{2,-}(N).
\e
Here $\Omega^{2,-}(N)$ is the bundle of self-dual two forms on $(M,g)$, and $\d^- = \pi^- \circ \d$, where $\pi^- (\eta) = \ha(\eta -\star_N \eta) $ is the projection onto the anti-self-dual forms.
\end{prop}
\begin{proof}
First, we show that there are canonical isomorphisms $m: TN^*  \simeq  \nu(N)$ and $n: E \simeq \La^4 \op \La^2_-$. We can take $m(\si) = J \si^\sharp $ to be the composition of the musical isomorphism $\sharp: TN^*\longra TN$ and $J$. Note that $J$ maps the tangent bundle of any Lagrangian to its normal bundle as $g(v, Jw) = \om(v, w) = 0$ for any pair of vectors $v,w \in T_pN$ by the Lagrangian condition. As for the morphism $n$, we can pull back forms on $T_pM$ via the map $\id \op J: TN \longra TN \op \nu(N)$, which when restricted to $E$ gives an injection into the anti-self-dual forms on $TN$. The kernel of this map is spanned by $v^\flat \wedge (Jv)^\flat$, and the projection onto these forms gives the $\La^4$ summand. More concretely, recall that $E_p$ is spanned by $e_1 \times Je_i$, where $\{e_i\}_{1 \le i \le 4}$ is an orthonormal basis of $T_pN$. The morphism $n$ then sends $e_i \times Je_i$ to the $\La^4$ summand, and identifies $e_i \times e_j$ for $i\ne j$ with the anti-self-dual form $\al_{ij} = \d x_{ij} -\d x_{kl}$, where the $\d x_i$ are dual to $e_i$ and $(i,j,k,l)$ is a positive permutation of $(1,2,3,4)$. Let now $f_i = Je_i \in \nu(N)$ complete the $e_i$ to a frame of $T_pM$, and suppose that $\d y_i$ are the corresponding dual $1$-forms. A computation shows that the vector product of $v = \sum_{i=1} a_i e_i \in T_pN$ and $w = \sum_{i=1} b_i f_i \in \nu(N)$ is given by:
\eas
v \times w &= \sum_{i=1}^4 a_i b_i \d x _i \wedge \d y_i \\
&+ \sum_{\si (i,j,k,l) =1} \frac{a_i b_j}{4}\underbrace{(\d x _i \wedge \d y_j-\d x _j \wedge \d y_i-\d x _k \wedge \d y_l+\d x _l \wedge \d y_k)}_{\be_{ij}}.
\eas
Here $\si(i,j,k,l)=1$ means that $(i,j,k,l)$ is a positive permutation of $(1,2,3,4)$. Note that $k,l$ are uniquely determined by $i$ and $j$. We now look at $n \circ \slashed{D} \circ m$ , where $\slashed{D}$ is the Dirac operator from equation \eqref{2_2_dirac_op}. For a one form $\eta = \sum_{i=1}a_ie_i \in \Om^1(N)$, where we extended the basis $\{e_1, e_2, e_3, e_4, f_1, f_2, f_3, f_4\}$ to a local parallel frame, we have:
\eas
\slashed{D}  (m(\eta )) &= \sum_{i,j=1}^4 e_i \times \nabla_{e_i} ^\perp(a_j f_j)\\
&= \sum_{i,j=1}^4 e_i \times \frac{\partial a_j}{\partial x_i} f_i\\
&= \sum_{i=1}^4\frac{\partial a_i}{\partial x_i} \d x_i \wedge \d y_i +\ha\sum_{i\ne j} \frac{\partial a_i}{\partial x_j} \be_{ij}.
\eas
As $n$ maps $\d x_i \wedge \d y_i$ to $\dvol \in \La^4$, and $\be_{ij}$ to $\al_{ij}$, we see that  
\eas
n\circ \slashed{D}  (m(\eta )) &=\sum_{i=1}^4\frac{\partial a_i}{\partial x_i}\dvol + \sum_{i< j}\left ( \frac{\partial a_i}{\partial x_j}-\frac{\partial a_j}{\partial x_i}\right)\al_{ij}\\
&= \d\star \eta +\ha\sum_{i< j}(\d \eta)_{ij}\al_{ij} = \d\star \eta +\pi^- \d \eta.
\eas

\end{proof}

\subsection{Analysis on manifolds with ends}

Below we investigate the deformation theory of Cayley submanifolds which admit conical ends. In this section we lay the groundwork for the analysis on Riemannian manifolds of this kind. The Fredholm properties of elliptic operators on compact manifolds can be extended to these special classes of noncompact manifolds by adapting the results by Lockhart and McOwen \cite{lockhartEllipticDifferentialOperators1985}.

\subsubsection{Manifolds with ends}

\begin{dfn}
A Riemannian $n$-manifold $(M, g)$ is \textbf{asymptotically conical} with rate $\eta < 1 $ ($\AC_\eta$) if there is a compact set $K \subset M$, a compact $(n-1)$-dimensional Riemannian manifold $(L,h)$ and a diffeomorphism $\Psi:  (r_0, \infty) \times L \ra  M\setminus K$ (for some $r_0 > 0$), such that for $i \in \N$:
\e
\md{\nabla^i(\Psi^*(g)-g_{\con})} = O(r^{\lambda -1- i}) \text{ as } r \ra \infty,\label{2_3_ac_conv}
\e
where $g_{\con} = \d r^2 +r^2h$ is the conical metric on $(r_0, \infty) \times L$, and $\nabla$, $\md{\cdot}$ are taken with respect to the conical metric. 
\end{dfn}

\begin{dfn}
Suppose that $(\R^8, \Phi)$ is an $\AC_\eta$ manifold for some $\eta < 1$ with asymptotic cone $\R^8$. Let  $A^m \subset \R^8$ be a smooth submanifold which has link $L^{m-1} \subset L_M$. Then $A$ is an $\AC_\la$ \textbf{submanifold} of $M$ ($\eta < \la < 1$), asymptotic to the cone $C = \R_+ \times L $ if there is a compact subset $K \subset A$ and a diffeomorphism $\Theta: (r_0, \infty) \times L \ra A\sm K$ such that if $\iota(r, p) = r\cdot p$ is the embedding of the cone $C \hookra \R^8$, then for every $i \in N$: 
\ea
&\io(r, p)-\Psi_M^{-1}\circ \Th(r, p)\in \nu_{(r, p)}(C) \\
&\md{\nabla^i(\Psi_M^{-1}\circ \Th(r,p)-\io(r,p))} \in O(r^{\la-i}), \text { as } r \ra \infty. \label{2_3_extrinsic_ac}
\ea
Here the norm is computed with respect to the conical metric on $(r_0, \infty) \times L$ coming from the embedding $\iota$, and the $\nabla^i$ are the higher covariant derivatives coming from the conical metric on $C$ coupled to the flat connection on $C \times \R^8$ given by the Levi-Civita connection on $\R^8$.
\end{dfn}

\begin{dfn}
Let $(M, \Phi)$ be an almost $\Spin(7)$-manifold and consider a point $p \in M$. We say that a parametrisation $\chi : B_\eta(0) \ra U$ of an open neighbourhood $U$ of $p$ is a \textbf{$\Spin(7)$-parametrisation} around $p$ if $\chi(0) = p$ and $\D \chi|_0^*\Phi_p = \Phi_0$, where $\Phi_0$ is the standard Cayley form on $\R^8$. We say that two $\Spin(7)$-parametrisations around $p$ are \textbf{equivalent} if their derivatives agree at $p$.
\end{dfn} 

\begin{dfn}
Let $N^n \subset (M, g)$ be a closed subset, and suppose that there are $z_1, \dots, z_l \in N$ such that $\hat{N} = S \setminus \{z_1, \dots, z_l\}$ is a smooth, embedded submanifold of $M$. For any $1 \le j \le l$ let $\chi_j$ be a $\Spin(7)$-coordinate system around $z_j$  and let $L_j \subset S^{7}$ be a connected $(n-1)$-dimensional Riemannian submanifold of the round sphere in $\R^8$. Then $N$ is an $\CS_{\bar{\mu}}$ \textbf{submanifold} of $(M,g)$ ($\bar{\mu}  = (\mu_1, \dots, \mu_l), 1 < \mu_j < 2$), asymptotic to the cones $C_j = \R_+ \cdot L_j \subset \R^8$ ($1 \le j \le l$) if the following holds. There is a compact subset $K \subset N$ such that $N = K \sqcup \bigsqcup_{j=1}^s U_j$ with $z_j \in U_j$ open, and  diffeomorphisms $\Psi_j = \chi_j \circ \Theta_j: (0, \R_0) \times L_j \ra U_j\setminus \{z_j\}$ for $1 \le j \le l$ such that if $\iota_j(p, r) = r\cdot p$ is the embedding of the cone $C_j$, then we have for every $i \in \N$:
\ea
&\io_j(r, p)-\Th_j(r, p)\in \nu_{(r, p)}(C_j) \nonumber \\
&\md{\nabla^i(\Th_j(r,p)-\io_j(r,p))} \in O(r^{\mu_j-i}), \text { as } r \ra 0. \label{2_3_extrinsic_cs}
\ea
Here the norm is computed with respect to the conical metric on $(0, R_0) \times L_j$ coming from the embedding $\iota_j$, and the $\nabla^i$ are the higher covariant derivatives coming from the conical metric on $C_j$ together with the flat connection on $ C_j\times \R^8$ given by the usual derivative on $\R^8$. 
\end{dfn}

\begin{rem}
 If we require that an embedded $\CS_{\bar{\mu}}$ submanifold is $\CS_{\bar{\mu}}$ with regards to any choice of $\Spin(7)$-parametrisations in the equivalence classes of $\chi_j$, we must restrict to $\mu_j < 2$. This is because the equivalence class of $\chi_j$ only determines it up to first order at the origin. The condition $\mu_j > 1$ ensures that the asymptotic cone is unique.
\end{rem}

We next define a generalisation of the radial coordinate on cones.
 
\begin{dfn}
Let $(M,\Phi)$ be $\Spin(7)$-manifold, and consider an embedded $\CS_{\bar{\mu}}$ submanifold $N \subset M$. A smooth function $\rho: M \ra [0,R_0]$ (with $R_0>0$) is a \textbf {radius function} for $N$ if near a singular point $z \in N$ is is given by the distance to $z$. Similarly, if $A \subset \R^8$ is an asymptotically conical submanifold, we say that the radial coordinate $r$ on $\R^8$ is a radius function for $A$.
\end{dfn}

\subsubsection*{Tubular neighbourhoods}

\label{2_3_tubular}
We introduce tubular neighbourhoods of the noncompact $\CS_{\bar{\mu}}$ and $\AC_\la$ manifolds, which shrink or grow like the asymptotic cones. This is a straightforward extension of the cited proposition.

\begin{prop}[{\cite[Prop. 3.4]{mooreDeformationsConicallySingular2019}}]
\label{2_3_CS_AC_tubular}
Let $C$ be either an $\AC_\la$ submanifold of $(\R^8, \Phi)$ or a $\CS_{\bar{\mu}}$ submanifold of $(M, \Phi)$. Suppose that $\rho: C \longra \R$ is a radius function. Let $\ep > 0$. Define the open subset $\nu_\ep(C) \subset\nu(C)$ as: 
\e
 \nu_\ep(C) = \{(p,v)\in \nu(C): \md{v} \le \ep\rho(p)\}.
\e
Then for sufficiently small $\ep > 0$ there is an open neighbourhood $N \subset U$ such that: 
\e
	\exp: \nu_\ep(C) \longrightarrow U \nonumber
\e
is a diffeomorphism.
\end{prop}

We note that in both cases the tubular neighbourhood scales like the radius function $\rho$ as one approaches the singular points in a $\CS$ manifold, or infinity in the $\AC$ case.

\subsubsection{Banach spaces}

We now introduce the weighted Banach spaces that appear in the deformation theory of manifolds with ends. 

\subsubsection*{Sobolev spaces}

Let  $(M,g)$ be an asymptotically conical or conically singular $n$-manifold with radius function $\rho$, with $(E, h)$ a metric real vector bundle over $M$, which admits a connection $\nabla^E$. Then the $L^p_{k, \de}$ weighted Sobolev norm is defined to be: 
\e
\nm{s}_{p, k, \de} = \left(\sum_{i = 0}^k \int_{M} \md{(\nabla^E)^i s\rho^{-\de+i}}_h^p \rho^{-n} \d \mu\right)^\frac{1}{p},
\e
and the \textbf{weighted Sobolev space} $L^p_{k,\de}(E)$ is defined to be the completion of the compactly supported sections under this norm. The sections in these spaces should be thought of as $L^p_{k, \loc}$ sections that have decay in $o(r^\delta)$. Naturally one can extend this definition to include different weights on multiple singularities. For a vector of weights $\bar{\de} \in \R^l$ these spaces will be denoted $L^p_{k, \bar{\de}}(E)$. In the above definition $\de$ must be replaced by a smooth function $w : M \longra \R$ which interpolates between the different weights. One can verify that the choice of $w$ does not impact the Banach space structure. If $E$ is a bundle of tensors, these spaces correspond to the spaces $W^p_{k, -\de, -\frac{n}{p}}$ ($\AC$ case) and $W^p_{k, \de, -\frac{n}{p}}$ ($\CS$ case) of \cite[Ch.4]{lockhartFredholmHodgeLiouville1987}, so we are able to translate their results into our setting. First of all, we have a Sobolev embedding theorem for the weighted spaces, which is an adaptation of Theorem 4.8 in \cite{lockhartFredholmHodgeLiouville1987}.

\begin{thm}
\label{2_3_Sobolev_embedding_acyl_ac}
Let $(M, g)$ be an $\CS/\AC$ manifold. Denote by $L^p_{k, \de}(E)$ the corresponding weighted Sobolev space. Suppose that the following hold: 
\begin{itemize}
\item[i)] $k - \bar{k} \ge n\left(\frac{1}{p}-\frac{1}{\bar{p}}\right)$ and either:
\item[ii)] $ 1 < p \le \bar{p} < \infty$ and $\bar{\de} \ge \de$ ($\AC$) or $\bar{\de} \le \de$ ($\CS$)
\item[ii')]$ 1 <  \bar{p} < p < \infty$ and $\bar{\de} > \de$ ($\AC$) or $\bar{\de} < \de$ ($\CS$)
\end{itemize}
Then there is a continuous embedding: 
\e
L^p_{k, \de}(E) \longra L^{\bar{p}}_{\bar{k}, \bar{\de}}(E).
\e

\end{thm}

\subsubsection*{Hölder spaces}

Let $(M, g)$ be a Riemannian manifold, and consider the induced geodesic distance function $d: M \times M \rightarrow \R$ on $M$. 

\begin{dfn}[Spaces of differentiable sections]
Let $E$ be e a metric vector bundle with a connection $\nabla$. For a section $s \in C^k(E)$ we define the $C^k$\textbf{-norm} as: 
\e
\nm{s}_{C^k} = \sum_{i=0}^k\sup_{p\in M}\md{\nabla^i s}(p).
\e
If $(M, g)$ is a conical manifold with a given radius function $\rho$, we also consider the $C^k$\textbf{-norm with weight} $\de\in\R$ instead: 
\e
\nm{s}_{C^k_\de} = \sum_{i=0}^k\sup_{p\in M}\md{\rho^{i-\de}\nabla^i s}(p).
\e
Denote the set of $C^k_{\loc}$-sections with finite $C^k_\de$-norm by $C^k_\de(E)$ and set: 
\e
C^\infty_\de (E) = \bigcap_{i=0}^\infty C^i_\de (E).
\e
Then $C^k_\de(E)$ are Banach spaces and $C^\infty_\de (E)$ is a Fréchet space. If multiple conical ends are present, the spaces $C^k_{\bar{\de}}(E)$ and $C^\infty_{\bar{\de}}(E)$ are defined in the obvious way.
\end{dfn}

For any point $p \in M$ there is an open neighbourhood $p \in U_p \subset M$ such that for any $q \in U_p$, there is a unique shortest geodesic of length $d(p,q)$ joining $p$ and $q$. In particular there is an open neighbourhood $V \subset M \times M$ of the diagonal such that for $(p,q) \in V$ we have $q \in U_p$. Let now $E$ be a metric vector bundle together with a connection. For $(p,q)\in V$ we identify the fibres $E_p$ and $E_q$ via parallel transport along the unique shortest geodesic connecting $p$ and $q$.

\begin{dfn}[Hölder spaces]
For a section $s \in C^k(E)$ and a constant $0<\al\le 1$ we define the $C^{0,\alpha}$\textbf{-semi-norm} as: 
\e
[s]_{\al} = \sup_{(x,y)\in V} \frac{\md{s(x)-s(y)}}{d(x,y)^\alpha}.
\e
The $C^{k,\al}_\de$\textbf{-Hölder norm} is then defined as: 
\e
 \nm{s}_{C^{k,\al}_\de} = \nm{s}_{C^{k}_\de}+[\rho^{k-\de+\al}\cdot \nabla^k s]_{\al}.
\e
The \textbf{Hölder space}  $C^{k, \al}_\de(E)$ is the subset of $C^{k}_\de(E)$ with finite $C^{k,\al}_\de$-Hölder norm. In the case of multiple weights, the we denote by $C^{k, \al}_{\bar{\de}}$ the corresponding Hölder space.
\end{dfn}
We also have a Sobolev embedding theorem into weighted Hölder spaces.
\begin{thm}[cf. {\cite[Thm. 2.9]{joyceReg}}]
\label{2_3_Sobolev_embedding_hoelder_acyl_ac}
Let $(M, g)$ be an $\CS/\AC$ manifold. Let $p > 1$, $k, l \ge 0$, $0< \al < 1$ and $\de \in \R$. If $k-\frac{n}{p} \ge l + \al $ then there is a continuous embedding: 
\e
L^p_{k, \de}(E) \longra C^{l, \al }_{\de}(E). 
\e
\end{thm}

\subsubsection{Elliptic operators and Fredholm results}

\label{sec_elliptic_fredholm}

Every elliptic operator on a compact manifold is Fredholm. However this useful fact does not generally hold in the noncompact setting. Consider the noncompact manifold $\R$ with the elliptic operator $\frac{\d}{\dt}$ acting on functions. 

\begin{prop}
\label{2_3_d_dt}
The elliptic operator $ \frac{\d}{\dt}: L^{2}_1(\R) \rightarrow L^2(\R) $ is not Fredholm. 
\end{prop}
\begin{proof}
We show that the image of $\frac{\d}{\dt}$ is not closed in $L^2(\R_+)$. Consider the functions $f_n \in L^{2}_1(\R_+)$ which are defined as follows:
\[ f_n(t) = \left\{ \begin{array}{cc}
     \frac{n}{t},  & t \le -n \\
     -1, & -n \le  t \le -1 \\
     t, & -1 \le t \le 1 \\
     1, & 1 \le  t \le n \\
     \frac{n}{t},  & n \le t
\end{array} \right.\]
Then clearly $f_n \in L^{2}_1(\R)$, since $\nm{f_n}_{L^2} = O(n) < \infty$ and $\nm{\frac{\d}{\dt}f_n}_{L^2} = O(1) <\infty$. As a consequence this family does not admit a limit in $L^2_1(\R_+)$. However the family of derivatives does converge in $L^2$ to the characteristic function $\chi_{[-1,1]}$, which is not in the image of $\frac{\d}{\dt}$. Thus the image of $\frac{\d}{\dt}$ is not closed, which precludes it from being Fredholm.
\end{proof}

More generally the same non-Fredholmness appears for operators on $\R \times N$ which are of the form $\frac{\d}{\dt} + A(t)$, where $A(t)$ is a self-adjoint elliptic operator on the cross-section $N$ which converges in a suitable sense to limiting operators $A_\pm$ with non-trivial kernel. In the example on $\R$ we had $A(t) = A_\pm=0$ over the point. The proof above can be generalised to this case, if we consider $f_n \psi$ instead, where $\psi$ is a non-zero element of the kernel of one of $A_\pm$. In fact, if $A_\pm$ have trivial kernel, the operator $\frac{\d}{\dt}+A$ will be Fredholm. A proof of this fact can be found in Robbin and Salamon's paper on the spectral flow \cite{robbinSpectralFlowMaslov1995}. Thus to ensure Fredholmness we need to shift the zero eigenvalues of $A_\pm$ to a non-zero value. This can be achieved by perturbing $A(t)$ to $A(t)-\de \id_N$. It turns out that this is equivalent to varying the Banach spaces by introducing the weight $e^{-\de t }$ into the norms, as we did in the previous section. Indeed, note that the norm $\nm{s}_{\cL^p_{k,\de}} = \nm{s e^{-\de\rho_t}}_{L^p_k}$ is equivalent to the previously introduced weighted norm $\nm{\cdot}_{L^p_{k,\de}}$. The advantage of this definition is that there is an isometry $\cL^p_{k, \de} \ra L^p_k$ given by sending $s \mapsto s e^{\de \rho_t}$. Thus an operator $\frac{\d}{\dt}+A(t): L^p_{k, \de} \ra L^p_{k-1, \de}$ will be Fredholm exactly when $e^{\de\rho_t}(\frac{\d}{\dt}+A(t))e^{-\de\rho_t}: L^p_k \ra L^p_{k-1}$ is. However: 
\eas
e^{\de\rho_t}\left(\frac{\d}{\dt}+A(t)\right)e^{-\de\rho_t} = \frac{\d}{\dt} +A(t)-\de \id.
\eas
In other words, perturbing the operator can be achieved by varying the weight in the definition of the Sobolev norm. This will recover the Fredholm results from the compact case. Note however that the index of an operator might depend on the weight chosen as seen in Proposition \ref{2_3_change_of_index}.
Let now $(M, g)$ be a cylindrical manifold, i.e. it admits ends which are isometric to Riemannian cylinders. Let $E$ and $F$ be two  metric vector bundles over $M$. A linear $r$-th order partial differential operator: 
\e
 D_\infty: C^{k+r}_{\loc} (E) \longra C^{k}_{\loc} (F) \nonumber
\e
is then \textbf{cylindrical} if for every section $f \in C^{k+r}_{\loc}(E)$ which is supported in an end $N = (0, \infty) \times L$ we have $(D_\infty s)(t+\cdot) = D_\infty (s(t+\cdot))$. Here $s(t+\cdot)$ denotes the translation action of $\R_+$ on the end. Now suppose that $ D: C^{k+r}_{\loc} (E) \ra C^{k}_{\loc} (F)$ is another operator and write these operators as:
\ea 
D(s) &= \sum_{i=0}^r D^i \nabla^i s, \label{2_3_operator_rep} \\
D_\infty(s) &= \sum_{i=0}^r D^i_\infty \nabla^i s, \label{2_3_operator_rep_inf}
\ea
for bounded coefficients $D^i_{(\infty)} \in C^\infty(TM^{\ot i} \ot F \ot E^*)$.  Then $D$ is \textbf{asymptotically cylindrical} if for any $j \in \N$:
\e
\md{\nabla^{j}(D^i_\infty - D^i)} \longra 0 \mbox{ as } t \ra \infty. \label{2_3_operator_conv}
\e
 Note that by translation invariance, the coefficients of $D_\infty$ are independent of $t$. Using this one can prove the following.

\begin{prop}
\label{2_3_cyl_bounded}
If $D$ is an asymptotically cylindrical operator, then for any $\de \in \R$, it extends to a well-defined map: 
\e 
	 D: L^p_{k+d, \de} (E, g) \longra L^p_{k, \de} (F, g).
\e
\end{prop}

\subsubsection*{Conical operators}

Suppose now that $(M, g)$ is $\AC_\la$ or $\CS_{{\bar{\mu}}}$, and assume that we have a radius function $\rho$ and a conical metric $g_{c}$ that $g$ is asymptotic to as $\rho \ra \infty$ and $\rho \ra 0$ respectively. We can now define  conical operators between bundles of exterior forms: 

\begin{dfn}
An linear $r$-th order operator $D: C^{k+r}_{\loc} (\La^m) \ra C^{k}_{\loc} (\La^{m'})$ is \textbf{conical with rate $\nu \in \R$} if
\e
	D^\nu = \rho^{-m'+\nu} D \rho^m	\nonumber
\e 
is an asymptotically cylindrical operator. Here the cylindrical metric is $\rho^{-2} g_c$.
\end{dfn}

 We now present the fundamental result concerning these operators, which is that they are Fredholm for almost all choices of weight $\de \in \R$. More concretely, we have the following:

\begin{thm}
\label{2_3_fredholm}
Let $D: C^{k+r}_{\loc} (\La^m) \ra C^{k}_{\loc} (\La^{m'})$ be a conical operator on an $(M,g)$ with rate $\nu$. Then for any $\de \in \R$, $P$ extends to a well-defined map: 
\e
D: L^{p}_{k+r,\de} (\La^m, g) \longra L^{p}_{k,\de-\nu}  (\La^{m'}).
\e
Furthermore if $D$ is elliptic, then this map is Fredholm for $\de$ in the complement of a discrete subset $\cD \subset \R$. This subset is determined by an eigenvalue problem on the asymptotic link.
\end{thm}
\begin{proof}
The operator $D$ is bounded whenever $D^\nu$ is. Now $D^\nu$ is bounded by Proposition \ref{2_3_cyl_bounded}. It is also Fredholm whenever $D^\nu$ is. This in turn is the case for all but a countable set of weights, which are determined by the cylindrical operator $D$ asymptotes to, as in \cite[Thm. 6.1]{lockhartEllipticDifferentialOperators1985}.
\end{proof}

Let $D$ be a conical operator of rate $\nu$, and let $D_\infty$ be the cylindrical operator that $D^\nu$ asymptotes to, as in \eqref{2_3_operator_conv}. Then the set of exceptional weights $\cD$ can be determined as follows. With respect to the parametrisation by $(t, p) \in (0, \infty) \times L$ of the cylindrical end, $D_\infty$ takes the following form: 
\e
\label{2_3_form_infty}
D_\infty = \sum_{i+j \le r} a^{i,j}_\infty \partial_t^i \nabla^j_L.
\e
At the start of this section, we have seen that the Fredholm property fails for the first order operator $\partial_t + A(t)$ if the limit $ A_+ = \lim_{t\ra\infty} A(t)$ has a zero eigenvalue. This was because the kernel gained a solution whose growth was of order $O(1)$, and thus not integrable, but could nonetheless be approximated within $L^p_k$. More generally the operator $\partial_t + A(t)- \de \id$ will not be Fredholm if $A_+- \de \id$ admits a kernel, i.e. $A_+$ admits a $\de$-eigenvector. Thus $\partial_t + A(t)$ will not be Fredholm as a map $L^p_{k, \de} \longra L^p_{k-1,\de}$ for those values $\de$ where a solution to the eigenvalue problem $A_+v = \de v$ exists. The generalisation of this to a higher-order operator in the form \eqref{2_3_form_infty} is that we consider the eigenvalue problem for: 
\e
\label{2_3_eigenvalue}
\hat{D}_\infty(\la) = \sum_{i+j \le r} a^{i,j}_\infty (i\la)^i \nabla^j_L.
\e
Denote by $\cC \subset \C$ the set of all the complex values for $\la$ for which \eqref{2_3_eigenvalue} admits a non-zero eigenvector. As Lockhart and McOwen describe in more detail in their paper \cite{lockhartEllipticDifferentialOperators1985}, this is a discrete subset of $\C$. The subset $\cD \subset \R$ of exceptional weights, which again is discrete, is then given by: 
\eas
\cD = \{ \im \la : \la \in \cC\}.
\eas
In this way the exceptional weights can be related to an eigenvalue problem on the link. Similar to the model case $\frac{\d}{\dt}+A(t)$, the existence of solutions to the eigenvalue problem implies that solutions of a certain exponential decay rate $\de$ exist. These then get added to the kernel once the rate $\de$ is passed, which makes the index jump discontinuously. Thus the Fredholm property cannot hold at these weights. Note that the dependence on the link means that operators on different $\CS$ or $\AC$ manifolds will have the same set of exceptional weights if their links agree as Riemannian manifolds, and the two operators approach the same limiting operator over that link, in the sense that the associated cylindrical operators limit to the same operator. Consider now for $\de \in \cD$ the dimension $d(\de) < \infty $ of the set of solutions to $D_\infty u = 0$, which have the form $e^{-\de t}p$, where $p$ is a polynomial in $t$ whose coefficients are sections of $E|_L$ and do not depend on $t$ . This is exactly the jump in index as a weight is passed. To be more precise, let $\de_1< \de_2$ be given such that $\de_1, \de_2 \not\in \cD$. Then we define: 
\eas
N(\de_1, \de_2) = \sum_{\de \in (\de_1,\de_2)\cap \cD} d(\de).
\eas
The relation between the indices of differential operators on differently weighted Sobolev spaces in then given as follows: 

\begin{thm}[cf. {\cite[Thm. 1.2]{lockhartEllipticDifferentialOperators1985}}]
\label{2_3_change_of_index}
Let $D$ be an elliptic conical operator of order $r\ge 0$ and rate $\nu \in \R$. Let $1<p<\infty$ and $k\ge 0$. Denote by $i_\de(P)$ for $\de \in \R$ the index of the following operator: 
\eas
D: L^p_{k+r, \de}(E) \longra L^p_{k, \de-\nu}(E).
\eas
We then have that: $i_{\de_2}(P)-i_{\de_2}(P) = N(\de_1, \de_2)$.
\end{thm}
\begin{prop}[cf. {\cite[Lem. 2.8]{joyceReg}}]
\label{2_3_dual}
Assume that $1<p,q< \infty$ are such that $\frac{1}{p}+\frac{1}{q}= 1	$. Then if $n \in \N$ is the dimension of the underlying manifold and $\de \in \R^l$ is a vector of weights, there is a perfect pairing $L^p_{\de} \times L^q_{-n-\de} \longra \R$, and thus $(L^p_{\de})^* = L^q_{-n-\de}$.
\end{prop}
\begin{ex}
\label{2_3_cayley_plane}
The simplest example of Cayley cone is a Cayley plane  $\Pi = \R^4 \times 0 \subset \R^8$ which a round $S^3$ as link. The limiting operator in this case is $D_L$ from Equation \eqref{2_2_dirac_op_cone}. Even more than that, we can think of a Cayley plane as being induced by a special Lagrangian plane for a carefully chosen embedding CY4 structure on $\R^8$. By Proposition \ref{2_2_cayley_sl} we can write:
\e
D = -\d\star\op \d^- = \frac{\d}{\d r} + r^{-1} D_{S^3}.
\e
Combining the work done in \cite{lotayStabilityCoassociativeConical2009}, where coassociative cones were analysed, and slightly extending the work done in \cite{kawaiDeformationsHomogeneousAssociative2014}, where the non-coassociative Cayley deformations of cones where studied (but not other critical rates), we can give all the critical rates $\cD$ in the range $(-4, 2)$. They are ${-3, -1, 0,1}$, and the eigenspaces have dimensions: 
\e
 d(-3) = 1 + 0,\quad d(-1) = 1 + 0,\quad d(0) = 3+1, \quad d(1) = 8 + 4.
\e
Here the first summand corresponds to the coassociative contribution, and the second is the truly Cayley contribution. 
\end{ex}
\begin{ex}
\label{2_3_quadratic_cone}
Consider the complex cone $C_q = \{x^2 + y^2 + z^2 = 0, w = 0\} \subset \C^4$ which has link $L \simeq \SU(2)/ \Z_2$. In particular, this is also a Cayley cone. The critical rates in $(-2, 2)$  are ${-1, 0,1,  -1+\sqrt{5}}$, and the eigenspaces have dimensions: 
\e
d(-1) = 2 + 0,\quad d(0) = 7+1, \quad d(1) = 16 + 6,\quad d(-1+\sqrt{5}) = 3 + 3.
\e
Here the first summand corresponds to the coassociative contribution, and the second is the truly Cayley contribution. 
\end{ex}

\section{Almost Cayley submanifolds}

Denote the Grassmannian of oriented $4$-planes in an $8$-dimensional vector space $V$ by $\Gr_+(4, V)$. If  $(V, \Phi)$ is a $\Spin(7)$-vector space we can additionally consider the \textbf{Grassmannian of Cayley planes} in $(V, \Phi)$, which we denote by $\Cay(V, \Phi)$. The group $\Spin(7)_\Phi$ acts on $\Gr_+(4, V)$, and acts transitively on $\Cay(V, \Phi)$. As the stabiliser group of any Cayley is isomorphic to $H = (\SU(2)^3)/\Z_2$, we have $\dim \Cay(V, \Phi) = \dim \Spin(7)- \dim H = 21 - 9 = 12$. As $\dim \Gr_+(4, V) = 16$, we see that $\Cay(V, \Phi)$ is a codimension $4$ submanifold of the Grassmannian of oriented four-planes. In other words the Cayley condition can be given in terms of four independent equations. This is the reason why the bundle $E$, which will appear later as the co-domain of the deformation operator of a Cayley, is a rank $4$ bundle. We can think of the fibers of $E$ as corresponding to the normal bundle of a given Cayley in  $\Cay(V, \Phi)\subset \Gr_+(4, V)$. Since $\Spin(7)_\Phi$ is compact and the action is smooth, there is a metric $g_{\Spin(7)}$ on $\Gr_+(4, V)$, such that $\Spin(7)$ acts by isometries. Such a metric can be realised by embedding $\Gr_+(4, V) \hookra \Lambda^4V$ via $\spn \{e_1,e_2,e_3,e_4\} \mapsto e_1\wedge e_2 \wedge e_3 \wedge e_4$. The $\Spin(7)$-invariant metric is then the restriction of the euclidean metric on $\Lambda^4 V$. The resulting distance map is uniformly equivalent to the following $\Spin(7)$-invariant distance defined in terms of the orthogonal projections onto planes: 
\eas
d_{\Gr}(E, E') = \nm{\pi_E - \pi_E'}_{\ope}.
\eas
Here $\pi_E, \pi_{E'}$ are the orthogonal projections onto $E$ and $E'$ respectively. Let us take a closer look at a tubular neighbourhood of the Cayley planes inside $\Gr_+(4, V)$.

\begin{prop}
\label{3_1_almost_cayley}
Let $\ep_1 \in [0, 1), \ep_2 \in [0, 2), \ep_3 \in [0, \infty)$ be given and consider the sets:  
\ea 
 E_1 = \{ \xi& \in \Gr_+(4, V) :  \Phi|_\xi > (1-\ep_1) \dvol_\xi \}, \nonumber \\
 E_2 = \{ \xi& \in \Gr_+(4, V) : \nm{\tau(f_1, f_2, f_3, f_4)}_{\La^2_7} < \ep_2, \nonumber \\
  &\{f_i\}_{i = 1, \dots, 4} \text{ is an orthonormal basis of } \xi \},  \nonumber \\
  \text{ and } E_3 = &\{ \xi \in \Gr_+(4, V) :  \xi = \spn\left\{e_1, e_2, e_3, \frac{e_4 + \sqrt{{\al}}v}{\sqrt{1+{\al}}}\right\}, \nonumber \\ 
 & e_4 = e_1 \times e_2 \times e_3,\   e_i \in V\text{ orthonormal },\   v \perp e_i, \nm{v} = 1, 0 \le \al < \ep_3\}.
\ea
Note that $\nm{\tau(f_1, f_2, f_3, f_4)}_{\La^2_7}$ is independent of the choice of basis of $\xi$. 
Then for a choice of one of the $\ep_i$ we can determine the other two such that the three sets agree. 
\end{prop} 
\begin{proof}
These three families of sets are $\Spin(7)$-invariant. The sets $E_1$ and $E_2$ are invariant by the definition of $\Phi$ and $\tau$ respectively. The invariance of $E_3$ follows from the fact that there are elements of $H \subset \Spin(7)$ which keep $\xi$ fixed while acting transitively on the unit sphere in $\xi^\perp$.

 Now note that the orbit of a single $\xi \in E_3$ is a sphere of a given radius in the open ball $E_3$. Thus choosing elements for every radius, we can exhaust all of $E_3$. Furthermore spheres in $E_1$ and $E_2$ (i.e. replacing the inequality with equality in the definition) must be unions of this set for different values of $\al$. What is left to show is that the value of $\al$ determines the radii of the spheres in $E_i$ ($i=1,2$) uniquely. For $E_1$ for instance it can be computed using equation \eq{2_1_cayley_form} that $1-r_1 = \frac{1}{\sqrt{1+\al}}$. For $r_2$ we see that $r_2 = \frac{2\sqrt{\al}}{\sqrt{1+\al}}$ , using the coordinate representation \eq{2_1_tau_coord} of $\tau$.
\end{proof}

For $\alpha \in (0, 1)$, consider consider the set of \textbf{almost Cayley planes}: 
\ea
\Cay_\al(V, \Phi) = \{ \xi \in \Gr_+(4, V):  \Phi|_\xi > \alpha \dvol_\xi\}.
\ea
By the $\Spin(7)$-invariance of $\Cay_\al(V, \Phi)$ this set also admits a canonical action of $\Spin(7)$. In fact, $\Cay_\al(V, \Phi)$ is a tubular neighbourhood of $\Cay(V, \Phi)$ under geodesic normal coordinates for $g_{\Spin(7)}$ and $\alpha$ sufficiently close to $1$. Let $0< \alpha_0 <1$ be such that for all $\al > \al_0$ the set $\Cay_\al(V, \Phi)$ has this tubular neighbourhood property. Let now $(M, \Phi)$ be an almost $\Spin(7)$-manifold. We can then consider the associated fibre bundles:

\begin{itemize}
\item $\Gr_+(4, TM) = P_{\Spin(7)} \times _{\Spin(7)} \Gr(4, \R^8)$
\item $\Cay(M) = P_{\Spin(7)} \times _{\Spin(7)} \Cay(\R^8, \Phi_0)$
\item $\Cay_\al(M) = P_{\Spin(7)} \times _{\Spin(7)} \Cay_\al(\R^8, \Phi_0)$
\end{itemize}

We see that a submanifold $N^4 \subset M$ is Cayley exactly when $TN$, seen as a section of $\Gr_+(4, TM)$ over $N$, takes values in $\Cay(M)$. Analogously we say that a submanifold of $M$ is $\al$-Cayley if the the section $TN$ takes values in $\Cay_\al(M)|_N$. Now for every $p \in M$, we have that as $\Spin(7)$-vector spaces $(T_pM, \Phi_p) \simeq (\R^8, \Phi_0)$, thus $\Cay_\al(T_pM, \Phi_p)$ will be a tubular neighbourhood of $\Cay(T_pM, \Phi_p)$ whenever $\al > \al_0$. In particular, for an $\al$-Cayley $N$ with $\al > \al_0$ we get a canonical section $\cay_N$ of $\Cay(M)|_N$ defined as the closest Cayley plane $\cay_N(p)$ to the given almost Cayley $T_pN$, as measured by the metric $d_{\Spin(7)}$. This Cayley plane is unique because of the tubular neighbourhood property. However, we cannot in general guarantee that this closest Cayley plane will have an intersection of dimension $3$ with our almost Cayley tangent plane. Putting this additional restriction on the canonical Cayley section is equivalent to finding a continuous section of a $\Gr_+(3, 4)$-bundle over $N$ (as a Cayley which has three-dimensional intersection with $T_pN$ is entirely determined by this three-plane, the fourth basis vector can be obtained via the triple product). However this bundle is in general non-trivial, and thus might not admit a global section. Thus we cannot in general expect to write $\cay_N$ in the form of elements of $E_3$ in Proposition \ref{3_1_almost_cayley}. We will therefore settle for a slightly more general form, which however still relates an adapted frame of $N$ with a $\Spin(7)$-frame adapted to $\cay_N$ in a satisfactory way.

\begin{prop}[Local coordinates around an $\al$-Cayley]
\label{3_1_local_coord}
Let $N$ be an $\al'$-Cayley submanifold of $M$, where $\al' > \al_0$, and let $\cay_N$ be the canonical Cayley section associated to $N$. Let $p \in N$. Write $\Phi|_N = \alpha \dvol_N$, for a smooth function $\alpha: N \ra (\al_0, 1]$. Then there is a constant $1>\al_1> \al_0$ such that for every $\al' > \al_1$ we can then find a $\Spin(7)$-frame $\{e_i\}_{i = 1, \dots, 8}$ adapted to $\cay_N$ around $p$ such that: 
\ea
T_pN = \spn\left\{\beta_i e_i + v_i  \right\}_{i = 1, \dots, 4} , \nonumber \\
\nu_p(N) = \spn\left\{\beta_i e_i + v_i  \right\}_{i = 5, \dots, 6}, 
\ea
where the basis vector fields $\beta_i e_i + v_i$ are orthonormal. Here $v_i$ for $i = 1, \dots, 8$ are vector fields such that: 
\eas
\left\{
\begin{array}{l}
v_1, v_2,v_3,v_4 \perp \cay_N ,\\
v_5, v_6,v_7,v_8 \parallel \cay_N, \\
\nm{v_i} \le C_{\al_1}( 1-\al) ,
\end{array}
\right.
\eas
and $\beta_i$ are functions such that $1-\beta_i \ge C_{\al_1}(1-\al)$.
\end{prop}
\begin{proof}
For every $p \in N$, the planes $\cay_N(p)$ and $T_pN$ are close enough together so that $\nm{\pi_{\cay_N(p)}-\pi_{T_pN}}_{\ope} \lesssim 1-\al$, where $\pi_V$ is the orthogonal projection onto $V\subset T_pM$. Thus, if we take a $\Spin(7)$-frame $\{e_i\}_{i = 1, \dots, 8}$ which is adapted to $\cay_N$, then the tangent vectors  $\tilde{f}_i = \pi_{T_pN}( e_i)$ for $i = 1, \cdots, 4$ will be such that $\nm{e_i-\tilde{f}_i} \lesssim 1-\al$. After applying the Gram-Schmidt orthogonalisation procedure to $\tilde{f}_i$ to obtain orthogonal vectors $f_i$, we still have that $v_i = \beta_i e_i-f_i \in \cay_N^\perp$ and $\nm{v_i} \lesssim 1-\al$, for some functions $\beta_i$ coming from the Gram-Schmidt algorithm. Note that the hidden constant in our $\lesssim$-notation only depends on $\al_1$. Similarly the coefficients $\beta_i$ tend to $1$ as $\al$ tends to $1$. An analogous argument applies to the normal vectors, using the fact that we also have $\nm{\pi_{\cay_N^\perp(p)}-\pi_{\nu_p(N)}}_{\ope} \lesssim 1-\al$. 
\end{proof}

\subsection{The deformation operator}
\label{3_2_deformation_operator}
Consider a $\Spin(7)$-manifold $M$. Let $\al_0 \in (-1, 1)$ be sufficiently close to $1$ such that $\Cay_{\al_0}(T_pM, \Phi_p)$ is a tubular neighbourhood of $\Cay(T_pM, \Phi_p) $ for every $p\in M$. Let $N$ be an $\al_1$-Cayley (where $\al_1 > \al_0$) with a tubular neighbourhood $N \subset U \subset M$. In other words, we require that the exponential map $\exp: V \subset \nu(N) \longra U$ defines a diffeomorphism onto its image, where $V$ is some open subset of the normal bundle of $N$. For $v \in C^\infty(N,V)$, i.e. $v$ is a normal vector field to $N$ with values in $V$, we define $\exp_v: N \hookra M$ to be the embedding given by $\exp_v(p) = \exp(v(p))$. This is a small perturbation of $N$ inside $U$, and in fact any $C^k$-small perturbation of $N$ (where $k > 0$) can be obtained as the image of $\exp_v$ of a unique $C^k$-small normal vector field $v$. We denote this image by $N_v$. Our goal is to construct a Cayley submanifold of the form $N_v$, whenever $N$ is close to being Cayley. As we have seen, $N$ admits a canonical section $\cay_N: N \ra \Cay(M)|_N$. This section can also be seen as a four-dimensional subbundle of $TM|_N$, where each fibre is a Cayley plane. This allows us to globalise the definition of the subspace $E_\xi \subset \La^2_7V^*$ of \eqref{2_2_E_bundle} and define the four-dimensional vector bundle: 
\e
	E_{\cay} = \{ \omega \in \La^2_7: \omega|_{\cay_N} = 0\}.
\e

 Let now $\eta \in \Omega^k(N_v, F|_{N_v})$ be a differential form with values in a bundle of tensors $F \rightarrow M$ over the submanifold $N_v$. The form $\tau$ is an example of such a form. Usually when one talks about the pull-back $\exp_v^* \eta$, the result is a form in $\Omega^k(N, \exp^*_v F|_{N_v})$. However when we write $\exp_v^*\eta$ in the following, we mean a form in $\Omega^k(N,F|_{N})$, which we define as follows. Extend the normal vector field $v$ on $N$ to a vector field on $U$, where $v(\exp_v(p))$ is defined to be the parallel transport of $v(p)$ along the geodesic which starts at $p$ and has initial velocity $v(p)$. By the definition of $U$ as a tubular neighbourhood for geodesic normal coordinates, this gives rise to a smooth extension of $v$ to all of $U$. This gives rise to a flow $\phi_t: U \ra M$ which has the property that $\phi_1(p) = \exp_v(p)$. We now define the pullback of a decomposable form $\eta = \omega \ot s$ to be given by the following formula: 
\e
\label{3_2_exp}
	\exp^*_v \eta = \phi_1^*\omega  \ot \D\phi_1^* s = \exp_v^*\omega  \ot \D\phi_1^* s.
\e
Here $\phi_1^*$ and $\exp_v^*$ are the usual pullback of differential forms, and $\D \phi_1^*: F_{\exp_v(p)}\ra F_p$ is the pullback induced from the linear isomorphism $\D \phi_1: F_{p} \ra F_{\exp_v(p)}$ induced by $\phi_1$ on any tensor bundle. To summarise, we extended the vector field $v$ in order to have a pullback operation that also pulls back the bundle in which the differential form is valued. We now define the \textbf{deformation operator} associated to $N$ as follows: 
\ea
\label{3_2_def_op}
\begin{array}{rl}
F: C^\infty (N,V) & \longrightarrow C^\infty(E_{\cay}) \\
v &\longmapsto \pi_E(\star_N \exp^*_v(\tau|_{N_v})) .
\end{array}
\ea
Here $\exp_v^*\tau$ is our non-standard definition of a pull back we introduced just now. This addition is necessary, otherwise the resulting sections would be valued in different bundles for varying $v$. Furthermore $\pi_E$ denotes the orthogonal projection $\La_7^2|_N \ra E_{\cay}$, which uses the chosen Cayley section $\cay_N$ over $N$. Note that in the definition of $F$, we include a projection onto $E_{\cay}$. This is so that the operator maps between bundles of the same rank, otherwise it could not be elliptic. However one might ask whether projecting onto a subspace loses information. Heuristically, the condition $F(v) = 0$ is given by ignoring $3$ of the $7$ equations obtained from the condition $\star_N \exp^*_v(\tau|_{N_v}) = 0$, one for each of the basis vectors of $\La_7^2|_{N_v}$. The fact that this still gives enough equations to determine the Cayleyness of a plane can be expected for a generic choice of a projection to a four-dimensional subbundle of $\La_7^2|_N$, since the Grassmannian of Cayley planes is of co-dimension 4 in the Grassmannian of oriented four-planes, i.e. the Cayley condition on a four-plane can be described with four independent equations. Notice however that the projection is onto a fibre of $E_{\cay}$ over $p\in N$, whereas the equations we chose are situated at $\exp_v(p)\in N_v$, so this is only approximately true, as we will see in the next proposition. We will crucially rely on the fact that we choose $v$ to be small in the $C^1$-norm to overcome this discrepancy and prove that we can ignore $3$ equations while still retain the Cayley-detecting property of $\tau$.

\begin{prop}[Detects Cayleys]
\label{3_2_does_detect_cayley}
Let $(M, \Phi)$ be a $\Spin(7)$ manifolds with uniformly bounded Riemmann curvature tensor $\md{R} < C$. Let $0 < \al_0  < 1$ such that $\Cay_{\al_0}(T_pM, \Phi_p)$ is a tubular neighbourhood of $\Cay(T_pM, \Phi_p)$. Then there is $\al_0 < \al_1 < 1$ and a constant $E_{C_1} = E_{C_1}(\Phi, \al_0, \al_1) < 1$ such that the following holds for any $\al_1$-Cayley $N$. If $v \in C^\infty(N, V)$ is such that for all $p \in N$:
\eas
	\md{v(p)} < \frac{E_{C_1}}{\min\{\md{R(q)}: q \in B(p, \md{v(p)})\}}, \quad  \md{\nabla v(p)} < E_{C_1},
\eas
then $N_v$ is Cayley exactly when $F(v) = 0$. 
\end{prop}
\begin{proof}
The submanifold $N_v$ is Cayley exactly when $\tau|_{N_v} = 0$, which is equivalent to $\star_N \exp_v^*(\tau|_{N_v}) = 0$. Thus we need to prove that $\star_N \exp_v^*(\tau|_{N_v}) = 0$ if and only if $\pi_E(\star_N \exp_v^*(\tau|_{N_v})) = 0$. Let $p \in N$ be given. Let $\al_0 <\tilde{\al}< \al_1$ be given. The set of $\tilde{\al}$-Cayley planes in a given $\Spin(7)$-vector space is an open $\ep$-neighbourhood of the set of $\al_1$-Cayley planes, for some $\ep> 0$ dependent on $\tilde{\al}$
. The same holds true globally in $\Gr^+(4, TM)$, since everything is pointwise isometric to the standard model $(\R^8, \Phi_0)$. In particular, since $T_{\exp_v(p)}N_v$ is determined from $T_pN$ by knowing only $v(p)$ and $\nabla v(p)$, this implies that if $\md{v(p)}$ is small compared to the curvature at any point within a distance $\md{v(p)}$ and $\md{\nabla v(p)}$ is sufficiently small at every point $p \in N$, say smaller than some $\tilde{E}_{C_1}$, then $N_v$ is still $\tilde{\al}$-Cayley. Thus we still have a  canonical Cayley section $\cay_{N_v}$ over $N_v$. We then apply Proposition \ref{3_1_local_coord} to choose a $\Spin(7)$-frame $\{e_i\}_{i = 1, \dots, 8}$ adapted to $\cay_{N_v}$ such that an orthonormal basis of $T_{\exp_v(p)}N_v$ is given by $f_i = \beta_i e_i + v_i$, where $\beta_i:U \ra [0, 1]$  are smooth functions. Under $(\D \exp_v(p))^{-1}$ these in turn get mapped to a local frame $\{f'_i\}_{i = 1, \dots, 4}$  of $TN$, which are not necessarily orthogonal. We now have that: 
\ea
\det (f'_i) \star_N \exp_v^*(\tau|_{N_v})(p) &= \exp_v^*(\tau|_{N_v})_p(f'_1,f'_2,f'_3,f'_4)   \nonumber \\
&= \D\phi_1^*\tau_{\exp_v(p)}(f_1,f_2,f_3,f_4) \nonumber,
\ea
where $\det (f'_i)$ is the volume of the $4$-parallelepiped spanned by the $f'_i$. We similarly get:
\ea
\det (f'_i) \pi_E  \star_N \exp_v^*(\tau|_{N_v})(p)& = \pi_E \D\phi_1^* \tau_{\exp_v(p)}(f_1,f_2,f_3,f_4) \nonumber \\
&= \D\phi_1^* \pi_{\tilde{E}} \tau_{\exp_v(p)}(f_1,f_2,f_3,f_4) \nonumber.
\ea
Here $\pi_E$ is the orthogonal projection onto $(E_{\cay})_{p}$ and $\pi_{\tilde{E}} = (\D\phi_1^*)^{-1}\pi_{E} \D\phi_1^*$ is a not necessarily orthogonal projection onto a subspace $\tilde{E} \subset \La_7^2|_{\exp_v(p)}$. As $\D \phi_1^*$ is an isomorphism, it suffices to show that: 
\eas
\pi_{\tilde{E}} \tau_{\exp_v(p)}(f_1,f_2,f_3,f_4) = 0\Ra \tau_{\exp_v(p)}(f_1,f_2,f_3,f_4) = 0.
\eas
Let $q = \exp_v(p)$ for simplicity. Consider $\tau_q : \Gr^+(4, T_{q}M) \ra (\La^2_7)_{q} M$ as a smooth map. As $\tau_q$ vanishes on a $12$-dimensional submanifold in the $16$-dimensional manifold $\Gr^+(4, T_qM)$, its derivative  at a Cayley can have rank at most $4$, in other words $\im D\tau_q \subset (\La^2_7)_{q} M$ is four-dimensional at most. Now we use the coordinate expression for $\tau_q$ in \eqref{2_1_tau_coord} with regards to the frame $\{e_i\}_{i=1, \dots, 8}$, as well as the special form $f_i = \beta_i e_i + v_i$ with $v_i \perp \cay_{N_v}$ to see that $\im D\tau_q = E$. This is because in the $E^\perp$-component of $\tau_q$ (which is spanned by $e_1 \times e_2$, $e_1 \times e_3$ and $e_1 \times e_4$) varies quadratically in both $e_i$ and $v_i$. This implies that if $f_1,f_2,f_3$ and $f_4$ span a sufficiently small perturbation of $\cay_{N_v}(q)$, as is the case in our setting by the $\tilde{\al}$-Cayleyness of $N_v$, then $\pi_E\tau(f_1, f_2, f_3, f_4) = 0$ can only occur if $\tau(f_1, f_2, f_3, f_4) = 0$, as $E$ and $E^\perp$ are transversal subspaces. In fact, if $E'$ is any four-dimensional subspace and $\pi_{E'}$ is any projection onto $E'$, not necessarily orthogonal, such that $\ker_{E'} \cap E = \{0\}$, then the implication: 
\e
\label{3_2_tau_condition}
\pi_{E'} \tau_{\exp_v(p)}(f_1,f_2,f_3,f_4) = 0\Ra \tau_{\exp_v(p)}(f_1,f_2,f_3,f_4) = 0 
\e
holds, given that the $f_i$ span a sufficiently small perturbation of $\cay_{N_v}$. In other words for a given $E'$, if $\tilde{\al}$ is sufficiently close to $1$, the implication above will be true. Moreover, having fixed a neighbourhood of $\cay_{N_v}$ for the $f_i$ to vary in, satisfying the above implication is clearly an open condition on $\pi_{E'}$.

Let now an $\tilde{\al}$-Cayley plane $\xi_p \subset T_pM$ be given, with associated canonical Cayley plane $\cay_p$. If we restrict the $f_i$ to vary so that the planes they span are in a sufficiently small neighbourhood of $\cay_p \in \Gr^+(4, T_pM)$, then the implication \eqref{3_2_tau_condition} holds for an open subset of maps $\pi_{E'}$ which contains $\pi_E$. Note that this open subset of maps $\pi_{E'}$ depends on the neighbourhood of $\cay_p$ that we fixed. Moreover, this is true at every point $p \in M$. Denote this open subset of projections which contains $\pi_E$ by $\Pi$, and let the fibre over the point $p$ be $\Pi_p$. It thus remains to show that $\pi_{\tilde{E}}\in \Pi_q$. For this notice that from a fixed $v(p) $ and $\nabla v(p)$, we can determine $\D\phi_t : T_pM \ra T_q M$. This is because first order deviations of geodesics at a later time depend solely on the first order deviation of geodesics at an earlier time. Thus we get a family of maps $\pi^t_{\tilde{E}}$ which act on the fibre of $\La^2_7$ over the point $\phi_t(p)$. We have $\pi^0_{\tilde{E}} = \pi_E$, thus  $\pi^0_{\tilde{E}}\in \Pi_p$. Next, by what we have said above it is clear that $\pi^t_{\tilde{E}}$ being contained in $\Pi_{\exp_{tv}(p)}$ is an open condition in $t$. Thus for a fixed $v(p)$ and $\nabla v(p)$ there is a $t_0 = t_0(v(p), \nabla v(p))$ such that for all $0 \le t < t_0$ we have $\pi^t_{\tilde{E}} \in \Pi_{\exp_{tv}(p)}$. Now note that for $\la > 0$:
\eas
t_0(\la v(p), \la \nabla v(p)) = \la^{-1}t_0( v(p), \nabla v(p)).
\eas
This implies that for $\md{v(p)}\cdot \md{R} +\md{\nabla v(p)}$ sufficiently small we have $\pi_{\tilde{E}}\in \Pi_q$. Now because the Riemann curvature tensor is bounded on $M$, we can find a constant $E_{C_1} < \tilde{E}_{C_1}$ such that whenever $\md{v(p)}\cdot \md{R} +\md{\nabla v(p)} < E_{C_1}$, then the implication \eqref{3_2_tau_condition} holds at $q$. 
\end{proof}

Next, we are interested in studying the linearisation of this linear operator and to show that it is elliptic at the zero section.

\begin{prop}[Linearisation]
Let $N$ be an $\bar{\al}$-Cayley with $\bar{\al} > \al_1$ with deformation operator $F$, such that $\Phi|_N = \alpha \dvol_N$. Let $p\in N$ and suppose that near $p$ we have a $\Spin(7)$-frame $\{e_i\}_{i=1,\dots, 8}$ and a frame $\{f_j\}_{j = 1, \dots, 8}$ which respects the splitting $TM = TN \op \nu(N)$ as in Proposition \ref{3_1_local_coord}. The linearisation of $F$ at $0$ is then given by: 
\ea
\label{3_2_linearised_def_new}
\begin{array}{rl}
D: C^\infty (\nu(N))  &\longrightarrow C^\infty(E_{\cay}),  \\
v &\longmapsto \pi_E(\beta\sum_{i=1}^4 f_i \times \nabla^\perp_{f_i} v  \\
&+ \sum_{i=1}^4\sum_{j=1}^8 \beta_{ij} f_{j} \times \nabla^\perp_{f_i} v \\
&+ \nabla_v \tau (f_1, f_2, f_3, f_4)),
\end{array}
\ea
where $\nabla^\perp$ is the Levi-Civita connection on the normal bundle $\nu(N)$, and $\nabla^\top_{f_i} v$ is the projection of $\nabla_{f_i}v$ to $TN$. Here $\beta, \beta_{ij}$ are smooth functions such that $1-\beta \ge C_{\al_1}(1-\al)$ and $\md{\beta_{ij}} \le C_{\al_1}(1-\al)$. The constant $C_{\al_1}$ only depends on the the choice of $\al_1$. When $N$ is a Cayley, then the second line is the Dirac operator associated to a Cayley from Proposition \ref{2_2_dirac_bundle}. Furthermore the third line vanishes in this case. Finally the fourth line vanishes if $(M, \Phi)$ is torsion-free.
\end{prop}

\begin{proof}
We have for $v \in C^\infty (\nu(N)) $:
\ea
	\D F[v] &= \frac{\d}{\dt}\bigg\rvert_{t = 0} F(tv) \nonumber \\
	 &= \pi_E \star_N \frac{\d}{\dt}\bigg\rvert_{t = 0} \exp^*_{tv}(\tau_{N_v})\nonumber .
\ea
Note that $\phi_t = \exp_{tv}$ is the flow of a vector field on a neighbourhood of $N$ which extends $v$, as we discussed before the definition of the deformation operator in equation \eq{3_2_def_op}. Thus we get from the definition of the Lie derivative, the expression \eq{3_2_exp} for our non-standard definition of pull-back, and the definition of the Levi-Civita covariant derivative on forms: 
\ea
	\D F[v] &= \pi_E \star_N \Lie_v\tau 
	=  \pi_E \Lie_v( \tau(f_1, f_2, f_3, f_4))  \\
	&=  \pi_E \left(\tau(\nabla_{f_1} v, f_2, f_3, f_4) +\tau(f_1, \nabla_{f_2} v, f_3, f_4)\right.\nonumber \\
	&\left.+\tau(f_1, f_2, \nabla_{f_3} v, f_4)+\tau(f_1, f_2, f_3, \nabla_{f_4} v)+\nabla_v\tau(f_1, f_2, f_3, f_4)\right)  \nonumber .
\ea

Here the last line uses the fact that $\nabla$ is torsion-free. Now consider the term $\tau(\nabla_{f_1} v, f_2, f_3, f_4)$. Since we have the orthogonal splitting $TM|_N =  TN \op \nu(N)$, the connection $\nabla$ on $TM|_N$ splits into $\nabla^\top + \nabla^\perp$, where $\nabla^\top = \pi_{TN} \circ \nabla$ and $\nabla^\perp = \pi_{\nu(N)} \circ \nabla$. Thus in particular: 
\ea
\tau(\nabla_{f_1}v, f_2, f_3, f_4) &= \tau(\nabla^\perp_{f_1}v, f_2, f_3, f_4) + \tau(\nabla^\top_{f_1}v, f_2, f_3, f_4) \nonumber \\
&= \nabla^\perp_{f_1} \times ( f_2 \times f_3\times f_4) + \tau(\nabla^\top_{f_1}v, f_2, f_3, f_4)\nonumber \\
&= \nabla^\perp_{f_1} \times ( f_2 \times f_3\times f_4)\nonumber  \\
&= \nabla^\perp_{f_1} \times ( \beta_2 e_2 \times \beta_3 e_3\times (\beta_4 e_4 + v_4)) \nonumber \\
&= \nabla^\perp_{f_1} \times ( \beta_2 \beta_3 \beta_4 e_1 + \sum_{j=1}^8\tilde{\beta}_{1j} f_j) \nonumber \\
&= \nabla^\perp_{f_1} \times ( \beta f_1 + \sum_{j=1}^8\beta_{1j} f_j). \nonumber \
\ea
For the third line we used the fact that $\nabla_w v(p) \perp T_pN $, since $v \in C^\infty(\nu(N))$. For the rest we use the definition of $\tau$ as well as the coordinate representation of $\Phi_0$. Here $\beta = \beta_1 \beta_2 \beta_3 \beta_4$. The computations for the remaining three terms are similar and lead to the claimed formula.

If $N$ is Cayley, then the second line corresponds exactly to the Dirac operator associated to a Cayley from Propositions \ref{2_2_dirac_bundle}, since in this case $\alpha \equiv 1$, which implies $\beta \equiv 1$, and the cross product of $f_i$ and $\nabla_{f_i} v$ already lies in $E_{\cay} = E$, so no further projection is required. In the same situation we see that the third line vanishes, as all the $\be_{ij}$ vanish. Finally, $\tau$ is covariantly constant if $\Phi$ is torsion-free, and this the last line vanishes in this case. To see this note that we can choose a local $\Spin(7)$-frame which is covariantly constant, from which it is clear that the cross and triple product send parallel sections to parallel sections. Since $\tau$ is defined in terms of these two products, the same must hold true for $\tau$. From this it is immediate that $\tau$ is covariantly constant. 
\end{proof}

\begin{prop}[Ellipticity]
\label{3_2_ellipticity}
There is a universal constant $\al_2 > \al_1$, such that if $N$ is an $\al_2$-Cayley submanifold, then its associated deformation operator is elliptic at the zero-section.
\end{prop}

\begin{proof}
Let $D$ be the linearisation of the deformation operator at $0$. From the the previous proposition we see that the symbol at $p \in N$  is given in an adapted frame as follows, where $\xi \in T^*_pN$ and $\xi_i = \xi(f_i)$: 

\ea
\si_\xi(D) &= \pi_E \left(\beta \sum_{i=1}^4 f_i \times \xi_i v + \sum_{i=1}^4\sum_{j=1}^8 \beta_{ij} f_{j} \times \xi_i v \right) \nonumber \\
&=  \pi_E \left( \left(\beta\xi^\sharp +\sum_{i=1}^4\sum_{j=1}^8 \beta_{ij} f_{j} \times \xi_i \right)\times  v \right).
\ea

Now if the $f_i$ span a Cayley plane, then this is exactly the symbol of the Dirac operator associated to a Cayley from Proposition \ref{2_3_clifford_mult}, and thus invertible. As we perturb the plane continuously, we see that both the product $\pi_E(\Box \times \Box)$ and the expression inside the bracket vary continuously. Thus invertibility of the composed expression is an open condition on the set of four-planes, which holds at Cayley planes. Since $\al$-Cayley planes for $\al \in [\al_1,1)$ form a neighbourhood basis of the Cayley planes in $\Gr_+(4, \R^8)$ we see that there is a universal $\al_2 > \al_1$ such that whenever the $f_i$ span a $\al_2$-Cayley, the symbol is invertible, and thus $D$ is elliptic.  
\end{proof}

\section{Deformation theory of Cayley submanifolds}

\subsection{Compact case}

\label{3_4_deformation_map}
We will now study the deformation theory of a compact Cayley submanifold in a almost $\Spin(7)$-manifold $(M, \Phi)$, where we also allow the $\Spin(7)$-structure to vary in a finite-dimensional smooth family $\{\Phi_s\}_{s \in \cS} \subset \mathscr{A}(M)$, with $\Phi = \Phi_{s_0}$ for some $s_0 \in \cS$. As we are only interested in the local deformation theory, we can and will assume that $M$ is compact.  The analysis will be done for almost Cayleys, which will be useful later when we will desingularise Cayleys with conical singularities. 

Let $N \subset (M, \Phi)$ be $\al$-Cayley with $\al$ strictly bigger than $\al_2$ from Proposition \ref{3_2_ellipticity}, so that the linearised deformation operator $D_N$ is elliptic. It is then clear that it will remain elliptic for small $C^1$-perturbations of $N$ and smooth perturbations of $\Phi$. For $\ep > 0$ denote by $\nu_\ep(N)$ the $\ep$-neighbourhood of the zero section in the normal bundle $\nu(N) = TM|_N / TN$, as measured by $g_{\Phi}$. For $\ep > 0$ sufficiently small its image under the exponential map corresponding to $\Phi_s$ will be a tubular neighbourhood of $N$, for any $s \in \cS$, after potentially restricting $\cS$ to a neighbourhood to $s_0$. Let $C^\infty(\nu_\ep (N)) \subset C^\infty(\nu (N))$ denote the subset of sections which take value in $\nu_\ep(N)$. Consider now the perturbation $N_v = \exp_v(N)$ of this compact almost Cayley. Its failure to be Cayley is measured by the deformation operator \eqref{3_2_def_op}, and is a section of $E_{\cay}$. We now mildly extend the results from the previous section for the compact case: 

\begin{prop}
\label{3_4_cpt_cayley_linearisisation}
Let $N$ be an almost Cayley submanifold of $(M, \Phi)$ almost $\Spin(7)$. Let $\{\Phi_s\}_{s \in \cS} \subset \mathscr{A}(M)$ be a smooth finite dimensional family of $\Spin(7)$ structures such that $\Phi = \Phi_{s_0}$ for some $s_0 \in \cS$. Consider the map: 
\[\begin{array}{rl}
F: C^\infty (\nu_\ep (N)) \times \cS & \longrightarrow C^\infty(E_{\cay}) \\
v &\longmapsto \pi_E(\star_N \exp^*_v(\tau_s|_{N_v})). 
\end{array}
\]
Here $\exp, \star$ and $\pi_E$ are induced from $\Phi$. After shrinking $\cS$, there is a constant $C > 0$ which depends on $M, \Phi$ and on the injectivity radius of $N$, such that if furthermore $\nm{v}_{C^1} < C$, then $N_v$ is Cayley for $\Phi_s$ exactly when $F(v, s) = 0$. Moreover, $F(\cdot, s)$ is elliptic at the zero section for every $s \in \cS$.
\end{prop}
\begin{proof}
The only non-trivial issue is that for $s \ne s_0$, we used the exponential map, Hodge star and $\pi_E$ associated to $\Phi_{s_0}$, not $\Phi_s$.  However by shrinking $\cS$ we can insure that both ellipticity and the Cayley detecting property will still be maintained, by the compactness of $N$, and the openness of these conditions in the space of all smooth $\Spin(7)$-structures. 
\end{proof}
This is a generalisation of Proposition 2.3 from Kim Moore's paper \cite{mooreDeformationsConicallySingular2019}, where a more direct proof can be found for the case where $N$ is Cayley and the $\Spin(7)$-structure is fixed. This in turn is based on the work by McLean in his foundational paper \cite{mcleanDeformationsCalibratedSubmanifolds1998}. We will now consider $F$ and $D = \D F(0, s_0)$ as maps between Banach manifolds. For this, we need some auxiliary results. 
\begin{lem}
\label{3_4_F_Q}
Let $N, M, \Phi, \{\Phi_s\}_{s \in \cS}, F$ be as in Proposition \ref{3_4_cpt_cayley_linearisisation}. The map $F$ has the following form, for $v \in C^\infty(\nu_\ep(N))$ and $p \in N$:
\eas
	F(v, s)(p) =& \boF(p, v(p), \nabla v(p), T_pN, s) \\
	=&  F(0, s)(p) + (D_sv)(p) + \boQ(p, v(p), \nabla v(p), T_p N, s).
\eas 
Here $D_s$ is the linearisation of $F(\cdot, s)$ at $0$ and $\boF,\boQ$ are smooth fibre-preserving maps: 
\eas
\boF, \boQ: TM_\eps \times_N (T^*M \otimes TM )_\eps \times \Cay_{\al}(M) \times \cS  \longra \boE_{\cay},
\eas 
where $\boE_{\cay} = \{(p, \pi, e): (p, \pi) \in \Cay_{\al}(M), e \in E_\pi\}$ and $\al$ is sufficiently large. Here we see both sides as fibre bundles over $\Cay_{\al}(M) \times \cS$. We define the map $Q: C^\infty(\nu_\ep(N)) \times \cS \longra C^\infty(E_{\cay})$ as $Q(v, s) = F(v, s)-D_sv$.
\end{lem}

Note that $Q$ is a map between function spaces, while $\boQ$ is a smooth map between manifolds of finite dimension, and the same naming convention is applied for $F$ and $\boF$. We also write $F_s(\cdot) = F (\cdot, s)$ to emphasise that $F_s$ is a differential operator depending on a parameter $s$, and similar for $Q$ and the smooth functions $\boF$ and $\boQ$.

\begin{proof}
The value of $\exp_v(p)$ is determined by $p\in N$ and $v(p)\in \nu_p(N)$ as a geodesic is uniquely determined by a starting point and an initial velocity vector. Similarly, $\D \exp_v(p)$ is a smooth function of $p, v(p)$ and $\nabla v(p) \in T^*_pN \ot \nu_p(N)$ since the first order deviation of two geodesics is determined by the first order deviation at a previous time. Finally $\exp^*_v(p)$ can be entirely determined form $p, v(p), \nabla v(p)$ and the tangent space $T_pN$. Thus $F$ itself is of the form $F(v, s)(p) = \boF(p, v(p), \nabla v(p), T_pN, s)$, as it is the pullback of a differential form (which depends on $s$) by $\exp_v$. Here $\bF$ is a smooth map which is independent of $N$. The smoothness of $\boF$ follows from the smoothness of $\exp_v$ in $v$. In particular, the same argument applies to the map $Q(v, s) = F(v, s)-D_s v$, since $(D_sv)(p)$ is a smooth map in $p, v(p), \nabla v(p)$ and $T_pN$ only, as it is a first order operator. 
\end{proof}

The name $\boQ$ is meant to suggest that the term $\boQ_{s, \pi}(p, x, y)$ contains all the quadratic and higher terms in the variables $x$ and $y$. Indeed, we clearly have $\boQ_{s, \pi}(p, 0, 0) = 0$, so no constant term. Let us denote by $\partial_x$ and $\partial_y$ respectively the partial derivatives with respect to $x$ and $y$. Note that this does not require choosing a connection, as $\boQ$ is a fibre-preserving map between subsets of metric vector bundles, and $x$ and $y$ are exactly the fibre coordinates. Let $v \in C^\infty(\nu_\eps(N))$ be such that $v(p) = x_0$ and $\nabla v(p) = 0$. Then: 
\eas
\partial_x \boQ_{s, \pi}(p, 0, 0)[x_0] &= \frac{\d}{\d t} \bigg\rvert_{t=0} \boQ_{s, T_pN}(p, tx_0, 0)  \\
&= \left(\frac{\d}{\d t} \bigg\rvert_{t=0} F_s(tv)-F_s(0)-tD_s v\right)(p) \\
&= (D_s v-D_s v)(p) = 0.
\eas 
An analogous derivation for the variable $y$ shows that $\boQ_{s, \pi}$ satisfies the following for $p \in M$:
\ea 
\label{3_4_Q_prop}
\boQ_{s, \pi}(p, 0, 0) = 0, \  \partial_x \boQ_{s, \pi}(p, 0, 0) = 0 \text{ and }  \partial_y \boQ_{s, \pi}(p, 0, 0) = 0.
\ea

From this we will now obtain bounds on $Q_s(v)-Q_s(w)$ which are formally similar to the bounds obtained for homogeneous quadratic polynomials on $\R^n$. If $q$ is such a polynomial, one can show that for a constant $C > 0$ there is an inequality of the form:
\eas
\md{q(x)-q(y)} \le C\md{x-y}(\md{x}+\md{y}).
\eas
In the following, we use the notation $\md{v}_{C^k} = \sum_{i=0}^k \md{\nabla^i v}$ for a pointwise norm of the derivative, and we also think of $TN$ as a section of $\La^4 T^*M|_N \rightarrow N$, so that $\md{TN}_{C^k}$ is well-defined. The analogous result for $Q_s$ is then the following:
\begin{lem}
\label{3_4_bound_Q}
There is an $\ep > 0$ which only depends on $\Phi_s$ for $s \in \cS$ such that for $k \ge 0$ and $v, w \in C^k(\nu_\ep(N))$ with $\nm{v}_{C^1}, \nm{w}_{C^1} < \ep$ we have the following inequality:
\ea
\label{3_4_bound_q_hard}
 \md{Q_s(v)-Q_s(w)}_{C^{k+1}} \le C\sum_{  
 \genfrac{}{}{0pt}{2}{i+\md{J}+r \le k+2}{0 \le r \le k}} \md{\nabla^i (v-w)} (\md{\nabla^J v}+\md{\nabla^J w})  \md{\nabla^r TN},  
\ea
where the summation is over a multi-index $I$, and $\nabla^I = \nabla^{I_1} \ot \dots \ot \nabla^{I_r}$. If we assume that $\md{v}_{{C^{k+1}}}, \md{w}_{{C^{k+1}}}$ are sufficietly small, then this can be simplified to yield: 
\ea
\label{3_4_bound_q_simple}
 \md{Q_s(v)-Q_s(w)}_{C^{k+1}} \le& C(1+ \md{TN}_{C^{k+1}}) \bigg(\md{v-w}_{C^{k+1}}(\md{v}_{{C^k}}+ \md{w}_{{C^k}}) \nonumber \\
 +& \md{v-w}_{C^{k}}(\md{v}_{{C^{k+1}}}+ \md{w}_{{C^{k+1}}}) \bigg).
\ea
\end{lem}
\begin{proof}
 Let $x \in T_pM$  and $y \in T_p^*M \ot T_pM$ be of sufficiently small norm. Then Taylor's theorem gives us uniformly in $s \in \cS$ and $\pi \in \Cay_p(M)_{\al}$:
\ea
\boQ_{s, \pi}(p, x, y) &= \boQ_{s, \pi}(p,0,0) + \partial_x \boQ_{s, \pi}(p,0,0)  x +  \partial_y \boQ_{s, \pi}(p,0,0)  y \nonumber \\
&+\boR_{1, s, \pi}(p,x,y) x \ot x   + \boR_{2, s, \pi}(p,x,y) x \ot y \label{3_4_Taylor_Q} \\
&+\boR_{3, s, \pi}(p,x,y) y \ot y ,\nonumber
\ea
where the $\boR_{i, s, \pi}$ are smooth remainder terms. If we now consider small $x_1,x_2 \in T_pM$ and $y_1, y_2 \in T^*_pM \ot T_pM$ and use the properties $\eqref{3_4_Q_prop}$ of $\boQ$, we see that: 
\ea
\label{3_4_Q_diff}
	\boQ(p, x_2, y_2)-\boQ(p,x_1,y_1) &= \boR_1 x_2 \ot x_2   + \boR_2 x_2 \ot y_2+  \boR_3 y_2 \ot y_2  \nonumber \\ 
	&- \boR_1 x_1 \ot x_1   - \boR_2 x_1 \ot y_1 - \boR_3 y_1 \ot y_1 ,
\ea
Consider for the moment only the difference $\boR_1(p,x_2,y_2, s) x_2 \ot x_2 - \boR_1(p,x_1,y_1, s) x_1 \ot x_1$. We can rearrange and apply Taylor's theorem to give: 
\eas
\boR_1&(p,x_2,y_2, s) x_2 \ot x_2 - \boR_1(p,x_1,y_1, s) x_1 \ot x_1 \\
&=  \boR_1(p,x_2,y_2, s)(x_2 \ot x_2 - x_1 \ot x_1) + (\boR_1(p,x_2,y_2, s)-\boR_1(p,x_1,y_1, s)) x_1 \ot x_1 \\
&= \boR_1(p,x_2,y_2, s)(x_2 \ot (x_2-x_1) + (x_2-x_1) \ot x_1) \\
&+(\partial_x \boR_1(p,x_1, y_1, s)(x_2-x_1) +\partial_y \boR_1(p,x_1, y_1, s) (y_2-y_1))\ot(x_1\ot x_1).
\eas
 Now since $M$ is compact, after shrinking $\cS$, the subset $\bV_c$ of $(p,x,y, \pi s) \in TM \times (T^*M \ot TM) \times \Cay_{\al}(M) \times\cS$ such that $\md{x}, \md{y}, \md{\tau(\pi)} \le c$ for a fixed constant $c \in \R$ is also compact. As $\boQ$ and the $\boR_i$ are smooth, we can thus bound the norm of any derivative of fixed degree over such a subset. This gives us the following point-wise estimate, provided that $x_1, x_2, y_1$ and $y_2$ all have sufficiently small norms.
\eas
 \md{&\boR_1(p,x_2,y_2, s) x_2 \ot x_2 - \boR_1(p,x_1,y_1, s) x_1 \ot x_1} \\
 &\le C ((\md{x_1} + \md{x_2})\md{x_2-x_1} + (\md{x_2-x_1} + \md{y_2-y_1})\md{x_1}^2 \\
 &\le C (\md{x_2-x_1} + \md{y_2-y_1})(\md{x_1}+\md{x_2}+\md{y_1}+\md{y_2}). 
\eas
Here the constant $C$ is independent of $p \in M$ and $\pi \in \Cay_{\al}(M)$.
We can bound the rest of  equation \eqref{3_4_Q_diff} by the same expression, using similar arguments, i.e. there is a pointwise estimate: 
\e
\label{3_4_Q_ptw_bound}
 \md{\boQ_{s, \pi}(p, x_2, y_2)-\boQ_{s, \pi}(p,x_1,y_1)} \le C (\md{x_2-x_1} + \md{y_2-y_1})(\md{x_1}+\md{x_2}+\md{y_1}+\md{y_2}) .
\e
We will now adapt the above reasoning to obtain bounds on the covariant derivatives $\md{\nabla^k(Q_s(v)-Q_s(w))}$. For this consider $v \in C^\infty(\nu_\eps(N))$ and note that for a curve $\gamma : \R \ra M$ with $\gamma(0) = p \in N$ and $\gamma'(0) = \xi \in T_pN$:
\ea
\nabla_{\xi}(Q_s(v))(p) &=  \Pi_{TE \rightarrow E} \frac{\d}{\dt}\bigg \rvert \boQ_{s}(\ga(t), v(\ga(t)), \nabla v(\ga(t)), T_{\ga(t)}N)  \nonumber \\
\label{3_3_nabla_Q}
&= \partial_p \boQ_s [w] + \partial_x \boQ_s [\nabla_\xi v (p)] +  \partial_y \boQ_s [\nabla_\xi  \nabla v (p)] + \partial_\pi \boQ_s [\nabla_\xi TN].
\ea
Here $\Pi_{TE \ra E}$ is the connection map, which maps $T_{(p, Tp_N, s)} \bE \ra \bE_{(p, Tp_N, s)}$ as induced from the Levi-Civita connection on $N$, which only depends on $T_pN$ and not on the curvature of $N$. The derivative $\partial_p$ is given as: 
\eas
\partial_p \boQ [w] = \D\boQ(p, v, \nabla v, T_pN, s) [w, w^{h, TM}, w^{h,T^*M \ot T^M}, w^{h, TM}, 0].
\eas
Here $w^{h, E}$ means the horizontal lift of the vector $w$ to the corresponding bundle $E$. Finally, we consider $TN$ as a section of $\La^4T^*M$, so that $\nabla TN$ is well-defined. The conclusion is that the dependence of $\nabla(Q_s(v))(p)$ on $\nabla^2 v(p)$ and $\nabla TN(p)$ is affine, and the coefficients can be bounded on subsets of the form $\bV_c$. The same argument also applies to the $\boR_i$, and using Equation \ref{3_4_Q_diff} we can show that: 
\e
\label{3_4_nabla_Q_exp}
\nabla^k Q(v) = \sum_{\genfrac{}{}{0pt}{2}{\md{I}+j \le k+2 }{0 \le j \le k}}\boR^{I, j}(p, v, \nabla v, TN, s) \nabla^I v \ot \nabla^j TN,  
\e
where the $\boR^{I, j}$ are smooth maps on $\bV_c$, for sufficiently small $c$, and for $I = (i_1, \dots, i_r)$ a multi-index we set $\nabla^I v = \nabla^{i_1} v \ot \dots \ot \nabla^{i_l} v$. Notice that there are no products of the form $\nabla^{k+1} v \ot \nabla ^{k+1}v$ appearing. From this we can deduce the claimed bounds, since the $\boR^{I, j}$ are defined on compact sets. 
\end{proof}

\begin{cor}
\label{3_4_continuity_C_k}
The map $Q_s: C^\infty(\nu_\eps(N)) \rightarrow C^\infty(E_{\cay})$ is a continuous map of Fréchet manifolds, after restricting to an open neighbourhood of $0$. Similarly, the maps  $Q_s: C^{k+1}(\nu_\eps(N)) \rightarrow C^k(E_{\cay})$ and $Q_s: C^{k+1, \alpha}(\nu_\eps(N)) \rightarrow C^{k, \alpha}(E_{\cay})$ are continuous maps of Banach manifolds in the same way. 
\end{cor} 
\begin{proof}
If $v \ra w \in C^{k+1}(\nu_{\eps}(N))$, then by Lemma \ref{3_4_bound_Q} and by compactness of $N$: 
\ea
\nm{Q_s(v)-Q_s(w)}_{C^k} \le \tilde{C}\nm{v-w}_{C^{k+1}}(\nm{v}_{C^{k+1}}+\nm{w}_{C^{k+1}}) \longrightarrow 0. 
\ea
The proof for the Hölder case is identical, and the statement about $C^\infty$ is obtained by combining the statements for $C^k$ for all finite $k$.
\end{proof}

\begin{lem}
\label{3_4_l_p_bound}
Let $p > 4$ and $k \ge 1$. Then there is an $\ep > 0$ and $C > 0$ which depend on  $\Phi_s$ for $s \in \cS$ and $N$ such that for $v, w \in L^p_k(\nu_\ep(N))$ with $\nm{v}_{L^p_k}, \nm{w}_{L^p_k}< \ep$ we have the following inequality:
\e
 \nm{Q(v)-Q(w)}_{L^p_{k}} \le C \nm{v-w}_{L^p_{k+1}}( \nm{v}_{L^p_{k+1}}+ \nm{w}_{L^p_{k+1}}). 
\e
\end{lem}
\begin{proof}
As $k+1 \ge 2$ and $p > 4$, we have that $L^p_{k+1} \hookrightarrow C^k$ continuously by the Sobolev embedding theorem. Thus by making $\eps > 0$ small, we can make sure that the $C^k$ norms of $v$ and $w$ are arbitrarily small, say less than $\delta$. We then prove the $L^p_k$ estimate on $Q$ using the pointwise estimate \eqref{3_4_bound_q_simple} from Proposition \ref{3_4_bound_Q} as follows: 
\eas
\int_N & \md{\nabla^k Q(v)-\nabla^k Q(w)}^p \dvol \\ 
\le C &\int_N \md{v-w}^p_{C^{k+1}}(\md{v}_{C^k}+\md{w}_{C^k})^{kp}+ \\
&\md{v-w}^p_{C^k}(\md{v}_{C^{k}}+\md{w}_{C^k})^{(k-1)p}(\md{v}_{C^{k+1}}+\md{w}_{C^{k+1}})^p \dvol  \\
\le C &\delta^{(k-1)p}(\md{v}_{C^k}+\md{w}_{C^k})^{p} \int_N \md{v-w}^p_{C^{k+1}}\dvol \\
&+  C \delta^{(k-1)p}\md{v-w}^p_{C^k}\int_N(\md{v}_{C^{k+1}}^p+\md{w}_{C^{k+1}}^p) \dvol \\
 \le C&\nm{v-w}_{L^p_{k+1}}^p (\nm{v}_{L^p_{k+1}}^p + \nm{w}_{L^p_{k+1}}^p).
\eas
Here, we used Minkowski's inequality in the second inequality and the Sobolev embedding of $L^p_{k+1} \hookrightarrow C^k$ in the third. This is also were the dependence of the constant $C$ on $N$ appears. Note that the key fact used in deducing this $L^p_k$ bound was that there were no terms of the form $\nabla^{k+1} v \ot \nabla^{k+1}v$ in our expression for $\nabla^k Q$. In fact the mapping $v \rightarrow \nabla^{k+1} v \ot \nabla^{k+1}v$ is not bounded from $L^p_{k+1}$ to $L^p$, thus the presence of such a term would make it impossible to deduce a bound of the above form on Sobolev norms. 
\end{proof}

One can capture the dependence on the parameter $s \in \cS$ in a similar fashion. 

\begin{lem}
\label{3_4_dependence_s}
For any $s_0 \in \cS$ and sufficiently small $\eps > 0$, there is an open neighbourhood $U_{s_0} \subset \cS$ of $s_0$ and a constant $C(\cS) > 0$, such that for all $s \in U_{s_0}$ and $v \in C^\infty(\nu_\ep(N))$ with $\nm{v}_{C^k} < \eps$ we have: 
\e
	\md{F_s(v)-F_{s_0}(v)}_{C^k} \le Cd(s, s_0).
\e
\end{lem}
\begin{proof}
Using Taylor's theorem we get that:
\eas
\md{F_s(v)-F_{s_0}(v)}(p) \le 2\md{\partial_s \boF(p, v(p), \nabla v(p), T_pN, s_0)} d(s, s_0). 
\eas
Thus the case $k = 0$ follows from the same argument as we had before for the $v$-dependence. Higher derivatives follow analogously to what we had before as well. 
\end{proof}

\begin{prop}
\label{3_4_def_smooth}
Let $p> 4$ and $k \ge 1$. For sufficiently small $\ep$, the map $F$ from Propositions \ref{3_4_cpt_cayley_linearisisation} extends to a $C^\infty$ map between Banach manifolds:
\[ F: \mathcal{L}_\eps = \{ v\in L^p_{k+1} (\nu_\ep (N)), \nm{v}_{L^p_{k+1}} < \ep \} \times \cS \longrightarrow L^{p}_{k}(E_{\cay})\]
for $4 < p < \infty$. Its linearisation $(0, s_0)$ is Fredholm. 
\end{prop}
\begin{proof}
Notice that $F$ is a continuous or $C^k$ map between Banach spaces exactly when $Q$ is. This is because the constant term $F_s(0)$ is smooth in $s$, as is the linear term $D_s$. And both those terms are clearly smooth in $v$, as they are constant and linear respective. Continuity of $Q$ between Sobolev spaces follows from Proposition \ref{3_4_l_p_bound} and \ref{3_4_dependence_s} in exactly the same way that continuity between $C^k$-spaces was proven in Corollary \ref{3_4_continuity_C_k}. It remains to show differentiability. We see from an application of Taylor's theorem that for $v,w \in C^\infty(\nu(N))$, $s: \R \rightarrow \cS$ a smooth curve and $t \in \R$ sufficiently small:
\e
\label{3_4_Q_linearisation}
Q(v+tw, s(t))-Q(v, s(0)) = \partial_x \boQ [tw] + \partial_y \boQ [t\nabla w] +\partial_r \boQ[t\dot{s}(0)]+ O(t^2). 
\e
Let now  $v \in \mathcal{L}_\eps$ and $s \in \cS$. Thus by the Sobolev embedding theorem we have that $v \in C^k(\nu_\eps(N))$ has bounded $C^k$-norm. Define the operator:
\eas
L_{v, s} (w, \xi) = \partial_x \boQ [w] + \partial_y \boQ [\nabla w] + \partial_s \boQ[\xi].
\eas
The operator $L_{v,s}$ is first order with continuous coefficients, and as such is a bounded operator $L^p_{k+1}\times T_s\cS \ra L^p_k$. From \eqref{3_4_Q_linearisation} it is clear that $L_{v, s}$ is the Fréchet derivative of $Q$ at the point $(v, s)$. It remains to show that varying $(v, s)$ continuously in $L^p_{k+1} \times \cS$ entails a continuous variation of $L_{v, s}$ in the space of bounded operators $B(L^p_{k+1} \times T_s\cS, L^p_k)$. By computations analogous to the ones from Proposition \ref{3_4_bound_Q} on may obtain a bound of the form: 
\eas
\nm{(L_{v, s}-L_{v+tu, s(t)})[w, \xi]}_{L^p_{k}} \le C t (\nm{u}_{L^p_{k+1}} \nm{w}_{L^p_2}+ \md{\dot{s}(0)}\md{\xi}) ,
\eas
were $t$ is assumed sufficiently small, $v \in \mathcal{L}_\eps$, $s$ a smooth curve in $\cC$, $u,w \in L^p_{k+1}$ and $\xi \in T_{s(0)}\cS$ (we identify the tangent spaces $T_{s(t)}\cS$ via a fixed trivialisation). Here we crucially use the fact that $v \in \mathcal{L}_\eps$ have uniformly bounded $C^k$-norm. From this we see that: 
\ea
\nm{(L_{v, s}-L_{v+tu, s(t)})}_{\ope} \le Ct (\nm{u}_{L^p_{k+1}}+\md{\dot{s}(0)}),
\ea
which shows that the derivative $L_{v, s}$ varies continuously as $(v, s)$ varies continuously. Higher differentiability follows analogously, as Finally $L_{0, s_0} = D \op T$ by \eqref{3_4_Q_prop}, where $T: T_{s_0}\cS \rightarrow L^p_{k}(E_{\cay})$ so the derivative of $F$ at $(0, s_0)$ is the sum of an elliptic operator on a compact manifold and a bounded linear map, and as such Fredholm. 
\end{proof}

Solutions to the equation $F_s(v) = 0$ which are in $L^p_2$ will be automatically smooth by elliptic regularity. 

\begin{prop}
\label{3_4_elliptic_reg}
Any $v \in \mathcal{L}_\ep $ such that $F_s(v) = 0$ for some $s \in \cS$ is smooth. 
\end{prop}
\begin{proof}
Denote by $D_s$ the linearisation of $F_s$ at $0$. This is an elliptic operator with smooth coefficients. It thus admits a formal adjoint $D_s^*$. Since $F_s(v) = 0$, we of course also have that $D_s^*F(v) = 0$. From the Taylor expansion $F_s(v) = F_s(0)+D_sv + Q_s(v)$ and our expansion of $\nabla Q_s (v)$ from equation \eqref{3_4_nabla_Q_exp} we obtain that: 
\eas
D_s^*F(v) &= D_s^*D_s v + \tilde{S_s}(v, \nabla v) + \tilde{R_s}(v, \nabla v) \nabla^2 v \\
&=  R_s(v, \nabla v) \nabla^2 v +S_s (v, \nabla v).
\eas
Here $S, \tilde{S}, R, \tilde{R}$ are smooth in their arguments. For fixed $v \in \mathcal{L}_\eps$ define the linear differential operator $K_{v_s}$ as follows: 
\eas
K_v :\ & L^p_{k+1}(\nu(N)) \longra L^p_{k-1}(E_{\cay}) \\
& w \longmapsto R_s(v, \nabla v) \nabla^2 w.
\eas
As $v \in L^p_{k+1} \subset C^{k, \alpha}$ and $R$ is smooth in its arguments, this linear differential operator has coefficients in $C^{k-1,\alpha}$. It is elliptic, as $D_s^*F$ is elliptic at $0$. Thus $v$ is a solution to the following equation: 
\eas
K_{v, s}(v) = S_s(v),
\eas 
which is a second order elliptic equation with $C^{k-1,\alpha}$ coefficients. However $S_s(v, \nabla v)$ is actually in $C^{k,\alpha}$. Thus we can apply Schauder regularity results such as Theorem 1.4.2 in \cite{joyceRiemannianHolonomyGroups2007}, which allows us to improve the regularity of $v$ to $C^{k+2,\alpha}$. Consequently, the coefficients of $K_{v, s}$ will have regularity $C^{k+1,\alpha}$, as will the section $S_s(v)$. It follows by bootstrapping that $v \in C^\infty(\nu(N))$.
\end{proof}

We can now specialise to Cayley submanifolds, and describe their family moduli spaces locally. To be precise, we consider the following moduli space for $N \subset M$ an immersed submanifold and $\{\Phi_s\}_{s \in \cS}$ a smooth family of $\Spin(7)$-structures. 
\ea
\label{3_3_cmpt_moduli}
\cM(N, \cS) = \{ (\tilde{N}, s): \tilde{N}& \text{ is an immersed Cayley submanifold of } (M, \Phi_s) \nonumber\\
 &\text{ with } \tilde{N} \text{ isotopic to } N\}. 
\ea
We endow $\cM(N, \cS)$ with the $C^\infty$ topology. Note that if $N$ is Cayley, then the bundles $E_{\cay}$ and $E$ agree, as $TN = \cay_N$. Now we can apply the theory of Kuranishi models to the Fredholm map $F$ to arrive at the following structure theorem for its zero set.

\begin{thm}[Structure]
\label{3_4_structure}
Let $N$ be an immersed Cayley submanifold of $(M, \Phi_{s_0})$. Then there is a family of non-linear deformation operators $F_s$ which for $\ep> 0$ sufficiently small give a $C^1$ map: 
\eas
F: \mathcal{L}_\eps = \{ v\in L^p_{k+1} (\nu_\ep (N)), \nm{v}_{L^p_{k+1}} < \ep \} \times \cS \longrightarrow L^{p}_{k}(E).
\eas
Then a neighbourhood of $(N, s_0)$ in $\cM(N, \cS)$ is homeomorphic to the zero locus of $F$ near $(0, s_0)$. Furthermore we can define the \textbf{deformation space} $\cI(N, \cS) \subset C^\infty (\nu (N))$ to be the the kernel of $D_{N, s_0} = \D F (0, s_0)$, and the \textbf{obstruction space} $\cO(N, \cS) \subset C^\infty(E_{\cay})$ to be the cokernel of $D_{N, s_0}$. Then a neighbourhood of $(N, s_0)$ in $\cM(N, \cS)$ is also homeomorphic to the zero locus of a Kuranishi map: 
\eas
	\kappa:  \cI(N, \cS) \longra \cO(N, \cS).
\eas
In particular if $\cO(N, \cS) = \{0\}$ is trivial, $\cM(N, \cS) $ admits the structure of a $C^1$-manifold near $(N, s_0)$. We say that $N$ is \textbf{unobstructed} in this case.
\end{thm}
To conclude, we refer to the DPhil thesis of Robert Clancy \cite[Theorem 6.3.1]{clancySpinManifoldsCalibrated2012} for a proof of the following formula for the index of $D_{N, s_0}$. We simply add $\dim \cS$ since our deformation problem also allows for deformations of the $\Spin(7)$-structure. 

\begin{thm}[Index]
\label{3_4_index}
Suppose we are in the situation of Theorem \ref{3_4_structure}. The index of $D_{N, s_0}$ is given by: 
\e
\label{3_3_form_index}
\ind D_{N, s_0} = \frac{1}{2}(\si(N) + \chi(N)) - [N] \cdot [N] + \dim \cS.
\e
Here $\si(N)$ denotes the signature of $N$ as a compact oriented four-manifolds, $\chi(N)$ the Euler characteristic, and $[N]\cdot [N]$ the self-intersection number in $M$. This is the expected dimension of the moduli space $\cM(N, \cS)$. In particular if $N$ is unobstructed $\cM(N, \cS) $ will be a smooth manifold of this dimension. 
\end{thm}
\begin{ex}
\label{3_3_complex_cayley}
Suppose that $N \subset M$ is a complex submanifold in a CY4 manifold $(M, \om,g,J,\Om)$. Then $N$ will also be Cayley in the $\Spin(7)$-manifold $(M, \Re \Om + \ha\om \wedge \om)$. If $N$ is compact, then Hodge theory on the Kähler manifold $N$ allows us to conclude that $\ker \bar{\partial}+\bar{\partial}^* = \ker \bar{\partial}\op\bar{\partial}^*$, and so according to Example \ref{2_2_ex_cplx_sl}, infinitesimal Cayley deformations and infinitesimal complex deformations agree. Now, if we furthermore assume that $N$ is unobstructed as a complex submanifold, which means that $\Ho^1(N, \nu^{0,1}) = 0$, then it is also unobstructed as a Cayley manifold. This is because:
\eas
\dim _\R \mathcal{I}(N) &=\dim _\R \{\text{inf. complex deformations}\}\\
&=2\dim _\C \Ho^0(N, \nu)\\
&= \dim _\C \Ho^0(N, \nu)-\Ho^1(N, \nu)+\Ho^2(N, \nu)\\
&= \chi(N, \nu) = \ind (\bar{\partial}+\bar{\partial}^*)\\
&= \ind(\slashed{D}_N)\\
&= \dim _\R \mathcal{I}(N)-\dim _\R \mathcal{O}(N).
\eas
Thus the obstruction space $\mathcal{O}(N)$ vanishes. Here we used the fact that since $M$ is CY4 the canonical class on $N$ is given by $\det ( \nu)$, and we thus get by Serre duality $\Ho^0(N, \nu) \simeq  \Ho^2(N, \nu) $. Moreover, the index of $\bar{\partial}+\bar{\partial}^*$ is exactly the holomorphic Euler characteristic of the coefficient bundle. 
\end{ex}
\begin{ex}
\label{3_3_sl_cayley}
If $L \subset M$ is a compact special Lagrangian submanifold in a CY4 manifold $(M, \om,g,J,\Om)$ then it is Cayley in the $\Spin(7)$-manifold $(M, \Re \Om + \ha\om \wedge \om)$. For Lagrangian submanifolds the normal bundle is intrinsic, as $\nu(N) \simeq TN$. Thus the formula for the index \eqref{3_3_form_index} yields:
\eas
\ind \slashed{D}_N = \ha( \si(N) + \chi(N))- [N] \cdot [N] = \ha( \si(N) - \chi(N)) = b_1(N) - b_2^-(N)- 1.
\eas
Compare this to the special Lagrangian deformation theory as described in \cite{mcleanDeformationsCalibratedSubmanifolds1998}, where it is shown that the moduli space of special Lagranians isotopic to $N$ has dimension $b_1(N)$. Thus the obstruction space for compact Cayleys coming from special Lagrangians never vanishes, as the obstruction space necessarily has dimension:
\e
\dim \mathcal{O}(N) \ge b_2^-(N)+ 1.
\e
Looking at the explicit form for the Cayley operator in Proposition \ref{2_2_cayley_sl}, we see that in fact:
\eas
\Ker &\slashed{D}_N \simeq \mathcal {H}^1,\\
\Coker & \slashed{D}_N \simeq \mathcal {H}^0 \oplus \mathcal {H}^{2,-}.
\eas
Here ${H}^k$ is the space of harmonic $k$-forms on $N$, and ${H}^{2,-}$ is the space of harmonic anti-self-dual forms. On a compact manifold, if $\d ^- \si= 0$, then:
\eas
0 =\int_{\partial N }\si \wedge \d \si = \int_N \d \si \wedge \d \si =\int_N \d^+ \si \wedge \d^+ \si=\int_N \d^+ \si \wedge \star \d^+ \si = \nm{\d^+ \si}_{L_2}.
\eas
Now we obtain the first isomorphism by Hodge theory on $N$. As for the second isomorphism, note that the adjoint of $-\d \star \op \d ^-$ is exactly $-\d \star + \d ^ *: \Om^n \op \Om^{2,-} \longra \Om^1 $. Hodge theory again leads to the desired result.
\end{ex}

\subsection{Asymptotically conical case}

The deformation theory of noncompact Cayleys with conical ends is analogous to the compact case, however in order to stay in the Fredholm setting we need to consider a deformation map acting between weighted function spaces, as in Proposition \ref{2_3_d_dt}. 

Let $(\R^8,\Phi)$ be an $\AC_\eta$ almost $\Spin(7)$-manifold, and suppose that $A \subset \R^8$ is a an $\al$-Cayley submanifold that is $\AC_\la$ for some $\eta < \la < 1$. If $\mathcal{S}$ is a smooth family of $\AC_\eta$ perturbations of $\Phi =\Phi_{s_0}$, then we would like to examine the moduli space:
\ea
\label{3_4_ac_moduli}
\cM_\AC^\la(A, \mathcal{S}) = \{ (\tilde{A}, \Phi_s): \tilde{A}& \text{ is an } \AC_\la \text{ Cayley submanifold of } (\R^8, \Phi_s) \nonumber\\
 &\text{ isotopic to } A\text{ and asymptotic to the same cone} \}. 
\ea
For $\al$ sufficiently close to $1$, $A$ admits is a canonical deformation map, just as in the compact setting. However we need to modify the definition from section \ref{3_2_deformation_operator} to account for the $\AC_\la$ condition and to ensure Fredholmness of the linearised problem. Thus we define:
\ea
F_{\AC} :  C^\infty_{\la}(\nu_\ep(A)) \times \cS & \longrightarrow C^\infty_{\loc}(E_{\cay}),
\ea
for $\eps > 0$ sufficiently small, and after potentially shrinking $\cS$. Recall that $\nu_\ep(A)$ is a tubular neighbourhood that grows linearly in the distance from the origin, so any sufficiently small $\AC_\la$ deformation will be contained in this neighbourhood. We would like to show the following result:
\begin{prop}
\label{3_4deformation_map_ac}
Let $p > 4$ and $k \ge 1$. For sufficiently small $\ep > 0$ and $\eta < 0 $ the map $F_{\AC}$ extends to a $C^\infty$ map of Sobolev spaces:
\[ F_\AC: \cL_{\eps} = \{ v \in L^p_{k+1, \la} (\nu_\ep (A)): \nm{v}_{L^p_{k+1, \la}} \le \eps\}\times \cS\longrightarrow L^{p}_{k, \la-1}(E_{\cay}).\]
Furthermore, any $v\in L^p_{k+1, \la} (\nu_\ep (A)) $ such that $F_{\AC}(v) = 0$ is smooth and lies in $C^\infty_{\la+1}$. The linearisation at $0$ is the bounded linear map: 
\e
D_{\AC} : L^p_{k+1, \la} (\nu (A))\times T_{s_0}\cS \longrightarrow L^{p}_{k, \la-1}(E_{\cay}). \nonumber
\e
Finally, $D_{\AC}$ is Fredholm for $\la$ in the complement of a discrete set $\cD_L \subset \R$, which is determined by the metric structure on the asymptotic link $L \subset S^7$. 
\end{prop}
The proof of this result follows the same outline as in the compact case. The crucial step is to obtain estimates on the weighted $C^k_{\de}$ and $L^p_{k, \de}$ norms of the various terms involved in the Taylor expansion: 
\eas
F_{\AC}(v, s) = F_{\AC}(0, s) + D_{\AC,s}[v] + Q_{\AC}(v,s).
\eas
Moreover we need to investigate the dependence on the parameter $s \in \cS$. First let us examine the constant term. 
\begin{prop}
\label{3_4F_0_bound_ac}
Suppose that $A$ is $\AC_\la$ to a Cayley cone, $\al$-Cayley for $\al$ sufficiently close to $1$, and let $k \in \N$. Then, after shrinking $\cS$  there is a constant $C> 0$, independent of $s \in \cS$ such that $\nm{F_{\AC}(0,s)}_{C^k_{\la-1}}  < C$. Thus $F_{\AC}(0, s) \in C^{\infty}_{\la-1}(E_{\cay})$.
\end{prop}
\begin{proof}

We can think of $\tau$ at any given point $p \in \R^8$ as a smooth map: 
\eas
\tau_p : \La^4 \R^8 \longrightarrow \La^2_7 \R^8.
\eas
As $\Phi$ is $\AC_\eta$, this linear map approaches $\tau_0$ (corresponding to the standard $\Spin(7)$-structure) uniformly in $O(r^{\eta-1})$ and for all $s \in \cS$. Moreover, all derivatives up to a finite order can be bound by a constant independent of $s$. We now think of the tangent bundles of $A$ and $C$ as maps $C \longrightarrow \La^4 \R^8$, ignoring the compact region of $A$. The $\AC_\la$ condition on $A$ now gives us that for $p \in C$: 
\eas
\md{\nabla^i(T_pA- T_pC)} \le K_i r^{\la-1-i}, \text { as } r \ra \infty.
\eas
Here $\nabla$ is with respect to the cone metric on $C$ and the flat metric on $\R^8$. However the same is true for the metric induced from the embedding of $A$ in $(\R^8, \Phi_s)$ and the connection associated to $\Phi_s$. This is because changing the metric to an asymptotically conical one only introduces errors which are asymptotically smaller than the right hand side. Thus an application of Taylor's theorem leads to: 
\eas
\md{\nabla^k \tau(T_pA)} &\le \md{\nabla^k \tau(T_pC)}+\md{\nabla^k (\tau(T_pA)-\tau(T_pC))} \\
&\le C r^{\eta-k-1}+ \sum_{i+j = k} \md{D^i\tau}\md{\nabla^j(T_pA- T_pC)} \le C r^{\la-1-k}.
\eas
Here we used that if $\Phi$ is $\AC_\eta$ to $\Phi_0$ then $\md{\nabla^k \tau(T_pC)} < C r^{\eta-1-k}$, as $C$ is a $\Phi_0$-Cayley cone. The projection $\pi_E$ worsens this bound by a constant factor by an analogous argument.
\end{proof}

Next, let us look at the quadratic term. 

\begin{prop}
\label{3_4_Q_ac}
Suppose that $A$ is $\AC_\la$ to a Cayley cone with $\eta < \la <  1$ and $\al$-Cayley for $\al$ sufficiently close to $1$. Fix $k \in \N$ and assume that $\eta \le C_k$ for some universal constants $C_k \le 0$. Let $\eps > 0$ be sufficiently small. Suppose $u, v \in C^k_{\la}(\nu_\ep(A))$ satisfy $\md{u}_{C^1_{1}}, \md{v}_{C^1_{1}} \le \eps$. Then: 
\eas
\md{Q_{\AC}(u,s)-Q_{\AC}(v,s)}_{C^k_{\la-1}} &\lesssim \md{u-v}_{C^{k+1}_{\la}}\bigg(\md{u}_{C^{k}_{\la}}+\md{v}_{C^{k}_{\la}}\bigg) + \\
&\md{u-v}_{C^{k}_{\la}}\bigg(\md{u}_{C^{k+1}_{\la}}+\md{v}_{C^{k+1}_{\la}}\bigg)  .
\eas
The constant is independent of $p,u,v$ and $s$.
\end{prop}
\begin{proof}
We first consider the flat, translation-invariant $\Spin(7)$-structure $\Phi_0$. For this structure $\boQ_{\pi}(p, v, w)$ is independent of the point $p$. Furthermore, $\boQ$ is translation invariant in the following sense: for $\ga > 0$ we have:
\e
\label{3_4_scaling_invariance}
\boQ_{\pi}(\ga \cdot v, w) = \boQ_{\pi}( v, w).
\e
This is a reformulation of the fact that $Q_\AC(\ga \cdot v) = Q_\AC(v)$ after identifying $\R^8 \simeq T_p\R^8$ for each $p \in \R^8$. Recall the Taylor expansion \eqref{3_4_Taylor_Q} for small $v,w$:
\e
\label{3_4_taylor_AC}
\boQ_{\pi}( v, w) = \boR_{1, \pi} v \ot v + \boR_{2, \pi} v \ot w + \boR_{3, \pi} w \ot w.
\e
For $v$ outside the initial domain of definition, we can define:
\eas
\boR_{1, \pi}(\ga \cdot v, w) &=  \ga^{-2} \boR_{1, \pi}(v, w), \\ 
\boR_{2, \pi}(\ga \cdot v, w) &=  \ga^{-1} \boR_{2, \pi}(v, w), \\ 
\boR_{3, \pi}(\ga \cdot v, w) &=  \boR_{3, \pi}(v, w). 
\eas
The extended $\boR_{i,\pi}$ then also satisfy Equation \eqref{3_4_taylor_AC}. Near infinity, the derivatives have the following scaling behaviour: 
\e
\label{3_4_scaling_R}
\partial_x^k \partial_y^l \boR_{i, \pi} (\ga \cdot v, w ) = \ga^{-3+i-k}  \boR_{i, \pi} (v, w). 
\e
In particular for $v \in \nu_\eps(A)$ and $\md{w} \le \eps$ with $\eps$ sufficiently small, there are bounds $\md{\partial_x^k \partial_y^l \boR_{i, \pi} (v, w )} \le C_{k,l} \md{v}^{-3+i-k}$. From this we deduce for $u, v \in C^\infty_{\la} (\nu_\eps (A))$: 
\eas
\md{Q_\AC(u)-Q_\AC(v)}\rho^{-2 (\la-1)} &\le \md{\rho^{2}R_1(u)} \md{u-v}\rho^{-\la}(\md{u} + \md{v})\rho^{-\la} \\ 
&+ \md{\rho^{2+\la}\partial_x R_1(v)} \md{u-v}\rho^{-\la}(\md{v}\rho^{-\la})^2 \\ 
&+ \md{\rho^{1+\la}\partial_y R_1(v)} \md{\nabla u-\nabla v}\rho^{1-\la}(\md{v}\rho^{-\la})^2 + (\cdots) \\
&\le C(1+\rho^{\la+1})\md{u-v}_{C^1_{\la}}(\md{u}_{C^0_{\la}} + \md{v}_{C^0_{\la}}) + (\cdots) \\
&\le C\md{u-v}_{C^1_{\la}}(\md{u}_{C^0_{\la}} + \md{v}_{C^0_{\la}}) + (\cdots).
\eas
Here we used the fact that $\rho^{\la-1} \ra 0$ as $\rho \ra +\infty$. The terms containing $R_2$ and $R_3$ have been omitted as they admit analogous scaling behaviour. We ultimately obtain $\md{Q_\AC(u)-Q_\AC(v)}_{C^0_{2(\la-1)}} \le C \md{u-v}_{C^1_{\la}}(\md{u}_{C^1_{\la}} + \md{v}_{C^1_{\la}})$. For higher derivatives, note that the translation invariance of $\boQ$ gives us the following analogue of Equation \eqref{3_3_nabla_Q}:
\eas
\nabla_{\xi}(R_i(v)) = \partial_x \boQ [\nabla_\xi v] +  \partial_y \boQ [\nabla_\xi  \nabla v ] + \partial_\pi \boQ [\nabla_\xi TA].
\eas 
Now again from the Taylor expansion \eqref{3_4_Taylor_Q} we see that: 
\eas
\partial_x \boQ (v, \nabla v) = \partial_x \boR_1 v\ot v + (\boR_1 + \partial \boR_2) v \ot \nabla v + (\boR_2 + \partial \boR_3) \nabla v \ot \nabla v. 
\eas
All the terms have the same scaling behaviour, so that $\partial_x \boQ(\ga \cdot v, \nabla v) = \ga^{-1} \partial_x \boQ(v, \nabla v)$. The terms $\partial_y \boQ$ and $\partial_\pi \boQ$ can be treated in a similar way. The upshot is that one can express $\nabla^k Q_\AC(u)- \nabla^k Q_\AC(v)$ as a sum of terms which are products of $\partial_x^k \partial_y^l\partial_\pi^m \boR_i$,  $\nabla^i u - \nabla^i v$, $\nabla^j u + \nabla^j v$ and $\nabla ^ r TA$. Then manipulations as above allow us to conclude that:
\eas
\md{\nabla^k Q_\AC(u)- \nabla^k Q_\AC(v)} \rho^{k-2(\la-1)}& \le C \md{u-v}_{C^{k+1}_{\la}}\bigg(\md{u}_{C^{k}_{\la}}+\md{v}_{C^{k}_{\la}}\bigg) + \\
&\md{u-v}_{C^{k}_{\la}}\bigg(\md{u}_{C^{k+1}_{\la}}+\md{v}_{C^{k+1}_{\la}}\bigg) ,
\eas
from which the claim of the proposition follows, in the flat case.

Now, if $\Phi$ is an $\AC_\eta$ perturbation of $\Phi_0$, one has for $k,l\ge 0$ and $\md{v} \le \eps r$, $\md{w} \le \epsilon$:
\e
\label{3_4_nonflat_F_bound}
\md{\partial^k_x\partial^l_y (\boF_\Phi(p,v,w) - \boF_{\Phi_0}(p,v,w))} = O(\nm{\Phi - \Phi_0}_{C^{k+1}_{\eta}}\md{p}^{\eta + 1-k-l}).
\e 
This can be seen by first observing that $\md{\nabla^k (\exp_{\Phi} -\exp_{\Phi_0})} = O(r^{\eta-k})$. This in turn can be obtained by analysing the geodesic equation for the curve $x(t)$ :
\e
\ddot{x}_k = \Ga^k_{ij} \dot {x}_i \dot{x}_j,
\e
where $\Ga^k_{ij}$ are the Christoffel symbols for the usual coordinates on $\R^8$. The $\AC_\eta$ condition implies that $\md{\Ga^k_{ij}} \lesssim r^{\eta-2}$. Now, as $\dot {x}$ is a vector uniformly bounded with respect to both $g_\Phi$ and the flat metric, we can deduce that
\eas
\md{\exp_{p,\Phi}(v) - \exp_{p,\Phi_0}(v) }&=\md{x(t) - x(0)-t\dot{x}(0)}\\ &\le \int_0^t \int_0^t \md{\Ga^k_{ij}} \d s \d s'\\
&= O(t^2\nm{\Phi - \Phi_0}_{C^{1}_{\eta}} r^{\eta-2}) = O(\eps^2\nm{\Phi - \Phi_0}_{C^{1}_{\eta-1}} r^{\eta}).
\eas
One can write down similar ODEs for the variation of $\exp$ with regards to the initial condition and perform an analogous analysis to bound $\nabla^k (\exp_{\Phi} -\exp_{\Phi_0})$ for $k \ge 1$. As $\tau_p$ is obtained from $\Phi_p$ by a smooth mapping, it too is in $C^\infty_\eta$ with the same norm, up to a universal multiplicative constant. Thus we see that: 
\eas
\md{\boF_\Phi(p,v,w) - \boF_{\Phi_0}(p,v,w))} &\le \md{ \exp_{\Phi} -\exp_{\Phi_0}}+\md{\nabla (\exp_{\Phi} -\exp_{\Phi_0})}+\md{ \tau - \tau_0}\\ &= O(\nm{\Phi - \Phi_0}_{C^{1}_{\eta}}\md{p}^{\eta}).
\eas
The proof for higher derivatives works in a similar way. We can now use this to get decay estimates similar to Equation \eqref{3_4_scaling_R}. For any $\AC_\eta$  $\Spin(7)$-structure $\Phi$ we can define 
\eas
\boR_1(p,v,w) = \int_0^1 (1-t) \partial^2_x \boF(p, t(v,w))\d t, \\
\boR_2(p,v,w) = \int_0^1 (1-t) \partial^2_{x,y} \boF(p, t(v,w))\d t, \\
\boR_3(p,v,w) = \int_0^1 (1-t) \partial^2_y \boF(p, t(v,w))\d t.
\eas
From the bounds \eqref{3_4_nonflat_F_bound} we see that for $\ga \ge 1$ and $\eta \le 0$:
\eas
\md{\boR_{1, \Phi}(p,\ga \cdot v,w)} &\le \md{\boR_{1, \Phi_0}( p,\ga \cdot v,w)}+(\ga\md{p})^{\eta-2}\\
&\le\ga^{-2}\md{\boR_{1, \Phi_0}( p,v,w)}+(\ga\md{p})^{\eta-2}\\
&\le\ga^{-2}(\md{\boR_{1, \Phi}( p,v,w)}+\nm{\Phi - \Phi_0}\md{p}^{\eta-2})+\nm{\Phi - \Phi_0}(\ga\md{p})^{\eta-2}\\
&\le\ga^{-2}\md{\boR_{1, \Phi}( p,v,w)}+\nm{\Phi - \Phi_0}\ga^{-2}\md{p}^{\eta-2}(1+ \ga^{\eta}) \\
&\lesssim \ga^{-2}(\md{\boR_{1, \Phi}( p,v,w)}+\nm{\Phi - \Phi_0}).
\eas
Thus formally similar estimates to the flat case hold if we replace the equality by an inequality and introduce an error term which depends on the size of the perturbation. Note also that the constant introduced in the last step only depends on $\nm{\Phi - \Phi_0}$, thus can be bounded uniformly in $\cS$. In fact the estimates on $\boR_{2, \Phi}$,  $\boR_{3, \Phi}$  and all the derivatives follow in a similar way, given that $\eta$ is sufficiently negative. We can now conclude the proof like in the flat case.
\end{proof}

\begin{proof}[Proof of Theorem \ref{3_4deformation_map_ac}]
In order to prove that $F_\AC(\cdot, s) : L^p_{k+1, \la} \ra L^p_{k,\la-1}$ is $C^\infty$ for a fixed $\Spin(7)$-structure, we can repeat the proof for the compact case, using our estimates from Propositions \ref{3_4F_0_bound_ac} and \ref{3_4_Q_ac} as well as the fact that $D_\AC$ is an asymptotically conical operator, and therefore bounded between the Sobolev spaces in question. Indeed, it can be seen from the presentation \eqref{3_2_linearised_def_new} of $D_\AC$ that its coefficients and all derivatives approach the values for the conical operator on $C$. It now also follows from the Lockhart and McOwen theory that there is a discrete set $\cD_L \subset \R$ such that $D_{\AC}$ is Fredholm for $\la$ in the complement of $\cD_L$. 

For $v \in L^p_{\la}$ such that $F_\AC(v) = 0$, elliptic bootstrapping applies locally like in Proposition \ref{3_4_elliptic_reg}, so that such $v$ are immediately in $C^\infty_{loc}$. Now we can invoke the Sobolev embedding theorem \ref{2_3_Sobolev_embedding_hoelder_acyl_ac} to get that $v \in C^\infty_{\la}$. 

What remains to show is that $F_\AC$ is also smooth with respect to the parameter $s \in \cS$. Certainly the derivatives $\partial^k_s F_\AC (v,s)$ exist as smooth functions. The key issue is that they might not in $L^p_{\la-1}$ a priori. Note however that the perturbations in the $\Spin(7)$-structure induced by a change in $s$ lie in $C^\infty_\eta \subset L^p_{k,\la}$ for any $k$, as $\eta<\la$ by assumption. From this it can easily seen that $\partial_s F_\AC (v,s)$ will be in $L^p_{k,\la-1}$ as well, and the argument applies equally to higher derivatives.
\end{proof}

We can now prove the analogue of Theorem \ref{3_4_structure} for the asymptotically conical case.

\begin{thm}[Structure]
\label{3_4_structure_ac}
Let $A$ be an $\AC_\la$ Cayley submanifold of $(\R^8, \Phi_{0})$, and let $\cS$ be a family of $\AC_\eta$ deformations of $\Phi_0$ with $\eta<\la < 1$ . Then there is a non-linear deformation operator $F_{\AC}$ which for $\ep> 0$ sufficiently small is a $C^\infty$ map: 
\eas
F_{\AC}: \mathcal{L}_\eps = \{ v\in L^p_{k+1, \la} (\nu_\ep (A)), \nm{v}_{L^p_{k+1, \la}} < \ep \} \times \cS \longrightarrow L^{p}_{k, \la-1}(E).
\eas
Then a neighbourhood of $(A, \Phi)$ in $\cM^\la_{\AC}(A, \cS)$ is homeomorphic to the zero locus of $F_{\AC}$ near $0$. Assuming that $\la  \not \in \cD_L$ we define the \textbf{deformation space} $\cI^{\la}_{\AC}(A) \subset C^\infty_{\la} (\nu (A))\times T_\Phi \cS$ to be the the kernel of $D_{\AC} = \D F_{\AC}(0)$, and the \textbf{obstruction space} $\cO^\la_{\AC}(A) \subset C^\infty_{\la-4}(E)$ to be the cokernel of $D_{\AC}$. Then a neighbourhood of $A$ in $\cM^\la_{\AC}(A, \cS)$ is also homeomorphic to the zero locus of a Kuranishi map: 
\eas
	\kappa^\la_{\AC}:  \cI^\la_{\AC}(A) \longra \cO^\la_{\AC}(A).
\eas
In particular if $\cO^\la_{\AC}(A) = \{0\}$ is trivial, $\cM^\la_{\AC}(A, \cS) $ admits the structure of a smooth manifold near $A$. We say that $A$ is \textbf{unobstructed} in this case.
\end{thm}
The expected dimension of this moduli space can be expressed like in the compact case in Theorem \ref{3_4_index}, but we need to apply the more general Atiyah-Patodi-Singer theorem (cf. \cite{atiyahSpectralAsymmetryRiemannian1975}), which computes the index in terms of topological data, but also analytical data that depends on the cone. For $N$ a possibly non-compact $2n$-manifold, its signature $\si(N)$ is defined to be the signature of the non-degenerate pairing $H^n(N) \times H^n_{cs}(N) \ra \R$. Here $H^n_{cs}(N)$ denotes cohomology with compact support. This of course agrees with the usual definition of the signature for compact manifolds. For $A\subset \R^8$ an $\AC$ manifold asymptotic to the cone $C = \R_+ \times L$, we will consider the following version of intersection number. Pick a section $u \in  C^\infty (L, S\nu(A)|_L)$ in the sphere bundle of normal vector fields (to $A$) on $L$. For any  normal vector field $v \in C^\infty (A, \nu(A))$ that converges to $u$ at infinity, the algebraic count of its zeros will be only depend on the homotopy class of $u$ in the sphere bundle. We denote this number by $[A] \cdot_{[v]} [A]$.
\begin{prop}[Index]
\label{3_4_index_ac}
Let $A \subset \R^8$ be $\AC_\la$ with $\la < 1$, asymptotic to a Cayley cone $C = \R_+ \times L$, and $\al$-Cayley for $\al$ sufficiently close to $1$. Assume moreover that $(\la, 1) \cap \cD(L) = \emptyset$. Pick a homotopy class $[u] \in [L, S\nu(A)|_L]$ of a section $u: L\ra S\nu(A)|_L$. Then the following holds: 
\eas
\ind D_{\AC} &= \frac{1}{2}(\si(A) + \chi(A)) - [A] \cdot_{[u]} [A] + \eta(L) +  T([u]) - \dim \cM^{G_2}(L) .
\eas
Here $\eta(L)$ is a number that depends on the Riemannian manifold $L$, and $T([u])$ is a topological term depending on the homotopy class of $u$. 
\end{prop}
\begin{proof}
This is a consequence  Atiyah-Patodi-Singer theorem.  The term $\dim \cM^{G_2}(L)$ appears, since we are considering an operator on the weighted space $L^p_{k+1, \la}$ with $\la < 1$ instead of $L^p_{k+1, 1}$. The difference in index can be computed using Theorem \ref{2_3_change_of_index}, noting that $d(1) = \dim \cM^{G_2}(L)$.
\end{proof}

\begin{rem}
\label{3_4_quadratic_cone}
Consider the complex fibration:
\eas
f: \C^4 &\longrightarrow \C^2 \\
(x,y,z,u) &\longmapsto (x^2 + y^2 + z^2, u).
\eas
Its singular fibres are cones of the form $C_q = f^{-1}(0, 0) = \{x^2 + y^2 + z^2 = 0, u = 0\} \subset \C^4$. For each $\eps \in \C\setminus \{0\}$ we get a complex surface $A_\eps = f(\eps, 0)$. We can write a compact subset of $A_\eps$ as the image of a normal section as follows:
\eas
C_q &\longrightarrow A_\eps \\
(p, u) &\longmapsto \bigg(p + \frac{\eps \bar{p}}{2\md{p}^2}, u\bigg).
\eas
Here $p = (x,y,z)$ and $\nu_{(p, 0)} (A_\eps) = \spn_\C \{\bar{p}\}$. From this we see that $A_\eps$ is $\AC_{-1}$ to the cone $C_q$. It can be shown that for $\de > 0$ small: 
\eas
\cM^{-1+\de}_\AC(A_\eps) &\simeq \C\setminus \{0 \} \\
[A_\eps] &\longmapsto \eps,
\eas
and that all $A \in \cM^{-1+\de}_\AC(A_\eps)$ are unobstructed. In particular:
\eas
\ind D_\AC = \dim \cM^{-1+\de}_\AC(A_\eps)  = 2.
\eas
The critical rates $\cD_L \cap [-1, 1]$ for this cone were determined in Example \ref{2_3_quadratic_cone} to be $\{-1, 0, 1\}$. As $d(-1) = 2$, this means that there are no deformations for rates below $-2$. The next critical rate above $-1$ is $0$, which corresponds to translations. Thus: 
\eas
\dim \cM^{\de}_\AC(A_\eps)  = 2 + 8.
\eas
Finally, the remaining critical rate is $1$, which corresponds to rotations and deformations of the link as an associative, but which our theory does not take into account so far, as it would correspond to $\la = 1$. 
\end{rem}

\begin{cor}
\label{3_4_index_AC_SL}
Suppose that $A \subset \C^4$ is a special Lagrangian $\AC_\la$ submanifold for $\la < 1$. With the notation from Proposition \ref{3_4_index_ac}, we have that the Cayley deformation operator of $A$ has index:
\e
\label{3_4_index_AC_SL_formula}
\ind D_{\AC} = \ha(\si(A)-\chi(A)) + \tilde{\eta}(L) - \dim \cM^{G_2}(L).
\e
\end{cor}
\begin{proof}
There is a distinguished section of the normal bundle of a special Lagrangian cone in $(\C^4, \om, \Om, J)$, owing to the isomorphism 
\eas
\sharp: TA &\longra \nu(A) \\
v &\longmapsto J(v). 
\eas
The outward pointing radial vector field $\partial_r$ is tangent to the cone, and so $J(\partial_r)$ is normal. The Poincaré-Hopf theorem then allow us to equate $[A] \cdot_{[J(\partial_r)|_L]} [A] = \chi(A)$, from which the formula follows. 
\end{proof}
\begin{ex}[Cayley plane]
\label{3_4_plane_ex}
Consider the case of a Cayley plane $\Pi \subset \R^8$ as an $\AC_\la$ manifold of rate $\la < 1 $. It can also be seen as a special Lagrangian plane, by choosing an appropriate Calabi-Yau structure on $\R^8$. Up to translations, Cayley planes are rigid and unobstructed. This can be seen by solving the infinitesimal deformation equation from Proposition \ref{2_2_cayley_sl} explicitly. It is given by:
\eas
D_\AC = -\d\star \op \d ^-: \Om^1_{\la} \longra \Om^4_{\la-1} \op \Om^{2,-}_{\la-1}.
\eas
If $\si \in \Ker D_\AC$, we get that $\d \si \in \Om^{2,+}$, and so we can deduce:
\eas
\d^* \d \si = \star \d \star \d \si = \star \d \d \si = 0.
\eas
Together with $\d \star \si = 0$, we get $\De \si = 0$. As we are in flat $\R^4$, we see that $\si = \sum_{i=1}^4 f_i \d x_i$ with $f_i$ harmonic functions that decay like $r^{\la}$. Thus each of the $f_i$ must be a constant and $\dim \Ker D_\AC = 4$. Now for the obstruction space $\cO^\la_\AC$, we can equivalently look at the kernel of the adjoint map:
\e
(-\d\star \op \d^-)^* = -\d \star + \d^*: \Om^n_{-4 - \la}\op\Om^{2,-}_{-4-\la} \longra \Om^{1}_{-5-\la}.
\e
If $f\in \Om^n_{-4 - \la}$ and $\eta \in \Om^{2,-}_{-4-\la}$ satisfy $\d ^* \eta = \d \star f$, then $\d^* \d \star f = 0$, i.e. $\star f$ is a harmonic function on $\Pi$. Similarly, using the anti-self-duality of $\eta$, we see that 
\eas
\d \d^* \eta = \d \d \star f = 0,\\
\d^* \d \eta = -\star \d \d^* \eta = 0.
\eas
Thus $\eta$ is an anti-self-dual harmonic two form. Now for $\la > 0$, we have that both $f$ and $\eta$ are in $L^2$. Thus by \cite[Example 0.15]{lockhartFredholmHodgeLiouville1987} we see that both must vanish, as $\Ho^0_{cs}(\Pi) = \Ho^{2}_{cs}(\Pi) = 0$. Hence for a round sphere $S^3 \subset S^7$  we find from Corollary \ref{3_4_index_AC_SL} that: 
\eas
4 = \dim\cM^\la_{\AC}(\Pi) =  \ha(\si(\Pi)-\chi(\Pi)) + \tilde{\eta}(S^3) - \dim \cM^{G_2}(S^3).
\eas
This implies that $\tilde{\eta}(S^3)- \dim \cM^{G_2}(S^3) = 4\ha$. 
\end{ex}

\begin{ex}[Lawlor neck]
\label{3_4_ex_lawlor}
Consider two distinct Lagrangian subspaces $\Pi_1$ and $\Pi_2$ that intersect transversly. They are given as:
\eas
\Pi_1 = \spn\{v_1, v_2, v_3, v_4\} \text{ and } \Pi_2 = \spn\{e^{i\th_1}v_1, e^{i\th_2}v_2, e^{i\th_3}v_3, e^{i\th_4} v_4\},
\eas
 with $\th_1 + \th_2 + \th_3+ \th_4 = k \pi$, $0 \le \th_i \le \pi$ and $k = 1, 2, 3$. Lawlor showed in \cite{lawlorAngleCriterion1989} that if $k = 1, 3$, then there is a one parameter family of $\AC_{-3}$ manifolds asymptotic to the cone $\Pi_1\cup\Pi_2$, the \textbf{Lawlor necks} $L_{t }$. They are all diffeomorphic to $S^3 \times \R$. We can now apply Corollary \ref{3_4_index_AC_SL} again to determine the expected dimension of their Cayley moduli space: 
\ea
\label{3_4_lawlor_cayley}
\dim\cM^\la_{\AC}(L_{t}) &= \ha(\si(L_{t})-\chi(L_{t})) + 2\tilde{\eta}(S^3) - 2\dim \cM^{G_2}(S^3) \nonumber \\
&= \ha(0-0) + 2 \cdot (4\ha) = 9.
\ea
This corresponds to translations in $\R^8$ and the rescaling action. Thus there are no additional infinitesimal strictly Cayley deformations of the Lawlor necks. Note that if $k = 2$, then there are no minimal desingularisations of $\Pi_1\cup\Pi_2$ \cite{lawlorAngleCriterion1989}, so in particular no Cayley desingularisations.
Next, notice that the same argument that allowed us to show unobstructedness of the Cayley plane in the previous Example \ref{3_4_plane_ex} also gives us unobstructedness for the Lawlor necks for small $\la < 1$, since $\Ho^0_{cs}(L_t) = \Ho^{2}_{cs}(L_t) = 0$.
\end{ex}

Note that the Lawlor necks can priori only desingularise the union of two special Lagrangian planes. It turns out however that any pair of transversely intersecting Cayley planes can be realised as a pair of special Lagrangian planes for a suitably chosen $\SU(4)$ structure.
\begin{prop}
Let $\Pi_1, \Pi_2 \in \Cay(\R^8, \Phi_0)$ be two transversly intersecting Cayley subspaces. Then there is an $\SU(4)$-structure $(J, g, \om,\Om)$ on $\R^8$, such that both the $\Pi_i$ are special Lagrangian with respect to it.
\end{prop}
\begin{proof}
Recall from \cite[Thm. IV.1.8]{HarvLaws} that $\Spin(7)$ acts transitively on Cayley planes with stabilizer $H \simeq (\Sp(1) \times \Sp(1) \times \Sp(1))/ \pm \id$. We now show that generically the action of $\Spin(7)$ on pair of Cayley planes is free. First note that $\SO(8)$ acts transitively on pairs of four-planes with fixed characterizing angles $0 < \beta_1 \le \beta_2 \le \beta_3 \le \beta_4 \le \frac{\pi}{2}$, as explained in \cite[Section 4] {lawlorAngleCriterion1989}. Now $\dim \SO(8) = 2\dim \Gr(4,8) - 4$, and so this action can at most admit discrete stabilisers for generic choices of the $\beta_i$. More precisely, this is the case when the $\be_i$ are pairwise different and not equal to $\frac{\pi}{2}$, in which case the stabilisers are trivial. In these cases the action on Cayley plane pairs is free as well. Now note that for Cayley planes, the angle $\be_4$ can be derived from the other three, by using the triple product. Thus generically a family of Cayley plane pairs with fixed angles is $2\dim \Cay -3 = 21 = \dim \Spin(7)$ dimensional. In particular, since both $\Spin(7)$ and Cayley plane pairs of given angles are connected the action of $\Spin(7)$ is transitive when restricted to pairs with fixed angle. Now for a fixed $\SU(4)$-structure, the set of special Lagrangian plane pairs contains examples for all possible characteristic angles that can appear for Cayleys. Thus we can always $\Spin(7)$ rotate a pair of generic Cayley planes to a pair of special Lagrangian two planes. To conclude, notice that the continuous action of $\Spin(7)$ on transversely intersecting special Langrangian plane pairs sweeps out a closed and dense subset of the transversely intersecting Cayley plane pair. Hence it must reach them all, and this proves the proposition.
\end{proof}

\begin{lem}
\label{3_4_cayley_lawlor}
Suppose that $\Pi_1, \Pi_2 \in \Cay(\R^8)$ are two Cayley planes that intersect negatively in a single point. Then there is a one-dimensional family of unobstructed $\AC_{-\frac{1}{2}}$ Cayley submanifolds, the \textbf{Cayley-Lawlor necks} which are asymptotic to the cone $\Pi_1 \cup \Pi_2$. 
\end{lem}
\begin{proof}
Two transversly intersecting special Lagrangian planes in $\C^4$ can be $\SU(4)$-rotated to be of the form $\spn \{ v_1, v_2, v_3, v_4\}$ and $\spn \{ e^{i\th_1}v_1, e^{i\th_2}v_2, e^{i\th_3}v_3, e^{i\th_4}v_4\}$ where $\th_1 + \th_2 + \th_3 + \th_4 = k\pi$ where $k = 1, 2, 3$. Now by the previous proposition, the same is true for Cayley planes. It now can be shown that the planes intersect negatively exactly when $k = 1, 3$, which are exactly the cases where Lawlor showed the existence of Lawlor necks. 
\end{proof}

Suppose that $0 < \la < 1$ is such that $(0, \la ) \cap \cD_L = \emptyset$. For any $\la < \tilde{\la} < 1$ we will then have the an isomorphism $\cM^{\la}_\AC \simeq \cM^{\tilde{\la}}_\AC$, as no deformations additional deformations appear for these rates. At $\tilde{\la} = 1$, which our theory does not cover at the moment, more deformations appear. To be precise they come from perturbing the link $L$ as an associative submanifold in the round $S^7$ with its nearly parallel $G_2$-structure. Denote their moduli space by $\cM^{G_2}(L)$.
We always assume that it is a smooth, finite-dimensional moduli space and of the expected dimension (at least in a neighbourhood $\cO$ of $L \in \cM^{G_2}(L)$), so that the family of deformations of the cone is smooth, and the cone is unobstructed. There then exists a smooth family $ \{A_f\}_{f \in \cO} $ of $\AC_\la$-manifolds, such that $A_f$ has link $f \in \cO \subset \cM^{G_2}(L)$. Such a family can be obtained by finding ambient isotopies which perturb the cones in the desired fashion (which we do in more detail in Proposition \ref{3_4ambiant_isotopy}). We obtain a smooth family of maps $\exp_{f, v}: \nu_\eps(A) \ra \R^8$ which form tubular neighbourhoods of the family $A_f$, all parametrised by the normal bundle of our initial $A$. We would like to study the moduli space: 
\eas
\cM^1_{\AC}(A) =  \bigsqcup_{L \in \cM^{G_2}(L)}\cM^\la_{\AC}(A_f).
\eas
This is independent of our choice of $\la$ as long as $(\la, 1) \cap \cD_L = \emptyset$. We define the following deformation operator: 
\ea
F_{\AC, 1} :  C^\infty_{\la}(\nu_\ep(N)) \times  \cO  &\longrightarrow C^\infty_{\loc}(E_{\cay}), \\
 (v, f) &\longmapsto \pi_E \star_4 \exp_f^*(\tau_{\Phi}|_{\exp_{f, v}(N)}). \nonumber
\ea
We can now give this operator the same treatment as $F_{\AC}$, with some mild modifications to make sure that we get a smooth map of Banach manifolds which is also smooth in the parameter variable $f \in \cO$. This is identical to the conically singular case, which we will treat in great detail in the next section. The upshot is the following theorem: 

\begin{thm}[Structure]
\label{3_4_structure_ac_0}
Let $A$ be an $\AC_\la$ Cayley submanifold of $(\R^8, \Phi_{0})$ with link $L$ which is unobstructed in its moduli space $\cM^{G_2}(L)$. Then there is a non-linear deformation operator $F_{\AC, 1}$ which for $\ep> 0$ sufficiently small is a $C^\infty$ map: 
\eas
F_{\AC, 1}: \mathcal{L}_\eps = \{ v\in L^p_{k+1, \la} (\nu_\ep (A)), \nm{v}_{L^p_{k+1, \la}} < \ep \} \times \cO \longrightarrow L^{p}_{k, \la-1}(E).
\eas
Here $\cO$ is a sufficiently small neighbourhood of $L \in \cM^{G_2}(L)$. Then a neighbourhood of $A$ in $\cM^1_{\AC}(A)$ is homeomorphic to the zero locus of $F_{\AC, 1}$ near $0$. Assuming that $\la \not \in \cD_L$ we define the \textbf{deformation space} $\cI^{\la}_{\AC}(A) \subset C^\infty_{\la} (\nu (A))$ to be the the kernel of $D_{\AC, 1} = \D F_{\AC, 1}(0)$, and the \textbf{obstruction space} $\cO^\la_{\AC}(A) \subset C^\infty_{4-\la}(E)$ to be the cokernel of $D_{\AC, 1}$. Then a neighbourhood of $A$ in $\cM^0_{\AC}(A)$ is also homeomorphic to the zero locus of a Kuranishi map: 
\eas
	\kappa^\la_{\AC, 0}:  \cI^\la_{\AC}(A)\op T_L\cO  \longra \cO^\la_{\AC}(A).
\eas
In particular if $\cO^\la_{\AC}(A) = \{0\}$ is trivial, $\cM^1_{\AC}(A) $ admits the structure of a smooth manifold near $A$. We say that $A$ is \textbf{unobstructed} in this case. The index of $D_{\AC, 1}$ is given as: 
\eas
\ind D_{\AC, 1} = \ind D_{\AC} + \dim \cM^{G_2}(L). 
\eas
Finally, the map  $\cM^1_{\AC}(A) \longrightarrow \cM^{G_2}(L)$ sending a manifold to its asymptotic link is a smooth fibre bundle.  
\end{thm}
\begin{rem}
Looking again at the fibration from remark \ref{3_4_quadratic_cone}, we note that the cone $C_q = L_q \times \R_+$ is in fact part of a two-dimensional family of Cayley cones, up to the action of $\Spin(7) $: 
\eas
\cC = \{C_{\bar{a}}: \{a_1 x^2 + a_2 y^2+ a_3 z^2 = 0, w = 0 \} | a_1 + a_2 + a_3 = 1, a_i \in \R_+\}.
\eas
A generic cone in this family has stabiliser $\{(e^{it}, e^{it}, e^{it}, e^{-3it}), t\in \R\} \subset \Spin(7)$, and so we have: 
\e
\dim \cM^{G_2}(L_q) = 21 + 2 -1 = 22. 
\e 
This corresponds exactly to the expected dimension $d(1)$, and so the link $L_q$ is unobstructed. Thus $\cM^{1}_\AC(A_\eps)$ is a smooth manifold of dimension $32$ near $A_\eps$.
\end{rem}

To conclude this section, we describe a natural completion of the moduli space $\cM^{\la}_\AC(A)$ obtained by adjoining the cone. 
\begin{dfn}
The \textbf{completed moduli space} $\overline{{\cM}}^{\la}_{\AC}(A)$ is the topological space $\cM^{\la}_\AC(A) \cup \{C\}$ such that the rescaling map  $A \mapsto t \cdot A$ where $0 \cdot A = C$ is continuous, and $\cM^{\la}_\AC(A)$ embeds homeomorphically into $\overline{{\cM}}^{\la}_{\AC}(A)$. 
\end{dfn}
We have for example that $\overline{\cM}^{\la}_\AC(A_\eps)$ from Remark \ref{3_4_quadratic_cone} is homeomorphic to $\C$. Note that the scaling action $A \mapsto t \cdot A$ acts as $\eps \mapsto t^2\epsilon$ in this picture. There is a notion of scale implicit in this description of the moduli space, which we now make precise. Suppose for this that every $ A \in \cM^{\la}_\AC(\tilde{A})$ is unobstructed an that we have chosen a smooth cross-section $S \subset \cM^{\la}_\AC(\tilde{A})$ of the scaling action. 
\begin{dfn}
The \textbf{scale} of  $A \in \overline{\cM}^{\la}_\AC(\tilde{A})$ with respect to the cross-section $S$ is:
\e
t(\ga A, S) = \ga,
\e
where $A \in S$. Note that the scale functions corresponding to different cross-sections are all uniformly equivalent. 
\end{dfn}

\subsection{Conically singular case}

Let $N \subset (M, \Phi)$ be a $\CS_{\bar{\mu}}$ Cayley submanifold with conical singularities at $\{z_1,\dots, z_l\}$ and rates $\bar{\mu} = (\mu_1, \dots, \mu_l)$, where the $\mu_i \in (1, 2)$ for $1 \le i \le l$. Fix a $\Spin(7)$-parametrisation $\chi_i$ around the singular point $z_i$. With regards to the parametrisation $\chi_i$, let $N$ be asymptotic to the cone $C_i \subset \R^8$. Thus $N$ decays to the cone $C_i$ in $O(r^{\mu_i})$ as the distance $r$ to the singular point goes to $0$. We denote the link of the cone $C_i$ by $L_i \subset S^7$. In the deformations of $N$ that we will consider, we will allow the singular points and the asymptotic cones to move via translations, rotations and associative deformations of the link. Furthermore we will allow the $\Spin(7)$-structure to vary in a family $\Phi \in \{\Phi_s\}_{s \in \cS}$. Thus we will study the following moduli space: 
\begin{align*}
\cM_{\CS}^{\bar{\mu}}(N, \cS) = \{ (\tilde{N}, s): & \tilde{N} \subset (M, \Phi_s) \text{ is a } \CS_{\bar{\mu}}\text{-Cayley with singularities } \tilde{z}_1, \dots, \tilde{z}_l \\ &\text{ and cones } \tilde{C}_1, \dots, \tilde{C}_l, 
\text{. Here } \tilde{N} \text{ is isotopic to } N \text{, where }\\ &\text{ the isotopy takes } z_i \text{ to } \tilde{z}_i   \text{, and  } \tilde{C}_i \text{ is a deformation of } C_i \}.
\end{align*}

Locally around the fixed Cayley $N$ this moduli space will be given as a zero set of a nonlinear operator between suitable Banach manifolds. This is an extension of the work done in \cite{mooreDeformationTheoryCayley2017}, where the deformations are required to fix the cones. In order to define the nonlinear operator, we first define the configuration space of small deformations of the tuple $(C_1, \dots, C_l)$. Let $U_i$ be an open neighbourhood of $z_i \in M$ and let $G_i \subset \Spin(7)$ be the stabiliser of the cone $C_i$, which we also assume is the stabiliser of any deformation of $C_i$. The configuration space is then given by:
\eas
 \cF = \prod_{i=1}^l\{(\tilde{L}_i, e_i, s) | e_i: \R^8 \ra T_{\tilde{z}_i}: e_i \Spin(7)\text{-frame for } \Phi_s, \tilde{z}_i \in U_i,\tilde{L}_i \in \cM^{G_2}(L_i) \} /G_i.
\eas

It is a $H = \prod_{i=1}^l \Spin(7)/G_i$-bundle over the spaces $\cU$ of possible vertex locations and cones for every member of $\Phi \in \{\Phi_s\}_{s \in \cS}$, i.e. $\cU =  \cS \times \prod_{i=1}^l (U_i \times \cM(C_i))$. Each element $(\bar{x}, \bar{L}, \bar{e}, s) \in \cF$ corresponds bijectively to a unique configuration of cones, since we took quotients by the stabilisers $G_i$. We now fix a reference $\CS_{\bar{\mu}}$-fourfold for each point in a small neighbourhood of the asymptotic data of $N$, which is $f_0 = (z_1, \dots, z_l, L_1, \dots, L_l, \D \chi_1(0), \dots, \D \chi_l(0), s_0)$. 

\begin{prop}
\label{3_4ambiant_isotopy}
There is a smooth family $N_f$ of $\CS_{\bar{\mu}}$-manifolds parametrised by $f\in \cO \subset \cF$, where $\cO$ is an open neighbourhood of $f_0$ and such that $N_f$ has asymptotic data $f$. We can choose $N_{f_0} = N$.
\end{prop}
\begin{proof}
Without loss of generality we can restrict to the case of a single vertex, while only perturbing $N$ in an arbitrarily small neighbourhood of the vertex to obtain the desired family. Let $z_0 \in N$ be singular with cone $C = \R_+ \times L$ with regards to the $\Spin(7)$-parametrisation $\chi_0$. Consider diffeomorphisms of the unit ball in $\R^8$ by the action of $(A, v) \in \GL(8) \rtimes \R^8$, denoted by $\phi_{A,v}.$ They are isotopic to the identity and in fact can be extended to a smooth family (also denoted by $\phi_{A, v}$) of self-diffeomorphisms of $\R^8$ which leave everything outside of the ball with radius $2$ unchanged (see for instance the Homogeneity Lemma in chapter 4 of \cite{milnorTopologyDifferentiableViewpoint1997}). This family, scaled down sufficiently, can be applied in the chart given by $\chi_0$ to apply any desired small translation and rotation to the asymptotic cone, while only perturbing an arbitrarily small neighbourhood of the vertex. Finally, since any $\tilde{L} \in \cM^{G_2}(L)$ is smoothly isotopic to $L$, we can perturb any $\CS_{\mu}$ manifold asymptotic to $C$ to be asymptotic to $\R_+ \times \tilde{L}$ instead, with the same rate. 
\end{proof}

If we restrict the previous family to a sufficiently small neighbourhood $\cO$ of $f_0$, its members will be $\alpha$-Cayley for any desired $\alpha < 1$. Let now $\rho$ be a radius function for $N$. By Proposition \ref{2_3_tubular} we know that $\nu_\ep(N)$ maps onto a tubular neighbourhood $U_{f_0}$ of $N$ inside $M$, for $\ep > 0$ sufficiently small. Composing this open embedding with the ambient isotopy from \ref{3_4ambiant_isotopy} taking $N_{f_0}$ to $N_f$ gives tubular neighbourhood $U_f$ of $N_{f}$, for $f$ sufficiently close to $f_0$. We denote these maps by: 
\[\exp_f : \nu_{\ep}(N) \rightarrow U_f.\]
Furthermore, given a normal vector field $v \in C^{\infty}(\nu_\ep(N))$ we define the embedding $\exp_{f, v}: N \rightarrow M$ as the composition $\exp_{f,v} = \exp_f \circ v$. Thus varying $f$ will perturb the asymptotic cones, while changing $v$ alters the shape of the $\CS_{\bar{\mu}}$-manifold, keeping the cones fixed. The moduli space $\cM_{\CS}^{\bar{\mu}}(N)$ is given as the zero set of the following non-linear differential operator: 
\ea
\label{3_4_F_CS}
F_{\CS} :  C^\infty_{\bar{\mu}}(\nu_\ep(N)) \times \cO &\longrightarrow C^\infty_{\loc}(E_{\cay}) \nonumber \\
 (v, f = (\bar{x}, \bar{L},\bar{e}, s)) &\longmapsto \pi_E \star_4 \exp_f^*(\tau_{\Phi_s}|_{\exp_{f, v}(N)}).
\ea

We will now address the necessary modifications to the proofs for compact Cayleys so that they extend to the conically singular setting. First, let us define the correct Banach manifolds. Let, for $\eps > 0$: 
\e
	\cL_{\ep} = \{ v \in L^p_{k+1, \bar{\mu}}(\nu_\ep(N)), \nm{v}_{L^p_{k+1, \bar{\mu}}} < \ep \}.
\e
In fact $F_{\CS}$ is mapping $\cL_{\eps} \times \cO \longra L^p_{k, \bar{\mu}-1}(E_{\cay})$ for sufficiently small $\eps$. We will again prove boundedness separately for the constant, linear and quadratic and higher terms in the expansion:
\e
F_{\CS}(v, f) = F_{\CS}(0, f) + D_{\CS, f} v + Q_{\CS}(v, f).
\e
First, note that the closeness of $N_f$ to a Cayley cone gives a bound on $F_{\CS}(0, f)$, which is measure of the failure of $N_f$ to be Cayley. 

\begin{prop}
\label{3_4F_0_bound}
There is a constant $C_k > 0$ such that for any $f \in \cO$ in an open neighbourhood $f_0 \in U_k$, we have $\nm{F_{\CS}(0, f)}_{C^k_{\bar{\mu}-1}}  \le C_k$.
\end{prop}
\begin{proof}
Let $C_f \subset \R^8$ be the asymptotic cone of $N_f$ near a fixed singular point $z \in M$, where $N_f$ has decay rate $\mu$, and consider everything in a small ball $B_\eta(0) \subset \R^8$ via the parametrisation $\chi$. Let $\io_f$ be the embedding of the abstract cone $C$ as $C_f$ and let $\Theta_f $ be a parametrisation of the end of $N_f$ by $C$. For both of these we implicitly choose some identification of the potentially different links for varying $f$. Note that in this formulation the $\Spin(7)$-structure on $B_\eta(0)$ only needs to agree with $\Phi_0$ at the origin. The assignment $(r, p, f) \rightarrow \Th_f(r, p)$ is smooth, and thus we have from the $\CS_{\mu}$-condition:
\e
\label{_3_4_bound_on_f}
\md{\nabla^i(\Th_f(r, p)-\io_f(r,p))} \le K_{i,f} r^{\mu-i},
\e
where the constant $K_{i,f}$ is continuous in $f$. In particular, after shrinking $\cO$, we can replace $K_{i,f}$ by a single constant $K_i$. Consider $\tau$ now as a vector bundle morphism: 
\[ \tau: \La^4 B_\eta(0) \rightarrow E_{\cay}.\]
The (higher) covariant derivatives of $\tau$ can then be considered as maps: 
\[ \nabla^i\tau: \La^4 B_\eta(0) \ot (TM)^{\ot^i} \rightarrow E_{\cay}.\]
By the compactness of the base, any finite number of derivatives can be bounded by a constant. We can also consider $TN_f$ and $TC_f$ as maps $C \rightarrow \La^4\R^8$, by the embedding $\Gr(4, M) \ra \La^4M$. The $\CS_\mu$ condition \eqref{_3_4_bound_on_f} then translates to: 
\begin{align*}
\md{\nabla^i(T_p N_f- T_p C_f)} \le K_i r^{\mu-1-i}.
\end{align*}
By Taylor's theorem, $\tau$ has the following decay behaviour near $z \in N_f$:
\begin{align*}
\md{\nabla^i\tau_{r, p}(T_{r,p}C_f)} = \md{\nabla^i(\tau_{r, p}(T_{r,p}C_f)-\tau_0(T_{r,p}C_f))} \le 2r\md{\nabla^{i+1} \tau_0}  \lesssim r
\end{align*}
Now we see that:  
\begin{align*}
\md{\nabla^i \tau_{r, p}(T_{r,p}N_f)} &\lesssim r + \md{\nabla^i(\tau_{r, p}(T_{r,p}N_f)-\tau_{r, p}(T_{r,p}C_f))} \\
&\lesssim r + \sum_{a+b = i}\md{\nabla^{a} \tau_0} \md{\nabla^b(T_p N_f- T_p C_f)}  \lesssim r^{\mu-1-i}.
\end{align*}
Thus the result holds for a single singular point. The generalisation to multiple singular points is straightforward.
\end{proof}

Next, we turn our attention to the quadratic estimates:
\begin{prop}
\label{3_4_Q_cs}
Suppose that $N$ is $\CS_{\bar{\mu}}$ and $\al$-Cayley for $\al$ sufficiently large. Fix $k \in \N$ and let $u, v \in C^k_{\bar{\mu}}(\nu_\ep(N))$ with $\eps > 0$ sufficiently small and $\md{u}_{C^1}, \md{v}_{C^1} \le \eps$. Then there is an open neighbourhood $f\in U \subset \cO$ such that: 
\eas
\md{Q_{\CS}(u, f)-Q_{\CS}(v, f)}_{C^k_{\bar{\mu}-1}} &\lesssim \md{u-v}_{C^{k+1}_{\bar{\mu}}}\bigg(\md{u}_{C^{k}_{\bar{\mu}}}+\md{v}_{C^{k}_{\bar{\mu}}}\bigg) + \\
&\md{u-v}_{C^{k}_{\bar{\mu}}}\bigg(\md{u}_{C^{k+1}_{\bar{\mu}}}+\md{v}_{C^{k+1}_{\bar{\mu}}}\bigg) . 
\eas
Here the constant hidden in $\lesssim$ is independent of $f$.
\end{prop}
\begin{proof}
Without loss of generality, consider the case where $N$ has just one singular point $z$ with rate $\mu$. We then define the smooth function $\boQ(p, v, \nabla v, T_pN, s)$  as we did for the compact case in Lemma \ref{3_4_F_Q}. Now, even though $N$ is not compact, there are still bounds on all derivatives of $\boQ$ as in the compact case. From our assumptions on $u$ and $v$ we can ensure that $(u, \nabla u)$ vary in a compact set for all the sections in question and any point in $N \times \cS$. Thus we can prove the bound \eqref{3_4_bound_q_hard} for a constant independent of $(p,s) \in N \times \cS$ or the section in question. Thus we obtain: 
\eas
 &\md{\nabla^k(Q_{\CS}(u)-Q_{\CS}(v))}\rho^{-(k+2)(\mu-1) +k}  \\
 &\lesssim \sum_{\genfrac{}{}{0pt}{2}{i+\md{J}+r \le k+2}{0 \le r \le k}} \md{\nabla^i (u-v)}\rho^{i-\mu } (\md{\nabla^J u}+\md{\nabla^J v}) \rho^{\md{J}-\sharp J \cdot\mu} \md{\nabla^r TN}\rho^{r-\mu+1}  \\
 &\lesssim \sum_{\genfrac{}{}{0pt}{2}{i+\md{J}+r \le k+2}{0 \le r \le k}}\md{u-v}_{C^i_{\mu+1}} (\md{u}_{C^J_{\mu}}+\md{v}_{C^J_{\mu}}) \md{TN}_{C^r_{\mu-1}} 
\eas
Here $\sharp J$ denotes the number of entries in the multi-index $J$ and $C^J$ is the product of the norms $\prod_{s=1}^l\md{v}_{C^{i_s}}$. We used the fact that $\rho^{i+2-\mu} $ is bounded on $N$ to remove extraneous factors of $\rho$. For this it was crucial to assume $\mu < 2$. Now simply note that $\nm{TN}_{C^r_{\mu-1}} < \infty$ by the $\CS_\mu$ condition. The result now follows, since $C^k_{(k+2)(\mu-1)} \hookrightarrow C^k_{\mu-1} $ is a continuous embedding. 
\end{proof}

The deformation map $F_{\CS}$ then extends to a map between Sobolev spaces as follows:
 
\begin{prop}
\label{3_4deformation_map_cs}
Let $p > 4$ and $k \ge 1$. For sufficiently small $\ep > 0$ the map $F_{\CS}$ extends to a $C^\infty$ map of Sobolev spaces:
\[ F_\CS: \cL_{\eps} = \{ v \in L^p_{k+1, {\bar{\mu}}} (\nu_\ep (N)) : \nm{v} \le \eps \} \times \cO \longrightarrow L^{p}_{k, {\bar{\mu}-1}}(E_{\cay}).\]
Furthermore, any $v\in L^p_{k+1, \de} (\nu_\ep (N)) $ such that $F_{\CS}(v) = 0$ is smooth and lies in $C^\infty_{\bar{\mu}}$. The linearisation at $0$ is the bounded linear map: 
\e
D_{\CS} : L^p_{k+1, {\bar{\mu}}} (\nu (N)) \longrightarrow L^{p}_{k, {\bar{\mu}-1}}(E_{\cay}). \nonumber
\e
Finally, $D_{\CS}$ is Fredholm if all the rates in ${\bar{\mu}}$ are in the complement of a discrete set $\cD \subset \R$, which is determined by the metric structure on the asymptotic cones $C_i \subset \R^8$. 
\end{prop}

\begin{proof}
The proof is identical to the one for the $\AC$ case \ref{3_4deformation_map_ac}, except that one needs to check that the dependence in $f \in \cF$ respects the weighted space, i.e. that the derivatives $\partial_f ^ k \cF (v, f) $, which are a priori maps $C^\infty_{\bar{\mu}} \times (T_f \cF)^k \ra C^\infty_{loc}$ can be extended to maps $L^p_{k+1, \bar{\mu}} \times (T_f \cF)^k \ra L^p_{k, \bar{\mu}-1}$. For this, consider a smooth deformation $f(t) \in \cF$ of a manifold $N$ with a unique singular point at the origin of $\R^8$ and rate $\mu$. Up to first order this is equivalent to deforming the $\Spin(7)$-structure in $C^\infty_1$ (which also takes care of the translations, since we always compare to the model $\CS$ manifold $N_f$), and perturbing the manifold near the cone by a vector field $u \in C^\infty_1(\nu(N))$, while keeping the singular point fixed, as well as the $\Spin(7)$-structure at the origin. Let $\phi_t$ be the flow associated to $u$. We then have for $v \in C^\infty_{\mu} (\nu_\eps(N))$ and $p \in N$:
\eas
\partial_f F_{v, f(0)}[\dot{f}(0)](p) &= \frac{\d}{\d t}\vert_{t = 0} \boF(\phi_t(p), (\phi_t)_{*} v(p), (\phi_t)_{*} \nabla v(p), (\phi_t)_{*} T_pN, \Phi(t)) \\
&= \D \boF \left[ u, -\cL_u v, -\cL_u \nabla v, -\cL_u  TN, \dot{\Phi}(0)\right].  
\eas
Now we see that all the arguments in square brackets are in $O(r^{\mu-1})$, either by definition (like $u$ an $\dot{\Phi}(0)$), or as a consequence thereof. The norm of term $\cL_u \nabla v$ for instance can be bounded by $\md{u} \md{\nabla^2 v} +\md{\nabla u} \md{\nabla v}$, which is in $O(r^{\mu-1})$ by assumption. As $\D \boF$ can be bounded by a constant independent of the chosen $\CS$ manifold, we find that: 
\eas
\md{\partial_f F_{v, f(0)}} \lesssim \md{v}_{C^2_1}+ \md{TN}_{C^1_0} + 1.
\eas 
The argument we presented also applies to higher derivatives and so we see that $\partial_f F$ maps $L^p_{k+1, \bar{\mu}} \times (T_f \cF)^k \ra L^p_{k, \bar{\mu}-1}$, as required. 
\end{proof}

Using this we can now prove the following result about the local structure of the moduli space  $\cM_{\CS}^{\bar{\mu}}(N, \cS)$, just as in the $\AC$ case. 

\begin{thm}[Structure]
\label{3_4_structure_cs}
Let $N$ be an $\CS_{\bar{\mu}}$ Cayley submanifold of $(M, \Phi)$, and suppose $ \{\Phi_s\}_{s \in \cS}$ is a smooth family of small perturbations of $\Phi$. Then there is a non-linear deformation operator $F_{\CS}$ which for $\ep> 0$ sufficiently small give a $C^\infty$ map: 
\eas
F_{\CS}: \mathcal{L}_\eps = \{ v\in L^p_{k+1, {\bar{\mu}}} (\nu_\ep (N)), \nm{v}_{L^p_{k+1, \bar{\mu}}} < \ep \} \times \cO \longrightarrow L^{p}_{k, {\bar{\mu}-1}}(E).
\eas
Then a neighbourhood of $N$ in $\cM^{\bar{\mu}}_{\CS}(N, \cS)$ is homeomorphic to the zero locus of $F_{\CS}$ near $0$. Furthermore we can define the \textbf{deformation space} $\cI^{\bar{\mu}}_{\CS}(N) \subset C^\infty_{\bar{\mu}} (\nu (N))$ to be the the kernel of $D_{\CS} = \D F_{\CS}(0)$, and the \textbf{obstruction space} $\cO^{\bar{\mu}}_{\CS}(N) \subset C^\infty_{4-\bar{\mu}}(E)$ to be the cokernel of $D_{\CS}$. Then a neighbourhood of $N$ in $\cM^{\bar{\mu}}_{\CS}(N, \cS)$ is also homeomorphic to the zero locus of a Kuranishi map: 
\eas
	\kappa^{\bar{\mu}}_{\CS}:  \cI^{\bar{\mu}}_{\CS}(N) \op T_{f_0}\cO \longra \cO^{\bar{\mu}}_{\CS}(N).
\eas
In particular if $\cO^{\bar{\mu}}_{\CS}(N) = \{0\}$ is trivial, $\cM^{\bar{\mu}}_{\CS}(N, \cS) $ admits the structure of a $C^1$-manifold near $N$. We say that $N$ is \textbf{unobstructed} in this case.
\end{thm}
\begin{rem}
\label{4_3_new_op}
The operator $F_{\CS}$ allows for the points of the singular cones to move. We could also fix the points while still allowing the links of the cones to deform, giving us an operator $F_{\CS, \text{cones}}$. We can give this operator the exact same treatment and reprove all the theorems in this section. Similarly one can consider an operator $F_{\CS, \text{fix}}$, where neither the points nor the cones are allowed to deform and again all the same statements are true for this operator.
\end{rem}
We again have a formula for the index, where we define $\si(N)$ and $[N] \cdot_{[u_1], \dots,  [u_l]} [N]$ in the same way as for the $\AC$ case. 
\begin{prop}[Index]
\label{3_4_index_cs}
Let $N$ be an $\CS_{\bar{\mu}}$ Cayley submanifold of $(M, \Phi)$ with cones $C_i = \R_i \times L_i (1 \le i \le l)$, and suppose $ \{\Phi_s\}_{s \in \cS}$ is a smooth family of small perturbations of $\Phi$. Assume that $(1, \mu_i] \cap \cD(L_i) = \emptyset $. Pick homotopy classes $[u_i] \in [L_i, S\nu(N)|_{L_i}]$. Then the following holds: 
\eas
\ind D_{\CS} &= \frac{1}{2}(\si(N) + \chi(N)) - [N] \cdot_{[u_1], \dots,  [u_l]} [N]   -\sum_{i = 1}^l ( \eta(L_i)+ T([u_i])) + \dim \cF.
\eas
Here $\eta(L)$ and $T([u_i])$ are the quantities from Proposition \ref{3_4_index_ac}.
\end{prop}

\begin{rem}
\label{3_4_index_additivity}
Suppose that $N$ is a $\CS_{\bar{\mu}}$ Cayley in $(M, \Phi)$ with a unique singular point, with an unobstructed cone $C = \R_+ \times L$. We consider the deformations of $N$ for a fixed $\Spin(7)$-structure. Let $A \subset \R^8$ be an asymptotically conical Cayley of rate $\la < 1$, with the same cone. Consider an almost Cayley manifold $ \tilde {N} =  N \sharp_L A$ with deformation operator $ D_{\tilde{N}}$. Pick an arbitrary class $[u] \in [L, S\nu(N)|_{L}] \simeq [L, S\nu(A)|_{L}]$, and assume that $[\la, 1) \cap \cD = (1, \mu] \cap \cD = \emptyset$. Then we have the following, where we consider the deformation problem with fixed points and cones on the conically singular side:
\ea
 {\ind}_\mu D_{\AC} +  {\ind}_\mu D_{\CS, \text{fix}}  &= \frac{1}{2}(\si(A) + \chi(A)) - [A] \cdot_{[u]} [A] \nonumber \\ & +\frac{1}{2}(\si(N) + \chi(N)) - [N] \cdot_{[u]} [N] \nonumber \\
 &+\eta(L) + T([u])- \eta(L) - T([u]) \nonumber \\
 &=  \frac{1}{2}(\si(A) + \si(N) + \chi(A) + \chi(N))  
 -( [A] \cdot_{[u]} [A] + [N] \cdot_{[u]} [N]) \nonumber \\
 &= \frac{1}{2}(\si(\tilde {N}) + \chi(\tilde {N}))  
 - [\tilde {N}] \cdot[\tilde {N}]  \nonumber \\
 &= \ind D_{\tilde{N}}.
\ea
Note that we only proved our $\AC$ index formula for rates $< 1$. However, using the unobstructedness of the cone (meaning $d(1)= \dim \cM^{G_2}(L)$) and Theorem \ref{2_3_change_of_index} we see that $ {\ind}_\mu D_{\AC} =  {\ind}_\la D_{\AC} + \dim \cM^{G_2}(L)$. We note that by construction:
\eas
\ind D_{\CS}-\ind D_{\CS, \text{fix}} = \dim \cF - \dim \cS = \dim \cF,
\eas
as $\dim \cS = 0$ by assumption. For the third inequality, notice that the link $L$ is a compact three-manifold, and thus has Euler characteristic $\chi(L) = 0$. Thus $\chi(A) + \chi(N) = \chi(\tilde{N})$. Similarly, the signature is also additive (cf. Theorem 4.14 in \cite{atiyahSpectralAsymmetryRiemannian1975}). Finally, the intersection numbers with fixed boundary behaviour are also additive, as they can be obtained by simply counting self-intersection points. Thus the indices of the conical operators add up to the index of the glued manifold plus, which is to be expected, as perturbations of the glued manifold should correspond bijectively to perturbations in either piece.
Note that we equally well have: 
\ea
\ind D_{\tilde{N}} = {\ind}_\la D_{\AC} +  {\ind}_\mu D_{\CS, \text{cones}}.
\ea
Whereas before all perturbations of rates $\le 1$ were considered part of the asymptotically conical piece, now the perturbations of rate exactly $1$ are considered perturbations of the conically singular piece. If the cone satisfies has no critical rates between $0$ and $1$, we can even go one step further and pick some $\la' < 0$ with $[\la', 0) \cap \cD = \emptyset$. We then have:
\ea
\ind D_{\tilde{N}} = {\ind}_{\la'} D_{\AC} +  {\ind}_\mu D_{\CS}.
\ea

\end{rem}

\addcontentsline{toc}{section}{References}
\bibliographystyle{alpha}
\bibliography{library.bib}

\begin{thebibliography}{LMO85}

\bibitem[APS75]{atiyahSpectralAsymmetryRiemannian1975}
Michael~F. Atiyah, Vijay~K. Patodi, and Isadore~M. Singer.
\newblock Spectral asymmetry and {{Riemannian Geometry}}. {{I}}.
\newblock {\em Mathematical Proceedings of the Cambridge Philosophical
  Society}, 77:43--69, 1975.

\bibitem[Ber55]{bergerGroupesHolonomieHomogenes1955}
Marcel Berger.
\newblock Sur les groupes d'holonomie homog\`enes de vari\'et\'es \`a connexion
  affine et des vari\'et\'es riemanniennes.
\newblock {\em Bulletin de la Soci\'et\'e Math\'ematique de France},
  83:279--330, 1955.

\bibitem[Bon66]{bonanVarietesRiemanniennesGroupe1966}
Edmond Bonan.
\newblock Sur des vari\'et\'es riemanniennes \`a groupe d'holonomie {$G_2$} ou
  {{Spin}}(7).
\newblock {\em Comptes Rendus Hebdomadaires des S\'eances de l'Acad\'emie des
  Sciences. S\'eries A et B}, 262:A127--A129, 1966.

\bibitem[Bry87]{bryantMetricsExceptionalHolonomy1987}
Robert~L. Bryant.
\newblock Metrics with {{Exceptional Holonomy}}.
\newblock {\em Annals of Mathematics}, 126:525--576, 1987.

\bibitem[BS89]{bryantConstructionCompleteMetrics1989}
Robert~L. Bryant and Simon~M. Salamon.
\newblock On the construction of some complete metrics with exceptional
  holonomy.
\newblock {\em Duke Mathematical Journal}, 58:829--850, 1989.

\bibitem[Cla12]{clancySpinManifoldsCalibrated2012}
Robert Clancy.
\newblock {\em Spin(7)-Manifolds and Calibrated Geometry}.
\newblock PhD thesis, University of Oxford, 2012.

\bibitem[HL82]{HarvLaws}
Reese Harvey and H.~Blaine Lawson.
\newblock Calibrated geometries.
\newblock {\em Acta Mathematica}, 148:47--157, 1982.

\bibitem[Joy96]{joyceCompact8manifoldsHolonomy1996}
Dominic~D. Joyce.
\newblock Compact 8-manifolds with holonomy {{Spin}}(7).
\newblock {\em Inventiones mathematicae}, 123:507--552, 1996.

\bibitem[Joy04]{joyceReg}
Dominic~D. Joyce.
\newblock Special {{Lagrangian}} submanifolds with isolated conical
  singularities. {{I}}. {{Regularity}}.
\newblock {\em Annals of Global Analysis and Geometry}, 25:201--251, 2004.

\bibitem[Joy07]{joyceRiemannianHolonomyGroups2007}
Dominic~D. Joyce.
\newblock {\em Riemannian Holonomy Groups and Calibrated Geometry}.
\newblock Oxford {{Graduate Texts}} in {{Mathematics}}. {Oxford University
  Press}, 2007.

\bibitem[Kaw14]{kawaiDeformationsHomogeneousAssociative2014}
Kotaro Kawai.
\newblock Deformations of homogeneous associative submanifolds in nearly
  parallel {$G_2$}-manifolds.
\newblock {\em Asian Journal of Mathematics}, 21:429--462, 2014.

\bibitem[Kov05]{KovalevFibration}
Alexei Kovalev.
\newblock Coassociative {{K3}} fibrations of compact {$G_2$}-manifolds, 2005.

\bibitem[Law89]{lawlorAngleCriterion1989}
Gary~R. Lawlor.
\newblock The angle criterion.
\newblock {\em Inventiones mathematicae}, 95:437--446, 1989.

\bibitem[LM16]{lawsonSpinGeometryPMS382016}
H.~Blaine Lawson and Marie-Louise Michelsohn.
\newblock {\em Spin Geometry ({{PMS-38}})}.
\newblock {Princeton University Press}, 2016.

\bibitem[LMO85]{lockhartEllipticDifferentialOperators1985}
Robert~B. Lockhart and Robert~C. Mc~Owen.
\newblock {Elliptic differential operators on noncompact manifolds}.
\newblock {\em Annali della Scuola Normale Superiore di Pisa - Classe di
  Scienze}, 12:409--447, 1985.

\bibitem[Loc87]{lockhartFredholmHodgeLiouville1987}
Robert~B. Lockhart.
\newblock Fredholm, {{Hodge}} and {{Liouville Theorems}} on {{Noncompact
  Manifolds}}.
\newblock {\em Transactions of the American Mathematical Society}, 301:1--35,
  1987.

\bibitem[Lot09]{lotayStabilityCoassociativeConical2009}
Jason~D. Lotay.
\newblock Stability of {{Coassociative Conical Singularities}}.
\newblock {\em Communications in Analysis and Geometry}, 20, 2009.

\bibitem[McL98]{mcleanDeformationsCalibratedSubmanifolds1998}
Robert~C. McLean.
\newblock Deformations of calibrated submanifolds.
\newblock {\em Communications in Analysis and Geometry}, 4:705--747, 1998.

\bibitem[Mil97]{milnorTopologyDifferentiableViewpoint1997}
John~W. Milnor.
\newblock {\em Topology from the {{Differentiable Viewpoint}}}.
\newblock {Princeton University Press}, 1997.

\bibitem[Moor17]{mooreDeformationTheoryCayley2017}
Kim Moore.
\newblock {\em Deformation Theory of {{Cayley}} Submanifolds}.
\newblock PhD thesis, University of Cambridge, 2017.

\bibitem[Moor19]{mooreDeformationsConicallySingular2019}
Kim Moore.
\newblock Deformations of {{Conically Singular Cayley Submanifolds}}.
\newblock {\em The Journal of Geometric Analysis}, 29:2147--2216, 2019.

\bibitem[Mor96]{morganSeibergWittenEquationsApplications1996}
John~W. Morgan.
\newblock {\em The {{Seiberg-Witten}} Equations and Applications to the
  Topology of Smooth Four-Manifolds}.
\newblock {Princeton University Press}, 1996.

\bibitem[Ohs16]{ohstDeformationsCayleySubmanifolds2016}
Matthias Ohst.
\newblock {\em Deformations of {{Cayley}} Submanifolds}.
\newblock PhD thesis, University of Cambridge, 2016.

\bibitem[RS95]{robbinSpectralFlowMaslov1995}
Joel Robbin and Dietmar~A. Salamon.
\newblock The {{Spectral Flow}} and the {{Maslov Index}}.
\newblock {\em Bulletin of the London Mathematical Society}, 27:1--33, 1995.

\bibitem[Var01]{varadarajanSpinSubgroupsSpin2001}
Veeravalli~S. Varadarajan.
\newblock Spin(7)-subgroups of {{SO}}(8) and {{Spin}}(8).
\newblock {\em Expositiones Mathematicae}, 19:163--177, 2001.

\end{thebibliography}

\medskip

\noindent{\small\sc The Mathematical Institute, Radcliffe
Observatory Quarter, Woodstock Road, Oxford, OX2 6GG, U.K.

\noindent E-mail: {\tt gilles.englebert@maths.ox.ac.uk.}}

\end{document}